\documentclass[12 pt]{amsart}
\usepackage[dvipsnames]{xcolor}
\usepackage{amsmath,amsthm,amssymb,shuffle,stmaryrd,mathtools,pdftexcmds}
\usepackage{subcaption}
\DeclareCaptionSubType * [Alph]{table}

\usepackage{ytableau}
\newcommand{\blank}{\phantom{2}}
\newcommand{\OG}{\mathrm{OG}}

\usepackage{enumitem}
\usepackage[hmargin=1 in, vmargin = 1 in]{geometry}
\usepackage{hyperref}
\usepackage{tikz-cd}
\tikzcdset{every label/.append style = {font = \footnotesize}}

\usepackage{algorithm}
\usepackage{algorithmic}

\definecolor{darkblue}{rgb}{0.0,0,0.7}
\newcommand{\darkblue}{\color{darkblue}}
\newcommand{\newword}[1]{\textcolor{darkblue}{\textbf{\emph{#1}}}}

\newcommand{\Z}{\mathbb{Z}}

\newcommand{\C}{\mathbb{C}}

\newcommand{\algcomment}[1]{\emph{ \scriptsize \darkblue \textbackslash \textbackslash #1}}

\newcommand{\Grass}[2]{\mathrm{Gr}_{#1}(\C^{#2})}

\newcommand{\LG}{\mathrm{LG}}

\newcommand{\Plactic}{{\mathbf{P}}}
\newcommand{\sPlactic}{\mathbf{S}}

\newcommand{\SSYT}{\mathrm{SSYT}}
\newcommand{\SYT}{\mathrm{SYT}}
\newcommand{\shSSYT}{\mathrm{shTab}}
\newcommand{\shSYT}{\mathrm{shSYT}}

\newcommand{\monomial}{\mathfrak{M}}

\newcommand{\PlacticSchur}{\mathcal{S}}
\newcommand{\FreeSchur}{\widehat{\mathcal{S}}}
\newcommand{\sPlacticSchur}{\mathcal{P}}
\newcommand{\sFreeSchur}{\widehat{\mathcal{P}}}

\newcommand{\symmetricgroup}{\mathfrak{S}}

\newcommand{\twiddle}{\mathord \sim}

\newcommand{\reading}{\mathsf{rw}}

\newcommand{\cA}{\mathcal{A}}
\newcommand{\cB}{\mathcal{B}}
\newcommand{\cN}{\mathcal{N}}
\newcommand{\cD}{\mathcal{D}}
\newcommand{\cNprime}{\mathcal{N'}}
\newcommand{\seq}{\mathsf{seq}}

\newcommand{\mix}{\mathrm{mix}}
\newcommand{\SW}{\mathrm{SW}}
\newcommand{\SK}{\mathrm{SK}}

\DeclareFontFamily{U}{cbgreek}{}
\DeclareFontShape{U}{cbgreek}{m}{n}{
        <-6>    grmn0500
        <6-7>   grmn0600
        <7-8>   grmn0700
        <8-9>   grmn0800
        <9-10>  grmn0900
        <10-12> grmn1000
        <12-17> grmn1200
        <17->   grmn1728
      }{}
\DeclareFontShape{U}{cbgreek}{bx}{n}{
        <-6>    grxn0500
        <6-7>   grxn0600
        <7-8>   grxn0700
        <8-9>   grxn0800
        <9-10>  grxn0900
        <10-12> grxn1000
        <12-17> grxn1200
        <17->   grxn1728
      }{}

\DeclareRobustCommand{\qoppa}{%
  \text{\usefont{U}{cbgreek}{\normalorbold}{n}\symbol{19}}%
}

\makeatletter
\newcommand{\normalorbold}{%
  \ifnum\pdf@strcmp{\math@version}{bold}=\z@ bx\else m\fi
}
\makeatother

\newcommand{\npowerser}[1]{ \!
  \left\langle\mkern-3mu\left\langle#1\right\rangle\mkern-3mu\right\rangle \!}

\usepackage{CJKutf8}
\newcommand{\zh}[1]{\begin{CJK}{UTF8}{gbsn}#1\end{CJK}}
\newcommand{\yama}{\text{\zh{山}}}

\newcommand{\ct}{\mathsf{ct}}

\newcommand{\bbb}{\mathsf{\mathbf{b}}}

\newtheorem{thm}{Theorem}[section]
\newtheorem{theorem}[thm]{Theorem}

\newtheorem{lemma}[thm]{Lemma}
\newtheorem{corollary}[thm]{Corollary}
\newtheorem{proposition}[thm]{Proposition}

\DeclareMathOperator{\rectify}{rect}

\theoremstyle{remark}
\newtheorem{definition}[thm]{Definition}
\newtheorem{remark}[thm]{Remark}

\newtheorem{examplex}[thm]{Example}
\newenvironment{example}
  {\pushQED{\qed}\examplex}
  {\popQED\endexamplex}

\numberwithin{equation}{section}

\subjclass[2020]{05E05, 05E10, 14M15, 14N15}

\AtBeginDocument{%
   \def\MR#1{}
}

\begin{document}

\title{Constructed tableaux and a new shifted Littlewood--Richardson rule}

\author{Santiago Estupi\~n\'an-Salamanca}
\address[SES]{Department of Combinatorics \& Optimization, University of Waterloo, Waterloo ON, Canada}
\email{sestupinan@uwaterloo.ca}
\author{Oliver Pechenik}
\address[OP]{Department of Combinatorics \& Optimization, University of Waterloo, Waterloo ON, Canada}
\email{opecheni@uwaterloo.ca}

\date{\today}

\keywords{shifted plactic monoid, Schur $P$-function, projective representation theory, isotropic Grassmannian}

\begin{abstract}
We give a new Littlewood--Richardson rule for the Schubert structure coefficients of isotropic Grassmannians, equivalently for the multiplication of $P$-Schur functions. Serrano (2010) previously gave a formula in terms of classes in his \emph{shifted plactic monoid}. However, this formula is challenging to use because of the difficulty of characterizing shifted plactic classes. We give the first algebraic proof of this formula. We then use it to obtain a new rule that is easy to implement. Our rule is based on identifying a subtle analogue of \emph{Yamanouchi tableaux}, which we characterize. We show that for some families of structure coefficients, our rule leads to an algorithm with exponentially better time complexity than the original rule of Stembridge (1989).
\end{abstract}

\maketitle
\setcounter{tocdepth}{1}
\tableofcontents

\section{Introduction}

The \emph{Grassmannian} $\Grass{k}{n}$ is the parameter space for $k$-dimensional linear subspaces of the vector space $\C^n$. Its integral cohomology has a basis of \emph{Schubert classes} $\sigma_\lambda$, where $\lambda$ ranges over integer partitions with at most $k$ parts and all parts at most $n-k$.
A classical Littlewood--Richardson rule for computing the structure coefficients $c_{\lambda,\mu}^\nu \in \Z_{\geq 0}$ appearing in the formula
\[
\sigma_\lambda \cdot \sigma_\mu  = \sum_{\nu} c_{\lambda,\mu}^\nu \, \sigma_\nu
\]
can be understood to identify these structure coefficients with counts of \emph{semistandard Young tableaux} as follows. First, one starts with the \emph{plactic monoid} $\Plactic$ \cite{Lascoux.Schutzenberger:plaxique}, which is a quotient of the free monoid on $\{1 < 2 < \cdots \}$ by certain homogeneous degree $3$ equivalences (\emph{Knuth relations} \cite{Knuth}); then one picks a family of plactic equivalence classes that is closed under taking right factors, and combinatorially characterizes the elements of these classes as well as their left factors; such a characterization of the left factors then yields a combinatorially defined set whose cardinality is $c_{\lambda,\mu}^\nu$. The traditional choice here is to take the plactic classes of \emph{highest-weight tableaux}, for which the elements of the plactic classes are characterized by a \emph{Yamanouchi} (a.k.a.~\emph{ballot} or \emph{lattice}) condition. The appearance of the plactic monoid $\Plactic$ in this story can be understood from the fact that $\Plactic$ is the largest quotient of the free monoid whose monoid algebra contains a copy of $H^\star(\Grass{k}{n})$ (and satisfying a few other simple properties) \cite{Lascoux.Schutzenberger:plaxique,EstupinanSalamanca.Pechenik}.

For the \emph{isotropic Grassmannians} $\OG(n,2n+1) \subset \Grass{n}{2n+1}$ (resp.\ $\LG(n,2n) \subset \Grass{n}{2n}$), parameterizing those $n$-planes in $\C^{2n+1}$ (resp.\ $\C^{2n}$) that are isotropic with respect to a nondegenerate symmetric (resp.\ skew-symmetric) bilinear form, one again has an integral cohomology basis of Schubert classes $\pi_\lambda$ (resp.\ $\qoppa_\lambda$), where $\lambda$ now ranges over \emph{strict} integer partitions with at most $n$ parts and all parts at most $n$. The analogues of the Littlewood--Richardson coefficients for $\OG(n,2n+1)$ are the structure constants $b_{\lambda, \mu}^\nu \in \Z_{\geq 0}$ defined by 
\[
\pi_\lambda \cdot \pi_\mu = \sum_{\nu} b_{\lambda,\mu}^\nu \pi_\nu,
\]
and we are interested in combinatorial formulas for these integers. One might expect an analogous problem for $\LG(n,2n)$; however, we have
\[
\qoppa_\lambda \cdot \qoppa_\mu = \sum_{\nu} 2^{\ell(\lambda) + \ell(\mu) - \ell(\nu)} \, b_{\lambda,\mu}^\nu  \qoppa_\nu,
\]
where $\ell(\xi)$ denotes the length of the partition $\xi$, so in fact the problem for $\LG(n,2n)$ is essentially equivalent to the $\OG(n,2n+1)$ problem.

The first combinatorial rule for $b_{\lambda, \mu}^\nu$ was given by J.~Stembridge \cite{Stembridge} in terms of \emph{shifted semistandard tableaux}, with the relation to isotropic Grassmannians discovered slightly later by P.~Pragacz \cite{Pragacz}. Our main result is a new combinatorial formula for $b_{\lambda, \mu}^\nu$, which we obtain by a strategy analogous to that that above for $c_{\lambda,\mu}^\nu$. The key tool is the \emph{shifted plactic monoid} $\sPlactic$ of L.~Serrano \cite{Serrano}, which is the largest quotient of the free monoid whose monoid algebra contains a copy of $H^\star(\OG(n,2n+1))$ (and satisfying a few other simple properties) \cite{EstupinanSalamanca.Pechenik}. 
We provide the first algebraic proof of Serrano's \cite{Serrano} shifted Littlewood--Richardson rule in terms of shifted plactic classes. 
While the rule of Stembridge \cite{Stembridge} can be understood in terms of \emph{shifted highest-weight tableaux}, the shifted plactic class of such tableaux turns out to be frustratingly complicated to describe combinatorially. We identify a different and less obvious shifted plactic class whose elements we characterize by a \emph{barely Yamanouchi} condition, which is a proper strengthening of the classical Yamanouchi condition on words and which is easy to describe. By further characterizing left factors of such words, we obtain our new tableau formula for $b_{\lambda, \mu}^\nu$: 

\begin{theorem}\label{thm:main}
    If $\lambda,\mu,$ and $\nu$ are strict partitions,
    then $b_{\lambda, \mu}^{\nu}$ equals the number of tableaux of shape $\lambda$ constructed from $\mu<\nu.$
\end{theorem}

Theorem~\ref{thm:main} was first announced in the FPSAC extended abstract \cite{EstupinanSalamanca.Pechenik:FPSAC}. We briefly defer the definition of ``constructed'' until Section~\ref{sec:constructed_intro}; the reader with some background in classical tableau combinatorics should be able to use Theorem~\ref{thm:main} to calculate structure coefficients after finishing this introduction.
For an explicit and precise algorithmic implementation of Theorem~\ref{thm:main} and further examples of its use, see Section~\ref{sec:examples}. A Python implementation is available at \cite{shifted-code}.

While our new rule and the older rule of Stembridge \cite{Stembridge} both involve shifted semistandard tableaux, the nature of the rules is extremely different and it does not seem straightforward to translate between them. We find the new rule more convenient for calculations by hand, as we find that using Stembridge's rule often involves drawing many tableaux that are then determined not to contribute to $b_{\lambda, \mu}^\nu$. We also show that, for some special classes of coefficients $b_{\lambda, \mu}^\nu$, our rule leads to an algorithm with exponentially better time complexity. For further shifted Littlewood--Richardson rules and proofs, see \cite{Assaf, Cho, Choi.Kwon, Gillespie.Levinson.Purbhoo,Grancharov.Jung.Kang.Kashiwara.Kim, Huang.Chu.Li, Nguyen, Serrano, Shigechi, Shimozono, Stembridge}; although we do not here carry out a detailed comparison of our rule with these, we remark that our formula appears to be meaningfully different from any of the previously known rules.

\begin{example}
    Let $\lambda = (5,3,1)$, $\mu = (5,4)$, and $\nu = (6,5,4,2,1)$. Then the structure coefficient $b_{\lambda,\mu}^\nu$ equals $2$, as witnessed by the $2$ constructed tableaux
    \[
      \begin{ytableau}
        {\color{orange}{1}}    & {\color{red}{1}} & 1 & 4' & {\color{blue}{6}}     \\
        \none&{\color{red}{2}}& 2 &{\color{brown}{5'}}   \\
        \none&\none    & 3    
    \end{ytableau}
    \quad
    \text{and}
    \quad
    \begin{ytableau}
        {\color{orange}{1}}    & {\color{red}{1}} & 1 & 4' & {\color{blue}{6}}     \\
        \none&{\color{red}{2}}& 2 &{\color{brown}{5}}   \\
        \none&\none    & 3    
    \end{ytableau}
    \]
    per Theorem~\ref{thm:main}. (Here, the coloring is redundant, but records steps of the construction process.) 

    In contrast, the Stembridge rule \cite{Stembridge} realizes $b_{\lambda,\mu}^\nu = 2$ from the skew tableaux
    \[
    \ytableaushort{\blank \blank \blank \blank \blank 1, \none \blank \blank \blank {1'} {2'},\none \none \blank {1'}1 {2'}, \none \none \none 1 {2'}, \none \none \none \none 2}
*[*(Gray)]{5,1+3,2+1} 
\quad
\text{and}
\quad
    \ytableaushort{\blank \blank \blank \blank \blank 1, \none \blank \blank \blank {1'} {2'},\none \none \blank {1'}12, \none \none \none 1 {2'}, \none \none \none \none 2}
*[*(Gray)]{5,1+3,2+1} 
    \]
    For the reader who is familiar with the Stembridge rule, we suggest the exercise of checking that there are no other skew tableaux satisfying the required conditions. 
\end{example}

In addition to this computational efficiency, a further motivation for our development of Theorem~\ref{thm:main} is the desire for analogues in richer cohomology theories. In torus-equivariant cohomology $H_T(\OG(n, 2n+1))$, there is no proved combinatorial rule for computing the analogues of $b_{\lambda,\mu}^\nu$ and indeed only a conjectural rule in a case of quotienting away almost all of the $T$-weights \cite{Robichaux.Yadav.Yong}. (Here, at least it is known that the $H_T(\OG(n, 2n+1))$ and $H_T(\LG(n, 2n))$ problems are essentially equivalent; see, e.g., \cite[Theorem~8.2]{Robichaux.Yadav.Yong}.) The partial conjecture of \cite{Robichaux.Yadav.Yong} is based on combining the non-equivariant rule of Stembridge \cite{Stembridge} with the ``edge-labeled'' combinatorics of \cite{Thomas.Yong:HT}. One might hope that the new rule of Theorem~\ref{thm:main} might extend to $H_T$ more easily. 

Similarly, the $K$-theoretic Schubert calculus of isotropic Grassmannians is only partly understood. There are combinatorial formulas for the analogues of $b_{\lambda,\mu}^\nu$ in $K^0(\OG(n, 2n+1))$ \cite{Clifford.Thomas.Yong, Buch.Samuel, Pechenik.Yong:genomic}. However, in the $K$-theoretic setting, the corresponding problem for $K^0(\LG(n, 2n))$ is \emph{not} equivalent to the $\OG(n, 2n+1)$ problem. Indeed, there is no known combinatorial formula for the $K$-theoretic Schubert structure constants of $K^0(\LG(n, 2n))$ (even conjecturally). (However, see \cite{Buch.Ravikumar} for a proved Pieri formula in this context and \cite{Pechenik.Yong:genomic} for conjectural bounds.) We hope that Theorem~\ref{thm:main} can provide a new avenue of attack on this problem. Further discussion of the equivariant and $K$-theoretic settings may appear elsewhere.

\subsection{Constructed tableaux}
\label{sec:constructed_intro}

Here we give the definitions of constructed tableaux needed to understand and apply Theorem~\ref{thm:main}. We momentarily assume the reader is familiar with the basic definitions of shifted shapes and shifted tableaux; if not, see Section~\ref{sec:background} as needed.

\begin{definition}
A skew shape is a \newword{horizontal strip} if it contains at most one box in each column, and it is a
\newword{vertical strip} if it has at most one box in each row. A set of boxes $\gamma$ is a \newword{generalized rimhook} it can be partitioned into a vertical strip $\xi/\pi$ and a horizontal strip $\theta / \eta$ with $\xi \subseteq \eta$. We write $\gamma = \xi/\pi \olessthan \theta/\eta$. If we can choose the partitions so that $\eta = \xi$, we say $\gamma$ is a \newword{rimhook}.

 Let $T$ be a tableau of generalized rimhook shape $\gamma = \xi /\pi \olessthan \theta/\eta$. We say that $T$ is a \newword{Serrano--Pieri strip} if $\xi/\pi$ is filled with unprimed letters that increase from top to bottom, $\theta/\eta$ is filled with primed letters that increase from left to right, and each label in $\xi/\pi$ is less than every label in $\theta/\eta$.
(Note that this is backwards from the ordinary Pieri fillings the reader may be familiar with in other contexts!)
\end{definition}

\begin{example}
In the tableau
    \ytableausetup{boxsize=normal}
    \[
\ytableaushort{1 {2'} {3'} 4 {*(SkyBlue) 5} {7'} {*(SkyBlue) 9'}, \none 5 {6'} {*(SkyBlue) 6} {*(SkyBlue) 8'}, \none \none {*(SkyBlue) 7}},
    \]
    the boxes shaded in blue form a Serrano--Pieri strip. Note that the blue boxes do not form a skew shape, so that it is not a rimhook, but only a generalized rimhook.
\end{example}

\begin{remark}
    Observe that the horizontal or vertical strip in a Serrano--Pieri strip is allowed to be empty. That is, a horizontal strip is a Serrano--Pieri strip, and likewise a vertical strip. 
\end{remark}

\begin{definition}\label{def:constructibleDef}
    Let $\alpha = (\alpha_1, \dots, \alpha_\ell), \beta = (\beta_1, \dots, \beta_\ell)$ be sequences of nonnegative integers with $\alpha_i < \beta_i$ for all $i$. 
    A tableau $T$  of shape $\lambda$ is \newword{constructible} from $\alpha < \beta$ if 
    
        \begin{itemize}
            \item for every $j$, the letters of $(\alpha_j, \beta_j]$ in $T$ form a Serrano--Pieri strip; and
            \vspace{0.5 cm}
            \item the unprimed entries of $(\alpha_j, \beta_j]$ occur before the unprimed entries of $(\alpha_i, \beta_i]$ for all $i<j$ when on the same row; and
            \vspace{0.5 cm}
            \item the primed entries of $(\alpha_j, \beta_j]$ occur before the primed entries of $(\alpha_i, \beta_i]$ for all $i<j$ when on the same column.
    \end{itemize}
\end{definition}

\begin{definition}
\label{def:extend}
Let $T$ be constructible from $\alpha < \beta$ and let $\gamma=(\xi/\pi) \olessthan (\theta/\eta)$ be the Serrano--Pieri strip for the interval $(\alpha_j, \beta_j]$. We say $\gamma$ can be \newword{extended} if there is a Serrano--Pieri strip $\gamma'=(\xi'/\pi') \olessthan (\theta'/\eta')$ such that $\gamma' \supsetneq \gamma$ and the letters in $\gamma'$ (including the primed ones but ignoring their primes) form an interval $(k,e]\supsetneq (\alpha_j,\beta_j]$.
\end{definition}

\begin{definition}\label{def:constructedDef}
    We say $T$ is \newword{constructed} from $\alpha<\beta$, if it is constructible from $\alpha<\beta$ and none of the Serrano--Pieri strips $\gamma=(\xi/\pi) \olessthan (\theta/\eta)$ in the decomposition can be extended. 
\end{definition}

\begin{example}
    The tableau on the left is constructible but not constructed, since the sequence encoded in blue can be extended using elements from the sequence encoded in red. 
           
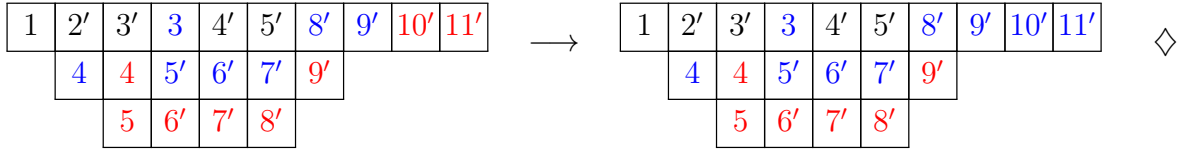
\[ \begin{ytableau}
        1   &2' & 3' &{\color{blue}{3}} & 4'              & 5' &{\color{blue}{8'}}&{\color{blue}{9'}} &{\color{red}{10'}}&{\color{red}{11'}}\\
        \none&{\color{blue}{4}}&{\color{red}{4}}&{\color{blue}{5'}}&{\color{blue}{6'}}&{\color{blue}{7'}}&   {\color{red}{9'}}      \\
        \none&\none            & {\color{red}{5}}                &{\color{red}{6'}}&   {\color{red}{7'}}            &  {\color{red}{8'}}  
    \end{ytableau} \quad \longrightarrow \quad \begin{ytableau}
        1   &2' & 3' &{\color{blue}{3}} & 4'              & 5' &{\color{blue}{8'}}&{\color{blue}{9'}} &{\color{blue}{10'}}&{\color{blue}{11'}}\\
        \none&{\color{blue}{4}}&{\color{red}{4}}&{\color{blue}{5'}}&{\color{blue}{6'}}&{\color{blue}{7'}}&   {\color{red}{9'}}      \\
        \none&\none            & {\color{red}{5}}                &{\color{red}{6'}}&   {\color{red}{7'}}            &  {\color{red}{8'}}  
    \end{ytableau} \qedhere \]
\end{example}

The intuitive idea behind the notion of a tableau $T$ of shape $\lambda$ constructed from $\alpha < \beta$ is as follows. First, one places the elements in the segment $(\mu_{\ell(\nu)}, \nu_{\ell(\nu)}]$ so that they form a vertical strip and a horizontal strip consisting of unprimed and primed elements, respectively, and in such a manner that the entries of the vertical strip are left of those in the horizontal strip.  
Then, a similar process is followed for the subsequent segments, in such a manner that they are placed below or to the right of the previous segments, and so that no later segment extends an earlier segment. As the unprimed entries of the next segment are positioned in various rows and columns, they may be placed before some primed entries of previous segments in that row, displacing those entries to the right, or be placed above primed entries of previous segments in that column, displacing those entries downwards. 
However, when placing the primed entries of the next segment, they may not displace unprimed entries of previous segments. 

\begin{example}
   Consider the strict partitions $\nu= (11,9,5)$, $\mu=(4,2)$, and $\lambda=(8,7,4)$. The following is an illustration of how to sequentially compute a tableau constructible from $\mu<\nu$ of shape $\lambda$:

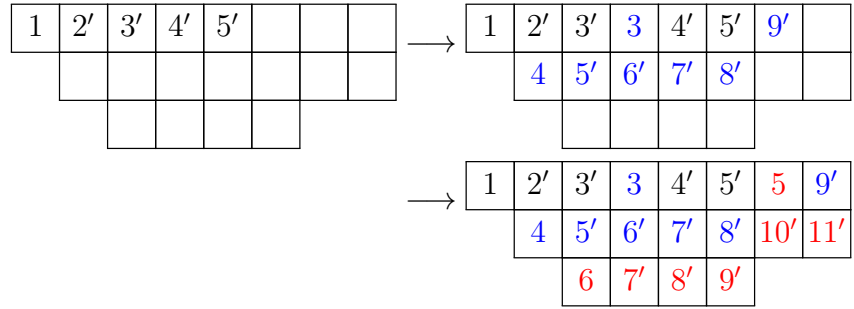
   \begin{align*}
       \begin{ytableau}
        1   &2' & 3' & 4'              & 5' & & &\\
        \none& & & & & &   &    \\
        \none&\none            &                 & &             &   
    \end{ytableau} &\longrightarrow \begin{ytableau}
        1   &2' & 3' & {\color{blue}{3}} & 4'              & 5' & {\color{blue}{9'}} &\\
        \none& {\color{blue}{4}} & {\color{blue}{5'}} & {\color{blue}{6'}} & {\color{blue}{7'}} & {\color{blue}{8'}} &   &    \\
        \none&\none            &                 & &             &   
    \end{ytableau} \\
    &\longrightarrow \begin{ytableau}
                1   &2' & 3' &{\color{blue}{3}} & 4'              & 5' &{\color{red}{5}}&{\color{blue}{9'}}\\
                \none&{\color{blue}{4}}& {\color{blue}{5'}}&{\color{blue}{6'}}&{\color{blue}{7'}}& {\color{blue}{8'}} &   {\color{red}{10'}}  &{\color{red}{11'}}     \\
                \none&\none            & {\color{red}{6}}                &{\color{red}{7'}}&   {\color{red}{8'}}            &  {\color{red}{9'}}  
            \end{ytableau}
   \end{align*}
    The resulting tableau is also constructed as none of the sequences can be extended. 
\end{example}

\subsection{Organization} The remainder of this paper is organized as follows. Section~\ref{sec:background} recalls necessary background on tableaux and symmetric functions. Section~\ref{sec:LR} gives the first algebraic proof of Serrano's shifted Littlewood--Richardson rule \cite{Serrano} in terms of shifted plactic classes, as well as of a rule equivalent to that of Stembridge \cite{Stembridge}. 
In Section~\ref{sec:barelyYamanouchi}, we introduce the family of \emph{barely Yamanouchi tableaux}, which will serve as our shifted analogues of ordinary Yamanouchi tableaux, and provide a combinatorial characterization. Section~\ref{sec:shifted_lattice} introduces a \emph{shifted lattice} condition on words and establishes some basic properties; we later show that barely Yamanouchi words (i.e., words that mixed insert to barely Yamanouchi tableaux) satisfy this shifted lattice condition. Section~\ref{sec:interlacingTableaux} introduces \emph{interlacing words} to characterize the left factors of barely Yamanouchi words and establishes many technical properties of the \emph{interlacing tableaux} that arise as mixed insertions of interlacing words. In Section~\ref{sec:construct}, we establish some basic combinatorial properties of constructed tableaux. Section~\ref{sec:insertions} studies the properties of tableaux obtained by mixed insertion of certain words into constructed tableaux. In Section~\ref{sec:interlacing=constructed}, we prove Theorem~\ref{thm:main} by establishing that constructed tableaux, interlacing tableaux, and mixed insertions of left factors of barely Yamanouchi words are all the same class of objects. Section~\ref{sec:examples} provides examples of the application of Theorem~\ref{thm:main} together with an algorithmic implementation of that rule, which can be used efficiently by hand.
In Section~\ref{sec:Pieri}, we use Theorem~\ref{thm:main}
to give a new proof of the Hiller-Boe shifted Pieri formula \cite{Hiller.Boe}. 
In Section~\ref{sec:complexity}, we discuss the time complexity of Theorem~\ref{thm:main} in comparison to that of Stembridge \cite{Stembridge}, and show that our rule is exponentially faster for certain families of structure coefficients.

\section{Background}\label{sec:background}

\subsection{Combinatorics for ordinary Grassmannians}
A \newword{partition} is a finite nonincreasing sequence of positive integers $\lambda = (\lambda_1, \lambda_2, \dots, \lambda_\ell)$ with $\lambda_1 \geq \lambda_2 \geq \dots \geq \lambda_\ell > 0$. We call $\ell = \ell(\lambda)$ the \newword{length} of the partition $\lambda$. We draw a partition $\lambda$ as its \newword{Young diagram}, an array of $\lambda_i$ left-justified boxes in each row $i$. Here we take the convention that row $1$ is the top row. For example, the Young diagram of $\theta = (3,1)$ is $\ytableausetup{smalltableaux} \ydiagram{3,1}$.

A \newword{semistandard (Young) tableau} of \newword{shape} $\lambda$ is a filling of the Young diagram of $\lambda$ with positive integers so that each row is nondecreasing when read from left to right and each column is strictly increasing when read from top to bottom. For example, $\ytableaushort{114,2}$ is a semistandard tableau of shape $\theta$.
The \newword{content} of a tableau $T$ is the integer vector $(c_1, c_2, \dots)$ where $c_i$ records the number of boxes of $T$ filled with the value $i$. We say a semistandard tableau is \newword{standard} if there is some $k$ such that $c_i = \begin{cases}
    1, & \text{if $i \leq k$};\\
    0, & \text{if $i > k$};
\end{cases}$
for all $i$. We write $\SSYT(\lambda)$ for the set of all semistandard tableaux of shape $\lambda$ and $\SYT(\lambda)$ for the subset consisting of standard tableaux. The
\newword{highest-weight tableau} (or \newword{Yamanouchi tableau}) $Y_\lambda \in \SSYT(\lambda)$ is the unique tableau of its shape such that all entries in each row $i$ are equal to $i$. For example, 
\[
Y_{(4,2,1)} = \ytableaushort{1111,22,3}.
\]
\ytableausetup{nosmalltableaux}

We write $\cN \coloneqq \mathbb{Z}_{>0}$ for the alphabet of positive integers. Then $\cN^*$ denotes the set of all finite words in the alphabet $\cN$. A \newword{biword} is an ordered pair $(u,v) \in (\cN \times \cN)^*$ of words of the same length. We also equivalently think of a biword as a word in the alphabet $\cN \times \cN$ of ordered pairs of positive integers. It is convenient to write an ordered pair $(i,j)$ as $\binom{i}{j}$ and we will often do so without comment. We write $[n]$ as shorthand for the set $\{1, 2, \dots, n\}$.

We identify $\cN^*$ with a special class of biwords as follows. Let $w = w_1w_2\ldots w_n \in \cN^*$ and $C\subseteq \cN$ with $|C|=n$ and $c_1<c_2<\dots <c_n \in C$.  Then $\binom{C}{w}$ is the biword 
    \[
    \binom{C}{w} \coloneqq
    \begin{pmatrix}
        c_1 & c_2 & c_3 & \dots & c_n \\
        w_1 & w_2 & w_3 & \dots & w_n
    \end{pmatrix}
     = \binom{c_1}{w_1}\cdot \binom{c_2}{w_2} \cdot \cdots \cdot \binom{c_n}{w_n}.
     \]
     We often identify $w$ with $\binom{[n]}{w}$.

We order pairs $\binom{i_1}{j_1}, \binom{i_2}{j_2}$ \newword{lexicographically} with respect to the top coordinate by  $\binom{i_1}{j_1} \leq \binom{i_2}{j_2}$ if and only if $i_1 < i_2$ or both $i_1 = i_2$ and $j_1 \leq j_2$. We define lexicographic order with respect to the bottom coordinate similarly.
Given a biword $\binom{u}{v}$, let $\overrightarrow{\binom{u}{v}}$ be the unique biword given by sorting the ordered pairs lexicographically with respect to the top coordinate, so that they increase from left to right.
Similarly, we write $\underrightarrow{\binom{u}{v}}$ for the biword given by sorting lexicographically with respect to the bottom coordinate. If two biwords $b, b' \in (\cN \times \cN)^*$ satisfy $\underrightarrow{b} = \underrightarrow{b'}$, we write $b \approx b'$.

\begin{example}\label{ex:biword_ordering}
    Consider the word $w=31542$ and the biword 
    \[
    \binom{[5]}{w}=
    \begin{pmatrix}
        1 & 2 & 3 & 4 & 5 \\
        3 & 1 & 5 & 4 & 2
    \end{pmatrix}.
    \]
    Then 
    \[
    \overrightarrow{\binom{[5]}{w}} = \binom{[5]}{w},
    \]
    while 
    \[
    \underrightarrow{\binom{[5]}{w}} =   \begin{pmatrix}
        2 & 5 & 1 & 4 & 3 \\
        1 & 2 & 3 & 4 & 5
    \end{pmatrix}.
    \]
    Note that the top line of this latter biword is $25143$, which equals $31542^{-1}$ when thought of as the one-line notation of a permutation.
\end{example}

We now recall properties of the \newword{RSK insertion} algorithm \cite{Robinson, Schensted, Knuth}, which constructs an ordered pair of semistandard tableaux of the same shape associated to any biword. For a modern treatment of this algorithm, see, e.g., \cite{Fulton:YT}. Let $b = \binom{u}{v}$ be a biword with $u,v \in \cN^*$. Then we row insert the word $v$ from left to right producing the \newword{RSK insertion tableau} $P(b)$; at the insertion of $v_i$ we create a new box $\bbb$, and we record $u_i$ in the corresponding box $\bbb$ of the \newword{RSK recording tableau} $Q(b)$. Note that $Q(b)$ is a standard tableau if and only if $u \in \symmetricgroup_n$ is the one-line notation of a permutation.
The \newword{plactic monoid} is the quotient $\Plactic$ of the free monoid on $\cN$, where we declare two words $w, w'$ are equivalent if $P\binom{[n]}{w} = P\binom{[n]}{w'}$. Knuth \cite{Knuth} showed that $\Plactic$ can be defined using only degree $3$ equivalences; for an alternative characterization, see \cite{Lascoux.Schutzenberger:plaxique, EstupinanSalamanca.Pechenik}. A word $w$ is \newword{Yamanouchi} if $P\binom{[n]}{w} = Y_\lambda$ for some $\lambda$; note that the set of Yamanouchi words is a union of plactic equivalence classes. The Yamanouchi words are exactly those $w$ such that in each final segment of $w$ the number of $i$s is at least the number of $(i+1)$s (for all $i > 0$).

For $T \in \SYT(\lambda)$, the \newword{free Schur function} $\FreeSchur_T$ is the noncommutative power series 
\[
\FreeSchur_T \coloneqq  \sum_{Q\binom{[n]}{w} = T} x^w \in \Z\npowerser{x_1, x_2, \dots},
\]
where $x^w \coloneqq x_{w_1} x_{w_2} \dots x_{w_n}$. The \newword{plactic Schur function} $\PlacticSchur_\lambda$ is the image of $\FreeSchur_T$ under declaring $x^w = x^{w'}$ when $w,w'$ are equivalent in the plactic monoid. (Although $\PlacticSchur_\lambda$ appears to depend on a choice of $T \in \SYT(\lambda)$, it does not.) The \newword{Schur function} $s_\lambda$ is the image of $\PlacticSchur_\lambda$ (equivalently, of $\FreeSchur_T$) in the ring of power series of commuting variables.

\subsection{Combinatorics for isotropic Grassmannians}\label{sec:isotropic}

A \newword{partition} $\lambda = (\lambda_1, \lambda_2, \dots, \lambda_\ell)$ with $\lambda_1 \geq \lambda_2 \geq \dots \geq \lambda_\ell > 0$ is \newword{strict} if its parts are all distinct. A strict partition $\lambda$ has a \newword{shifted Young diagram} obtained by indenting the $i$th row of the ordinary Young diagram for $\lambda$ by $i-1$ positions. The \newword{(main) diagonal} consists of the leftmost box in each row.
For example, the shifted Young diagram of $\theta = (3,1)$ is $\ytableausetup{smalltableaux} \ydiagram{3,1+1}*[*(red)]{1,1+1}$, where we have shaded the boxes of the main diagonal in \textcolor{red}{red}.

We write $\cD$ for the doubled alphabet 
\[
1' < 1 < 2' < 2 < 3' < \cdots
\]
and set $\cN' \coloneqq  \cD \setminus \cN$.
A \newword{shifted (semistandard Young) tableau} of \newword{shape} $\lambda$ is a filling of the shifted Young diagram of $\lambda$ with elements of $\cD$ so that rows and columns are nondecreasing, each $i' \in \cN'$ appears at most once in any row, each $i \in \cN$ appears at most once in any column, and no $i' \in \cN'$ appears on the main diagonal. 
For example,     \begin{ytableau}
            1 &1 & 5'&6 \\
            \none&4&5'&7
    \end{ytableau} is a shifted tableau of shape $(4,3)$.
The \newword{content} of a shifted tableau $T$ is the integer vector $(c_1, c_2, \dots)$ where $c_i$ records the number of boxes of $T$ that are filled with either $i$ or $i'$. We say a shifted tableau $T$ is \newword{standard} if there is some $k$ such that $c_i = \begin{cases}
    1, & \text{if $i \leq k$};\\
    0, & \text{if $i > k$};
\end{cases}$
for all $i$ and $T$ contains no primed entries. We write $\shSSYT(\lambda)$ for the set of all shifted tableaux of shape $\lambda$ and $\shSYT(\lambda)$ for the subset consisting of standard tableaux. The
\newword{shifted Yamanouchi tableau} $\yama_\lambda \in \shSSYT(\lambda)$ (pronounced ``Yama'') is the unique tableau of its shape such that all entries in each row $i$ are equal to $i$. For example, 
\[
\yama_{(4,2,1)} = \ytableaushort{1111,\none 22,\none  \none 3}.
\]

\ytableausetup{nosmalltableaux}

Given a shifted tableau $T$ and a letter $x \in \cN$, we insert $x$ in row $i$ of $T$ as follows. If no entry of row $i$ is strictly greater than $x$, we place $x$ in a new box at the right end of row $i$. Otherwise, let $y \in \cD$ be the leftmost entry that is strictly greater than $x$. We replace the entry $y$ with the letter $x$. (Note that, in general, this construction may not produce a valid tableau; however, we will only apply it in cases where the result is valid.) 

Similarly, we insert $x' \in \cN'$ into column $j$ of $T$ as follows. If no entry of column $j$ is strictly greater than $x'$, we place $x'$ in a new box at the bottom of column $j$. Otherwise, let $y \in \cD$ be the topmost entry that is strictly greater than $x'$. We replace the entry $y$ with the letter $x'$.

We now recall the \newword{mixed insertion} algorithm of \cite{Haiman}.
Let $b = (u,v) \in (\cN \times \cN)^*$ be a biword. Then the \newword{mixed insertion tableau} of $b$ is the shifted tableau $P_\mix(b)$ constructed as follows:

\begin{itemize}
    \item Starting with an empty shifted tableau, insert $v_1, v_2, \dots, v_n$ in turn into the first row.
    \begin{itemize}
        \item[$\bullet$] Each time that an element $y' \in \cN'$ in column $j$ is replaced, insert $y'$ into column $j+1$.
        \item[$\bullet$] Each time that an element $y \in \cN$ is replaced on the main diagonal of column $j$, insert $y'$ into column $j+1$.
        \item[$\bullet$] Each time that an element $y \in \cN$ is replaced in an off-diagonal position of row $i$, insert $y$ into row $i+1$.
    \end{itemize}
\item The resulting tableau is $P_\mix(b)$.
\end{itemize}

Insertion of each letter $v_k$ results in adding a unique new box to the shape of the tableau $P_\mix(b)$ being constructed. Recording $u_k$ in the corresponding box of another tableau of the same shape yields the \newword{mixed recording tableau} $Q_\mix(b)$.

\begin{figure}[htbp]
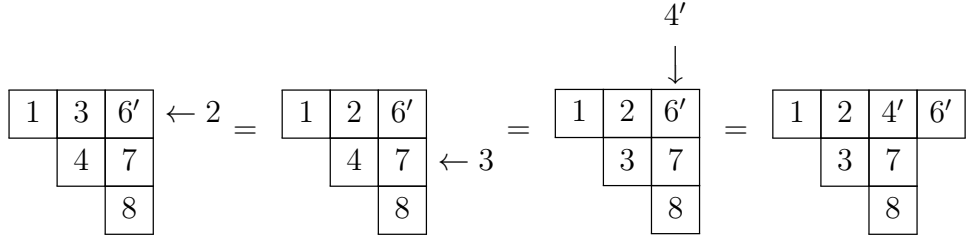

    \centering
    \[
  \begin{ytableau}
        1&3&6'&\none[\quad \leftarrow 2]\\
        \none&4&7 &\none\\
        \none & \none & 8 & \none
    \end{ytableau}  \hspace{3 mm} =\hspace{2 mm}
      \begin{ytableau}
        1&2&6'&\none\\
        \none&4&7 &\none[\quad \leftarrow 3]\\
        \none & \none & 8 & \none
    \end{ytableau}  \hspace{3 mm} =\hspace{2 mm}
    \raisebox{3.1em}{\begin{ytableau}
        \none&\none&\none[4']\\
        \none&\none&\none[\big \downarrow]\\
        1&2&6'\\
        \none&3&7 \\
        \none & \none & 8 
    \end{ytableau}} \hspace{2 mm} = \hspace{2 mm}
    \begin{ytableau}
        1&2&4'&6'\\
        \none&3&7 \\
        \none & \none & 8 
    \end{ytableau}
\]
    \caption{An example of the mixed insertion algorithm. The depicted shifted tableaux illustrate the process of inserting the number $2$ into the shifted semistandard Young tableau on the far left.}
    \label{fig:mixedEx}
\end{figure}

The \newword{shifted plactic monoid} $\sPlactic$ is the quotient of the free monoid on $\cN$, where we declare two words $w, w'$ are equivalent if $P_\mix\binom{[n]}{w} = P_\mix \binom{[n]}{w'}$. Serrano \cite{Serrano} showed that $\sPlactic$ is the quotient by $8$ relations of degree $4$:
\begin{proposition}[\cite{Serrano}]\label{def:shifted_plactic}
    The shifted plactic monoid $\sPlactic$ is the quotient of the free monoid on $\cN$ by the \newword{shifted plactic relations}:
    \begin{align}
        &abdc\sim adbc \quad \text{for all }  a\leq b \leq c<d; \tag{SP.1}\label{eq:SP1} \\
    &acdb\sim acbd \quad \text{for all }  a\leq b < c \leq d; \tag{SP.2}\label{eq:SP2} \\ 
        &dacb\sim adcb \quad \text{for all }  a\leq b < c <d;\tag{SP.3}\label{eq:SP3} \\ 
        &badc\sim bdac \quad \text{for all }  a< b \leq c<d; \tag{SP.4}\label{eq:SP4} \\ 
        &cbda\sim cdba \quad \text{for all }  a< b <c \leq d; \tag{SP.5}\label{eq:SP5} \\ 
        &dbca\sim bdca \quad \text{for all }  a< b \leq c<d; \tag{SP.6}\label{eq:SP6} \\ 
        &bcda\sim bcad \quad \text{for all }  a< b \leq c\leq d; \label{eq:SP7} \tag{SP.7} \\ 
        &cadb\sim cdab \quad \text{for all }  a\leq  b < c\leq d. \tag{SP.8} \label{eq:SP8}
    \end{align}
\end{proposition}
For an intrinsic characterization by universal property, see \cite{EstupinanSalamanca.Pechenik}. It is straightforward to see from either of these perspectives that $\Plactic$ is a quotient of $\sPlactic$.

We now recall how to insert primed letters into rows and unprimed letters into columns.
We insert $x' \in \cN'$ into row $i$ as follows. If all entries of row $i$ are strictly less than $x'$, we place $x'$ in a new box at the right end of row $i$. Otherwise, let $y \in \cD$ be the leftmost entry that is greater than or equal to $x'$. We replace the entry $y$ with the letter $x'$, unless $y = x$ is on the diagonal, in which case we replace $y$ with the letter $x$.

Similarly, we insert $x \in \cN$ into column $j$ as follows. If all entries of column $j$ are strictly less than $x$, we place $x$ in a new box at the bottom of column $j$. Otherwise, let $y \in \cD$ be the topmost entry that is greater than or equal to $x$. We replace the entry $y$ with the letter $x$.

We may now recall the \newword{Sagan--Worley insertion} \cite{Worley, Sagan} of a biword $b=(u,v)  \in  (\cN \times \cD)^*$. 
\begin{itemize}
    \item Starting with an empty shifted tableau, insert $v_1, v_2, \dots, v_n$ in turn into the first row.
    \begin{itemize}
        \item[$\bullet$] Each time that an element $y \in \cD$ in an off-diagonal position of column $j$ is replaced in a column insertion, insert $y\in \cD$ into column $j+1$.
        \item[$\bullet$] Each time an element $y \in \cD$ in an off-diagonal position of row $i$ is replaced in a row insertion, insert $y\in \cD$ into row $i+1$.
        \item[$\bullet$] Each time an element $y' \in \cN'$ is replaced on the main diagonal of column $j$ from insertion of $x' = y'$, insert $y\in \cN$ into column $j+1$.
        \item[$\bullet$] Each time an element $y \in \cD$ is replaced on the main diagonal of column $j$ from insertion of $x \neq y$, insert $y\in \cD$ into column $j+1$.
    \end{itemize}
    \item The resulting tableau is the \newword{Sagan--Worley insertion tableau} $P_\SW(b)$. (In general, this tableau may have primed entries on the diagonal and so may not strictly be a shifted semistandard tableau as we have defined such tableaux; this scenario will not arise when $v \in \cN^*$.)
\end{itemize}

Insertion of each letter $v_k$ results in adding a unique new box $\bbb$ to the shape of the tableau $P_\SW(b)$ being constructed. Recording $u_k$ in the corresponding box $\bbb$ of another tableau of the same shape when $\bbb$ appeared from a row insertion, while recording $u'_k$ when $\bbb$ appeared from a column insertion, yields the \newword{Sagan--Worley recording tableau} $Q_\SW(b)$.

\begin{figure}[htbp]
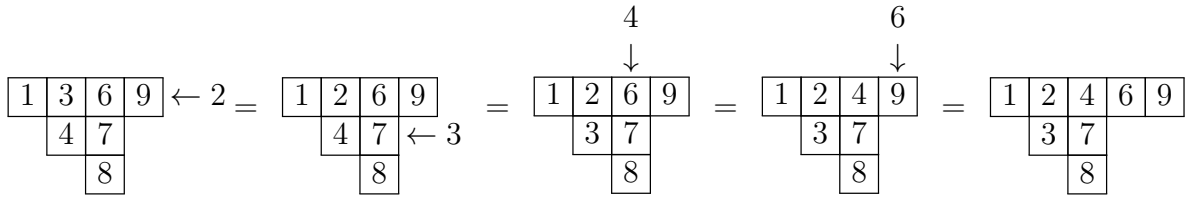

    \centering
    \ytableausetup{boxsize=1.2em}
    \[
  \begin{ytableau}
        1&3&6& 9 &\none[\quad \leftarrow 2]\\
        \none&4&7 &\none\\
        \none & \none & 8 & \none
    \end{ytableau}  \hspace{3 mm} =\hspace{2 mm}
      \begin{ytableau}
        1&2&6&9 &\none\\
        \none&4&7 &\none[\quad \leftarrow 3]\\
        \none & \none & 8 & \none
    \end{ytableau}  \hspace{0 mm} =\hspace{2 mm}
    \raisebox{2.5em}{\begin{ytableau}
        \none&\none&\none[4]\\
        \none&\none&\none[ \downarrow]\\
        1&2&6 & 9\\
        \none&3&7 \\
        \none & \none & 8 
    \end{ytableau}}
    \hspace{2 mm} = \hspace{2 mm}
    \raisebox{2.5em}{\begin{ytableau}
      \none & \none & \none & \none[6] \\
      \none&\none& \none & \none[ \downarrow]\\
        1&2&4&9\\
        \none&3&7 \\
        \none & \none & 8 
    \end{ytableau}}
    \hspace{2 mm} = \hspace{2 mm}
     \begin{ytableau}
        1&2&4&6& 9\\
        \none&3&7 \\
        \none & \none & 8 
    \end{ytableau}
\]     \ytableausetup{boxsize=normal}
    \caption{An example of the Sagan--Worley insertion algorithm. The depicted shifted tableaux illustrate the process of inserting the number $2$ into the shifted semistandard Young tableau on the far left.}
    \label{fig:SaganWorleyEx}
\end{figure}

Although {\it prima facie} very different, these two insertion algorithms are related by a duality on biwords. For a biword $b \coloneqq \binom{u}{v}$, let $\sigma(b) \coloneqq \binom{v}{u}$. Then, for all biwords $b\in (\cN \times \cN)^*$, we have \cite{Haiman, Fomin} (see also, \cite[Theorem~3.6]{SerranoThesis}) that
\begin{equation}\label{eq:duality}
    P_\mix(\overrightarrow{b}) = Q_\SW(\sigma(\underrightarrow{b})) \quad \text{and} \quad Q_\mix(\overrightarrow{b}) = P_\SW(\sigma(\underrightarrow{b})).  
\end{equation}

For $T \in \shSYT(\lambda)$, the \newword{free $P$-Schur function} $\sFreeSchur_T$ is
\begin{align}\label{eq:freePSchurDef}
    \sFreeSchur_T \coloneqq \sum_{Q_\mix\binom{[n]}{w} = T} x^w \in \Z \npowerser{x_1, x_2 ,\dots}.
\end{align}
 The \newword{plactic $P$-Schur function} $\sPlacticSchur_\lambda$ is the image of $\sFreeSchur_T$ under the quotient where $x^w = x^{w'}$ whenever $w, w'$ are equivalent in the shifted plactic monoid $\sPlactic$. (Again, the plactic $P$-Schur function is independent of the choice of  $T \in \shSYT(\lambda)$.) We write $\Z \npowerser{\sPlactic}$ for the span of the plactic $P$-Schur functions. 
    The \newword{$P$-Schur function} $P_\lambda$ is the symmetric function 
    \[P_\lambda \coloneqq \sum_{T\in \shSSYT(\lambda)}x^{c(T)}=\sum_{T\in \shSSYT(\lambda)}x_1^{c_1}x_2^{c_2}\ldots x_{\ell}^{c_{|\lambda|}}.\] The $P$-Schur function $P_\lambda$ is the image of the plactic $P$-Schur function $\sPlacticSchur_\lambda$ under the abelianization map making all variables commute.

    The \newword{$Q$-Schur function} $Q_\lambda$ is 
    defined by 
    \[
    Q_\lambda \coloneqq 2^{\ell(\lambda)} P_\lambda.
    \]
We can also think of $Q_\lambda$ as the generating function for a generalization of shifted semistandard tableaux where we allow primed entries to appear on the main diagonal, as in the most general kind of Sagan--Worley insertion tableau.

\section{New proof of Serrano's shifted plactic Littlewood--Richardson rule}\label{sec:LR}
In this section, we recall Serrano's shifted Littlewood--Richardson rule \cite{Serrano} in terms of shifted plactic classes; see Theorem~\ref{thm:placticLRrule}. We then give a new proof of this rule. Our proof differs from Serrano's in proceeding by developing relations among free $P$-Schur functions. In this way, our proof more closely tracks the approach to the ordinary Littlewood--Richardson rule taken in \cite{Lascoux.Schutzenberger:plaxique} (see also, \cite{Lothaire}). In Corollary~\ref{cor:stembridge},  we recall a shifted Littlewood--Richardson rule equivalent to Stembridge's \cite{Stembridge}, and we extract a new algebraic proof for this as well.

We write $[w]$ for the shifted plactic class of a word $w$. If $T$ is a shifted tableau with $P_\mix(w) = T$, we write $[T] \coloneqq [w]$.

\begin{theorem}[Serrano \cite{Serrano}]\label{thm:placticLRrule}
    Let $\lambda, \mu, \nu$ be strict partitions and fix a tableau $T$ of shifted shape $\nu$. Then the shifted Littlewood--Richardson coefficient $b_{\lambda,\mu}^\nu$ can be determined as follows:
    \[ b_{\lambda,\mu}^\nu = |\{([U],[V]) \: : \: [U]\cdot [V]=[T]\in \sPlactic, \: U \in \shSSYT(\lambda), \text{ and }  V \in \shSSYT(\mu) \} | . \]
\end{theorem}

For $T$ a shifted semistandard tableau, the \newword{reading word} of $T$ is the word $\reading(T)$ obtained by reading the entries of $T$ by rows from left to right and from bottom to top. For example, the reading word of 
\[
S = \ytableaushort{1 {2'} 3, \none 2}
\]
is $\reading(S) = 21 2' 3$. It is straightforward that $P_\SW(\reading(T)) = T$ for all $T$.

Let $\mathcal{A} < \mathcal{B}$ be two alphabets with $a < b$ for all $a \in \mathcal{A}$ and $b \in \mathcal{B}$. 
If $u$ and $v$ are words in the alphabets $\cA$ and $\cB$, respectively, then $u \shuffle v$ denotes the set of \newword{shuffles}, i.e., the set of words $w \in (\cA \cup \cB)^*$ such that $w|_\cA = u$ and $w|_\cB = v$. If $C$ and $D$ are sets of words, then \[
C \shuffle D \coloneqq  \bigcup_{\substack{c \in C \\ d \in D}} c \shuffle d.\]

The following is a related, but more subtle, definition of shuffle for a pair of tableaux.
 \begin{definition}\label{def:tableauShuffle}
Suppose $T$ and $T'$ are shifted standard Young tableaux in the alphabets $\mathcal{A}$ and $\mathcal{B}$, respectively.
We write $T\shuffle T'$ for the set of shifted standard tableaux $S$ such that for some word $u$, $u|_{\mathcal{A}} = \reading(T)$, $P_{\SW}(u|_{\mathcal{B}})= T'$, and $P_\SW(u)=S$. 
 \end{definition}

Sagan \cite{Sagan} shows that two words $w$ and $u$ without repetitions or primed entries have the same Sagan--Worley insertion if and only if they are related to each other by a finite sequence of applications of the equivalences 
\begin{align}
        &\mathbf{x}acb\mathbf{y}\sim_\SK \mathbf{x}cab\mathbf{y}, \quad &\text{for all} \quad a< b < c \in \cN \quad \text{and} \quad \mathbf{x,y} \in \cN^*;\\
        &\mathbf{x}bca\mathbf{y} \sim_\SK \mathbf{x}bac\mathbf{y}, \quad &\text{for all} \quad a< b < c \in \cN \quad \text{and} \quad \mathbf{x,y} \in \cN^*;\\
        &ab\mathbf{x}\sim_\SK ba\mathbf{x}, \quad &\text{for all} \quad a,b\in \cN \quad \text{and} \quad \mathbf{x} \in \cN^*.
    \end{align}
 Importantly, these equivalences are invariant under restriction to intervals in $\mathcal{A}\cup \mathcal{B}\subseteq \cN.$

Consider $S=P_\SW(u)\in T\shuffle T'.$ The following lemma gives an alternative characterization of Definition~\ref{def:tableauShuffle}. For the definition of \emph{jeu de taquin rectification} in this context, see \cite{Worley, Sagan}.

\begin{lemma}\label{lem:algLemma}
    Fix $T$ and $T'$ shifted standard tableaux in the alphabets $\mathcal{A}$ and $\mathcal{B}$, respectively. Then $T\shuffle T'$ is equal to the set of tableaux $R$ such that for some $u$ with $P_\SW(u)=R$, we have $u|_\mathcal{A}=\reading(T)$ and that $R|_\mathcal{B}$ rectifies to $T'$ under shifted jeu de taquin.
\end{lemma}
\begin{proof}
       Fix $T$ and $T'$ and suppose that $S=P_\SW(u)\in T\shuffle T'$, $u|_\mathcal{A}=\reading(T)$ and $P_{\SW}(u|_\mathcal{B})=T'$. Since rectification respects shifted Knuth equivalence \cite[Theorem~6.4.3]{Worley} and Sagan--Worley insertion does as well \cite[Theorem~6.2.2]{Worley}, we have that
    \begin{align*}
        \reading(\rectify(S|_\mathcal{B})) &\sim_{\SK} \reading(S|_\mathcal{B}) \\
        &= \reading(P_\SW(u)|_\mathcal{B}) \\
        &\sim_\SK \reading(P_\SW(u|_\mathcal{B})) 
    \end{align*}
    But there is a unique straight-shaped tableau in each shifted Knuth equivalence class \cite[Theorem~7.2]{Sagan}, so this implies
    \[\rectify(S|_\mathcal{B})  =  P_{\SW}(u|_\mathcal{B})=T'. \qedhere\]
\end{proof}

The next lemma characterizes words whose insertion is an element of the shuffle $T \shuffle T'$.

\begin{lemma}\label{lem:shuffleDescription}
     Let $\mathcal{A} < \mathcal{B}$ be two alphabets with $\mathcal{A},\mathcal{B}\subseteq \cN$. Suppose $T \in \shSYT(\lambda)$ and $T' \in \shSYT(\mu)$ are shifted standard tableaux in the alphabets $\mathcal{A}$ and $\mathcal{B}$, respectively. Then
         \begin{equation}\label{eq:shuffleDescription}
        \{ u \: : \: P_{\SW}(u) \in  T\shuffle T'\} =\{v \: : \: v \in (\mathcal{A}\cup \mathcal{B})^*, \:  P_{\SW}(v|_{\mathcal{A}})=T, \:\text{and }\: P_\SW(v|_\mathcal{B})=T'\}.
    \end{equation}
\end{lemma}
\begin{proof}
    We prove that the shuffle $[\reading(T)] \shuffle \{w \: : \: P_{\SW}(w) = T'    \}$ equals both sides of \eqref{eq:shuffleDescription}, starting with the right side. 

   Suppose $v$ is an element of the right side of \eqref{eq:shuffleDescription}.
   Since $P_{\SW}(v|_{\mathcal{A}})=P_{\SW}(\reading(T))$, we have that $v|_{\mathcal{A}} \sim_\SK \reading(T)$ even if $v|_\cA \neq \reading(T)$. But then $v\in [\reading(T)] \shuffle \{w \: : \: P_{\SW}(w) = T'    \}$. 

   Conversely, if $c \in [\reading(T)] \shuffle \{w \: : \: P_{\SW}(w) = T'    \}$, then $c|_\mathcal{A} \sim_{\SK} \reading(T)$ and $P_\SW(c|_\mathcal{A})=P_\SW(\reading(T))=T$. Hence $c  \in \{v \: : \: v \in (\mathcal{A}\cup \mathcal{B})^*, \:  P_{\SW}(v|_{\mathcal{A}})=T, \:\text{and }\: P_\SW(v|_\mathcal{B})=T'\}.$

   Consider now $u$, an element of the left side of \eqref{eq:shuffleDescription}. By definition, $u$ satisfies that $P_{\SW}(u)=S$ for some $S\in T\shuffle T'$.  Say $S \in \shSYT(\nu)$. As all such choices of $S$ are tableaux in the alphabet $\mathcal{A}\cup \mathcal{B}$, the tableaux $T$ and $T'$ are standard, and $\mathcal{A} < \mathcal{B}$, it follows from the definition of Sagan--Worley insertion that $P_\SW(u)|_\mathcal{A} =P_\SW(u|_\mathcal{A})$. Since $S|_\mathcal{A}= P_\SW(u)|_\mathcal{A}$, this yields    \[P_\SW(u|_\mathcal{A}) = S|_\mathcal{A}.\]
But by Definition~\ref{def:tableauShuffle}, there exists $r \in (\cA \cup \cB)^*$ such that $S=P_\SW(r)$ and $r|_\mathcal{A} = \reading(T)$. Hence,
\[ S|_{\mathcal{A}} =P_{\SW}(r)|_\mathcal{A} =P_{\SW}(r|_\mathcal{A})=P_{\SW}(\reading(T))=T,      \]
using again that $\mathcal{A}<\mathcal{B}.$ Thus $P_{\SW}(u|_\mathcal{A})=T$ and therefore $u|_\mathcal{A}\in [\reading(T)]$. 

Similarly, 
\[P_{\SW}(u|_\mathcal{B}) = S|_\mathcal{B} = T'.   \]
In other words, $u|_\mathcal{A} \sim_{\SK} \reading(T) $ and $P_{\SW}(u|_\mathcal{B}) = T'$, implying \[
u\in [\reading(T)] \shuffle \{w \: : \: P_{\SW}(w) = T'    \}.\]

Letting now $c\in [\reading(T)] \shuffle \{w \: : \: P_{\SW}(w) = T'    \}$, $R=P_{\SW}(c)$, and $z=\reading(R)$, we have $z\sim_{\SK} c$ so that
\[T=P_{\SW}(c)|_\mathcal{A} = P_{\SW}(z)|_\mathcal{A} = P_{\SW}(z|_\mathcal{A}).      \]
Similarly,
\[T' = P_{\SW}(c)|_\mathcal{B} = P_{\SW}(z)|_\mathcal{B}. \]
Now, note that $z|_\mathcal{A}=\reading(T)$  because $R$ is a tableau such that the letters of $\mathcal{A}$ inside it form the tableau $T$, while the letters of $\mathcal{B}$ make a skew tableau $M$ of shape $\nu/\lambda$. Hence $z\in T\shuffle T'$. Thus $c\in \{ u \: : \: P_{\SW}(u) \in  T\shuffle T'\}$.   
\end{proof}

\begin{lemma}\label{lem:TheGreatShuffleLemma}
    Let $\mathcal{A} < \mathcal{B}$ be two alphabets with $\mathcal{A},\mathcal{B}\subseteq \cN$. Suppose $T$ and $T'$ are shifted standard Young tableaux in the alphabets $\mathcal{A}$ and $\mathcal{B}$, respectively. 
    Then
    \begin{equation}\label{eq:shuffle}
         \{ w : P_{\SW}(w)=T \}  \shuffle \{ v :  P_{\SW}(v)=T' \}   = \{ u :  P_{\SW}(u) \in T \shuffle T' \}. 
    \end{equation}
\end{lemma}
\begin{proof}
  Note that the elements of the left side of \eqref{eq:shuffle} are exactly the words $m$ such that $P_{\SW}(m|_\mathcal{A})=T$ and $P_{\SW}(m|_\mathcal{B})=T'$.
   The result now follows from Lemma~\ref{lem:shuffleDescription}.
\end{proof}

Recall that we identify the word $w$ of length $n$ with the biword $\binom{[n]}{w}$. Given two such biwords $\binom{[n]}{w}$ and $\binom{[n']}{w'}$, we define a product 
\[
\binom{[n]}{w} \ast \binom{[n']}{w'} \coloneqq \binom{[n]}{w} \cdot \binom{[n+1, n + n']}{w'},
\]
where the $\cdot$ denotes concatenation of biwords.

Recall also that the monomials in Equation~\eqref{eq:freePSchurDef} are indexed by words. To avoid exponents, we write $\monomial(w) \coloneqq x^w$.  Moreover, if $b$ is a biword with $\overrightarrow{b} = \binom{u}{v}$, then we write $\monomial(b) = x^v$. Note that $\monomial(b) = \monomial(b')$ whenever $b \approx b'$. 

For two shifted standard tableaux $T$ and $U$, the monomials of $\sFreeSchur_T \cdot \sFreeSchur_U $ are indexed by concatenations of words $w\cdot w'$, where $Q_\mix(w)=T$ and $Q_\mix(w')=U$. 
Accordingly, we have that
\begin{align*}
    \sFreeSchur_T \cdot \sFreeSchur_U &= \sum_{Q_\mix(w)=T} \monomial(w) \cdot \sum_{Q_\mix(w')=U} \monomial(w') = \sum_{\substack{Q_\mix(w)=T \\ Q_\mix(w')=U}} \monomial \binom{[\ell+m]}{w\cdot w'} \\ &= \sum_{\substack{ Q_\mix\binom{[\ell]}{w}=T  \\ Q_\mix\binom{[m]}{w'}=U  }} \monomial \left( \binom{[\ell]}{w} \ast \binom{[m]}{w'} \right) \\ 
    &= \sum_{\substack{ Q_\mix\binom{[\ell]}{w}=T  \\ Q_\mix\binom{[m]}{w'}=U  }} \monomial \left( \binom{[\ell]}{w} \cdot  \begin{pmatrix}
        \ell+1 & \ell+2 & \dots & \ell+m \\
        w'_1 & w'_2 & \dots & w'_m
    \end{pmatrix} \right).
\end{align*} 
That is,
\begin{equation}\label{eq:conclusionFreeDiscussion}
    \sFreeSchur_T\cdot \sFreeSchur_U = \sum_{Q_\mix(w)=T} \monomial\binom{[\ell]}{w} \sum_{Q_\mix(w')=U} \monomial\binom{[\ell+1,\ell+m]}{w'}.
\end{equation}

\begin{theorem}\label{thm:placticShiftedLR}
    The plactic $P$-Schur functions satisfy
    \[
    \sPlacticSchur_\lambda \sPlacticSchur_\mu = \sum_{\nu} b_{\lambda,\mu}^\nu \sPlacticSchur_\nu,
    \]
    where $b_{\lambda,\mu}^\nu$ are the structure coefficients of the (ordinary) $P$-Schur functions.
\end{theorem}
\begin{proof}
Let $T_\lambda$ and $T_\mu$ be shifted standard tableaux of shapes $\lambda$ and $\mu$, respectively. Set $\ell = |\lambda|$ and $m = |\mu|$. We also let $T_\mu'$ be the tableau obtained after applying the unique strictly increasing map \[
\iota:[m]\longrightarrow [\ell+1,\ell+m]
\]
to the entries of $T_\mu.$ 
Then we have
\[ \sFreeSchur_{T_\mu'}= \sum_{Q_\mix \binom{[\ell+1,\ell+m]}{w'} = T_\mu'} \binom{[\ell+1,\ell+m]}{w'} = \sum_{Q_\mix \binom{[\ell+1,\ell+m]}{w'} = T_\mu'} \begin{pmatrix}
        \ell+1 & \ell + 2 & \dots & \ell+m \\
        w'_{1} & w'_{2} & \dots & w'_{m}
    \end{pmatrix}.\]

 Therefore by \eqref{eq:duality} and \eqref{eq:conclusionFreeDiscussion}, we have
\begin{align}
    \sFreeSchur_{T_\lambda} \sFreeSchur_{T_\mu} &= \left( \sum_{Q_\mix \binom{[\ell]}{w} = T_\lambda} \monomial(w) \right) \left(  \sum_{Q_\mix \binom{[\ell+1,\ell+m]}{w'} = T_\mu'} \monomial(w') \right) \nonumber \\ 
    &= \left( \sum_{P_\SW \overrightarrow{\binom{w}{[\ell]}} = T_\lambda} \monomial \binom{[\ell]}{w} \right)  \left(  \sum_{P_\SW \overrightarrow{\binom{w'}{[\ell+1,\ell+m]}} = T_\mu'} \monomial \binom{[\ell+1,\ell+m]}{w'} \right) \nonumber\\
    &= \left( \sum_{P_\SW \overrightarrow{\binom{w}{[\ell]}} = T_\lambda} \monomial \left( \underrightarrow{\binom{[\ell]}{w} } \right) \right) \left(  \sum_{P_\SW \overrightarrow{\binom{w'}{[\ell+1,\ell+m]}} = T_\mu'} \monomial \left( \underrightarrow{\binom{[\ell+1,\ell+m]}{w'}} \right) \right) \nonumber\\
    &\eqqcolon \left( \sum_{ \substack{ \rightarrowtriangle  \\P_\SW (r) = T_\lambda  }} \monomial \binom{r}{p} \right) \left(  \sum_{ \substack{ \rightarrowtriangle \\P_\SW (r') = T_\mu' }} \monomial \binom{r'}{p'}\right),  \label{eq:biwordEquation}
\end{align}
where a triangle-headed right arrow $\rightarrowtriangle$ on the summation indicates that all of its summands are weakly increasing as biwords with respect to the bottom row. That is, $r,p \in \cN^*$ are such that 
\[\binom{r}{p}= \underrightarrow{\begin{pmatrix}
        1 & 2  & \dots & \ell \\
        w_1 & w_2  & \dots & w_{\ell}
    \end{pmatrix} }, \]
    and $r',p' \in \cN^*$ are such that
\[\binom{r'}{p'} = \underrightarrow{\begin{pmatrix}
        \ell+1 & \ell + 2 & \dots & \ell+m \\
        w'_{1} & w'_{2} & \dots & w'_{m}
    \end{pmatrix}}.\]
Note that $r$ and $r'$ are determined by $w$ and $w'$, respectively. 

Remark that all the biword monomials in the product~\eqref{eq:biwordEquation} are such that their upper coordinate belongs to the shuffle $r\shuffle r'$, while their lower coordinate is free, provided the biword is weakly increasing with respect to the bottom row. 

Conversely, since Sagan--Worley insertion respects restriction to intervals, every biword monomial in the set
\[ \left\{\binom{r''}{p''} \: : \: r''\in r\shuffle r' \text{ and the biword is weakly increasing}   \right\} \]
comes in a unique way from the product of two monomials, one obtained by restricting to those biletters with upper coordinate in $[1,\ell]$ and the other by restricting to those in $[\ell+1, m+\ell].$ Hence,
\[\sFreeSchur_{T_\lambda} \sFreeSchur_{T_\mu} = \sum_{\substack{r \: : \:  P_\SW (r)=T_\lambda \\  r' \: : \:   P_\SW(r')=T_\mu' }} \sum_{\substack{ \rightarrowtriangle  \\r''\in r\shuffle r'}} \monomial
\binom{r''}{p''}\]

Note that for every possible $r''\in r \shuffle r'$ there exists at least one word $p''$ with 
\[ \underrightarrow{\binom{r''}{p''}} = \binom{r''}{p''},  \]
so that all such $r''$ contribute nontrivial summands therein.

But, using Equation~\eqref{eq:duality} and Lemma~\ref{lem:TheGreatShuffleLemma}, we can also cast the last expression in terms of its mixed insertion to obtain that
\begin{align}\label{eq:yay}
\begin{split}
  \sFreeSchur_{T_\lambda} \sFreeSchur_{T_\mu}  &=  \sum_{S\in T_\lambda\shuffle T_\mu'} \; \sum_{Q_\mix(w'')=S} \monomial \begin{pmatrix}
        1 & 2  & \dots & \ell &\ell+1 & \ell + 2 & \dots & \ell+m \\
         &   &  & &w''
         &  &  & 
    \end{pmatrix}  \\
    &= \sum_{S\in T_\lambda\shuffle T_\mu'}\sFreeSchur_S.
    \end{split}
\end{align}
Projecting Equation~\eqref{eq:yay} from $\Z \npowerser{x_1, x_2, \dots}$ to $\Z\npowerser{\sPlactic}$, we then obtain that
\begin{equation}\label{eq:projectedEq}
    \sPlacticSchur_{\lambda}\sPlacticSchur_{\mu} = \sum_{\nu}f_{\lambda, \mu}^\nu \, \sPlacticSchur_\nu,    
\end{equation}
where $f_{\lambda, \mu}^\nu$ is the number of tableaux $S\in T_\lambda \shuffle T_\mu$ whose shape is $\nu$. By further projecting Equation~\eqref{eq:projectedEq} to $\Z \llbracket x_1, x_2, \dots \rrbracket$, we also obtain
\begin{equation}
    P_{\lambda}P_{\mu} = \sum_{\nu}f_{\lambda, \mu}^\nu \, P_\nu.    
\end{equation}
But the set $\{P_\nu\}_\nu$ is linearly independent. Hence, $f_{\lambda, \mu}^\nu = b_{\lambda, \mu}^\nu$, as desired. 
\end{proof}

\begin{proof}[Proof of Theorem~\ref{thm:placticLRrule}]
    This is immediate from Theorem~\ref{thm:placticShiftedLR} and the definition of $\sPlacticSchur_\lambda$ and $\sPlacticSchur_\mu$ as the sums over all shifted plactic classes of shapes $\lambda$ and $\mu$, respectively. 
\end{proof}

Using Theorem~\ref{thm:placticShiftedLR}, we now extract a shifted Littlewood--Richardson rule that is essentially equivalent to that of Stembridge \cite{Stembridge}, and give it a new algebraic proof.

\begin{corollary}\label{cor:stembridge}
    Let $\lambda$, $\mu$, and $\nu$ be strict partitions, and $T_\lambda$ a fixed shifted standard tableau of shape $\lambda$. Then $b_{\lambda,\mu}^\nu$ is equal to the number of shifted tableau $U$ of shape $\nu/\mu$ such that $\rectify(U) = T_\lambda.$
\end{corollary}
\begin{proof}
    For convenience, we show instead that $b_{\lambda,\mu}^\nu$ is the number of tableaux of shape $\nu/\lambda$ rectifying to a fixed standard shifted tableau $D$ of shape $\mu.$ 

    By the reasoning in the proof of Theorem~\ref{thm:placticShiftedLR}, the coefficient $b_{\lambda,\mu}^\nu$ is equal to the number of tableaux $S\in T_\lambda \shuffle T_\mu$ with shape $\nu$. As this holds for arbitrary standard shifted tableaux $T_\lambda$ and $T_\mu$ of the appropriate shapes, let us set $T_\mu:= D$ and choose any $T_\lambda$ of shape $\lambda$. We assume that the entries of $T_\lambda$ and $T_\mu$ belong to alphabets $\mathcal{A}$ and $\mathcal{B}$, respectively, such that $\mathcal{A} < \mathcal{B}$.

    By Definition~\ref{def:tableauShuffle}, all such tableaux $S\in T_\lambda \shuffle T_\mu$ will have $S|_{\mathcal{A}}=T_\lambda$, so that deleting all letters of $\mathcal{A}$ from $S$ leaves us with a tableau of skew shape $\nu/\lambda$. It is a consequence of Lemma~\ref{lem:algLemma}, that $S$ must also satisfy $\rectify(S)=T_\mu=D$. 

    Noting that we can also produce a tableau in $T_\lambda\shuffle T_\mu$ from a skew tableau $R$ of shape $\nu/\lambda$ with $\rectify(R)=T_\mu$; namely, by placing $T_\lambda$ so as to make $R$ a (non-skewed) tableau; we obtain that $b_{\lambda,\mu}^\nu$ is equal to the number of skew tableaux of shape $\nu/\lambda$ rectifying to $D$. 
\end{proof}

\section{Barely Yamanouchi words}\label{sec:barelyYamanouchi}

\subsection{Definition and motivation}
We will now consider a generalization of Yamanouchi tableaux to the shifted context, that is, a family of shifted tableaux with analogous properties to classically Yamanouchi tabelaux and from which a new Littlewood--Richardson rule will be established for $P$-Schur functions. 

\begin{definition}\label{def:barelyYamanouchi}
    Let $\nu$ be a strict partition of $n$. The \emph{barely Yamanouchi sequence associated to $\nu$} is the sequence $\hat{y}_\nu\in \cN^*$ given by:
    \[ \nu_\ell \quad \nu_{\ell} -1 \quad \ldots \quad 1 \quad \nu_{\ell-1} \quad \nu_{\ell -1} -1 \quad \ldots \quad 1 \quad \ldots \quad \nu_{1} \quad \nu_{1}-1 \quad \ldots \quad 1 \]
    where $\ell$ is the length of $\nu.$
\end{definition}

\begin{example}
    For the strict partition $\nu=(6,4,3)$, we have that its barely Yamanouchi sequence corresponds to:
    \[ \hat{y}_\nu = 3 \quad 2 \quad 1 \quad 4 \quad 3 \quad 2 \quad 1 \quad 6 \quad 5 \quad 4 \quad 3 \quad 2 \quad 1. \qedhere \]
\end{example}

\begin{definition}
    Given a strict partition $\nu$ of $n$, we define the \newword{barely Yamanouchi tableau} of shape $\nu$, denoted by $\hat{Y}_\nu$, to be the shifted tableau resulting from the mixed insertion of the barely Yamanouchi sequence from the same shape $\hat{Y}_\nu = P_\mix (\hat{y}_\nu)$. Any word $w\in \cN^*$ satisfying $P_\mix(w)= \hat{Y}_\nu$ for some strict partition $\nu$ will be called a \newword{barely Yamanouchi word}. 
\end{definition}

We now compare and contrast the definition of $\hat{Y}_\lambda$ to the definition of shifted Yamanouchi tableaux $\yama_\lambda$ from Section~\ref{sec:isotropic}. Visually, $\yama_\lambda$ is more obviously an analogue of the ordinary Yamanouchi tableau $Y_\mu$, for $\mu$ a nonshifted shape, than $\hat{Y}_\lambda$ is. For example, if $\lambda = (6,3)$ we have that
\[ \hat{Y}_\lambda = 
\begin{ytableau}
     1 & 1 & 2' & 3' & 5' & 6' \\
    \none & 2 & 3' & 4' &\none &\none 
\end{ytableau}, \quad \text{while} \quad 
\yama_\lambda = 
\ytableaushort{111111, \none 222}.
\]
On the other hand, the class of words that mixed insert to $\yama_\lambda$ (that is to say, the shifted plactic class of $\yama_\lambda$ or the \newword{shifted Yamanouchi words}) is difficult to characterize. In particular, this class of words does not satisfy an obvious analogue of the classical Yamanouchi condition. Moreover, the class of shifted Yamanouchi words is not closed under taking right factors.
For example, 
\[
P_\mix(111222111) = \yama_{(6,3)},
\]
while the right factor $222111$ satisfies
\[
P_\mix(222111)=
\begin{ytableau}
    1&1&1&2'\\
    \none&2&2&\none 
\end{ytableau},
\]
which is not $\yama_\theta$ for any strict partition $\theta$. On the other hand, we will show in Corollary~\ref{cor:rightisbarelyYam} that the class of barely Yamanouchi words, like the class of classically Yamanouchi words, is closed under taking right factors. Moreover, we will show in Lemma~\ref{lem:barelyYamanouchiareShiftedLattice} that the class of barely Yamanouchi words has an elegant combinatorial description. In Corollary~\ref{cor:barelyYamisYam}, we will show that barely Yamanouchi words sit inside the class of classically Yamanouchi words.

The motivation for the definition of barely Yamanouchi words and tableaux is combinatorial in nature and proceeds from a characterization of Yamanouchi words. To wit, by Greene's theorem we know that the shape of an RSK insertion tableaux is determined by its longest decreasing (increasing) subwords. From this perspective, it is immediate that for an arbitrary partition $\nu$, any Yamanouchi word $y_\nu$ has an insertion tableau $P(y_\nu)$ of shape $\nu$; after all, the word $y_\nu$ is constructed to have a transparent decreasing (increasing) structure. Moreover, we can regard the Yamanouchi words corresponding to the partition $\nu$ as those words $w$ of content $\nu$ such that the longest decreasing (increasing) subwords of $w$ have the same length as the longest decreasing (increasing) subwords of a fixed Yamanouchi word $y_\nu.$

Thinking in this way of the Yamanouchi property as stemming from combinatorial properties of a word, it is natural to consider the shifted analogue of the same concept. The appropriate concept in the shifted setting is that of the longest hook subwords as made clear in Theorem~\ref{thm:greeneGeneralization} below. 

\begin{definition}[Serrano~\cite{Serrano}]
    Let $w= w_1 w_2 \ldots w_n \in \cN^*$ be a word. We say that it is a \emph{hook word} if there exists $1\leq k \leq n$ such that 
    \[ w_1 > w_2 > \ldots > w_k \leq w_{k+1} \leq w_{k+2} \leq \ldots \leq w_n \]
    where $k$ is possibly $1$ or $n.$
\end{definition}

\begin{definition}[Serrano~\cite{SerranoThesis}]
    Let $w\in \cN^*$ and $(w_{i_1})_{i_1}, (w_{i_2})_{i_2}, \ldots, (w_{i_k})_{i_k}$ be subwords of $w$ such that they share pairwise at most one letter. We say that those subwords form a \emph{$k$-hook subword} if every subword is a hook subword and if every letter on $w$ appears in at most two subwords.
\end{definition}

\begin{remark}
    A subword of a word $w$ differs from a subsequence of the same word, by virtue of the order of its elements being identical to the relative order they had in the word $w$. That is, if a letter $u$ appears before $v$ in $w$, it also shows up before $v$ in a subword of $w$ containing both letters. This need not be the case for a subsequence thereof.
\end{remark}

The length of a $k$-hook subword is the sum of the sizes of all the subwords comprising it. We let $I_k(w)$ stand for the length of the longest $k$-hook subword of $w$. We now restate Serrano's \cite[Theorem~3.7]{SerranoThesis} (correcting a typo contained in the original version). 

\begin{theorem}[Serrano~{\cite[Theorem~3.7]{SerranoThesis}}]\label{thm:greeneGeneralization}
    Let $w\in \cN^*$ and $T=P_\mix(w)$ be a shifted tableau. Then, the shape $\lambda$ of $T$ is completely determined by the tuple \[(I_1(w), I_2(w), \ldots, I_{\ell(\lambda)}(w)).\] Furthermore,
    \[I_k(w) = \lambda_1 + \ldots + \lambda_k + \binom{k}{2}. \]
\end{theorem}

\begin{example}
    Let $w= 3 \quad 1 \quad 2 \quad 1 \quad 4 \quad 3 \quad 2 \quad 1 \quad 1 \quad 1 $. The longest hook subword of $w$ is $w_1 = 4 \quad 3 \quad 2 \quad 1 \quad 1 \quad 1$. A hook subword is a $1$-hook subword of maximum length, yielding that $I_1(w)=6.$
    
    To form the longest $2$-hook subword we extend the longest hook subword of the remainder, the word $u= 3 \quad 2 \quad 1$, to include the element $2$ of $w_1$, resulting in the word $w_2= 3 \quad 2 \quad 1 \quad 2$ and the $2$-hook subword $\{w_1, \: w_2\}.$ Since $w_1$ and $w_2$ are allowed to share at most one element by definition, this is a $2$-hook subword of maximum length. 
    
    Finally, appending the two letters $2$ from $w_1$ and $3$ from $w_2$ to the only remaining element we get $w_3=1 \quad 2 \quad 3$, thus arriving at the $3$-hook subword $\{w_1, \: w_2, \: w_3\}$ which is of maximal length amongst $3$-hook subwords as it uses all the letters from $w$. 

    Accordingly, we have that $I_1(w)=6$, $I_2(w)=10$, and $I_3(w)=13$. This in turn implies that 
    \[ \lambda_1 = I_1(w) = 6 \]
    \[ \lambda_2  = I_2(w) - \binom{2}{2} - \lambda_1 = 3 \]
    \[ \lambda_3 = I_3(w) - \binom{3}{2} - \lambda_2 - \lambda_1 = 1 .\]
    Comparing with the mixed insertion of $w$, we see that:
    \[ P_\mix(w) = 
    \begin{ytableau}
            1&1&1&1&1&3' \\
            \none&2&2&4'\\
            \none&\none&3
    \end{ytableau} \]
    so that the shape agrees with our calculations.
\end{example}

Following in the logical path set by our understanding of the Yamanouchi class $[y_\nu]$ as consisting of those words with content $\nu$ and having longest decreasing subwords of length $\nu_1,$ $\nu_2$, $\ldots$, and $\nu_{\ell(\nu)}$ respectively; it is only natural in light of Serrano's theorem to define Yamanouchi words in the shifted setting analogously, letting the notion of a hook subword replace that of a decreasing subword. Then, the shifted class of the word $\hat{y}_\nu$ from Definition~\ref{def:barelyYamanouchi} is a natural generalization of Yamanouchi words as shown in the lemma that follows. The word $\hat{y}_\nu$ is specially designed to have a trivial hook--structure, that is, for it to be easy to ascertain its hook subwords. The length of the largest, second largest, and so on, hook lengths will also essentially be $\nu_1,$ $\nu_2$, $\ldots$, and $\nu_{\ell(\nu)}$, respectively. 

\begin{remark}
    Given that in the definition of $k$-hook subwords we are allowing subwords with non-trivial intersections, the resulting length of the maximal subwords constructed in the proof of the lemma will be strictly longer than the desired respective lengths $\nu_1,$ $\nu_2$, $\ldots$, and $\nu_{\ell(\nu)}$. However, if the points in the intersection are accounted for, the lengths correspond without difference to the parts of the partition $\nu.$
\end{remark}

\begin{lemma}\label{lem:hookSequencesLemma}
    Let $\nu$ be a strict partition and $d_n \in \cN^*$ stand for the sequence $d_n = n\quad n-1 \quad n-2 \quad \ldots \quad 1.$ Then, the collection of subwords
    \[ \{  d_{\nu_j} \cup \{\min( u\in d_{\nu_{r}} \: | \: u \notin d_{\nu_j})     \}_{1\leq r < j} \}_{1\leq j \leq \ell(\nu)}\] 
    is a $j$-hook subword of $\hat{y}_\nu$ of maximal length.
\end{lemma}
\begin{proof}
    Suppose by contradiction that this is not the case for a partition $\nu$. Let $k$ be minimal with the property that $\mathcal{D}:=\{ d_{\nu_j} \}_{1\leq j \leq k}$ is not a $k$-hook subword in $\hat{y}_\nu$. Then there exists a $k$-hook subword $\mathcal{B}:=\{(w_{i_j})_{i_j}\}_{1\leq j \leq k}$ of $w$ with length strictly greater. 

    By construction of $\mathcal{D}$ and the hypothesis that $\nu$ is a strict partition, we have that after a possible re-indexing of the sequences in $\mathcal{B}$ the biggest number in each sequence therein, which we denote by $b_m:=\max((w_{i_m})_{i_m})$, satisfies the inequality $b_m\leq \nu_m.$ Moreover, either $b_m$ is the first element of the sequence $(w_{i_m})_{i_m}$ or it is not.
    
    Let $inc_m$ and $dec_m$ denote the weakly increasing and strictly decreasing parts of the hook sequence $(w_{i_m})_{i_m}$ considered as a word. We take the strictly decreasing part to be of maximum length for the choice of $inc_m$ and $dec_m$ to be unique and well defined. 
    
    If $b_m$ is the first element of the sequence $(w_{i_m})_{i_m}$, then $dec_m$ has at most $b_m$ elements as the minimum letter in the word $y_\nu$ is $1$. Moreover, as $(w_{i_m})_{i_m}$ is a subword of $y_\nu$ and considering that $y_\nu$ is a union of decreasing blocks (decreasing subwords of the form $a \quad a-1 \quad a-2 \quad \ldots \quad 1$), $inc_m$ has at most $m-1$ elements. Thus,    
    \[|(w_{i_m})_{i_m}| \leq b_m + m-1 \leq \nu_m + m -1.\]
    
    But the number of letters in $  d_{\nu_m} \cup \{\min( u\in d_{\nu_{r}} \: | \: u \notin d_{\nu_m})     \}_{1\leq r < m}$ is $\nu_m + m -1.$ Hence, 
    \[|(w_{i_m})_{i_m}| \leq | d_{\nu_j} \cup \{\min( u\in d_{\nu_{r}} \: | \: u \notin d_{\nu_j})     \}_{1\leq r < j} |.\]
    This contradicts that $\mathcal{B}$ is a $k$-hook subword of length greater than $\mathcal{D}.$

    Now suppose that $b_m$ is the last element of the subword $\mathcal{B}.$ There are two cases for the leftmost element $e$ of $\mathcal{B}$: either it satisfies $e<b_m$ or $e=b_m$. If the latter occurs, then $b_m$ is also the first element of $\mathcal{B}$ and the previous case applies again, for which the statement follows.

    If $e<m$, then $e$ is contained in a decreasing block possibly not being the greatest element in the block, and to bound $|\mathcal{B}|$ it is necessary to consider how to optimize $inc_m$ and $dec_m$ simultaneously so as to maximize their sum. For $inc_m$ to be optimal, $e$ should be eventually followed in $\mathcal{B}$ by an instance of $1$ belonging to the same decreasing block in $y_\nu$ so that the increasing sequence $inc_m$ can be as long as possible. But notice that this strategy also yields the optimal sequence $dec_m$: indeed, the subword $e \quad e-1 \quad \ldots \quad 1$ is optimal with the constraint of $e$ being its greatest element.

    Let now $\nu_t$ be the maximum element in $y_\nu$ in the decreasing block of $e$. From the analysis of optimality, we have that $|dec_m| \leq e$ and $| inc_m| \leq t-1$. If $t\leq m$, then 
    \[ | \mathcal{B} | \leq e + t-1 \leq b_m + m - 1 \leq \nu_m + m - 1. \]
    If on the other hand $t>m$, then
    \begin{align*}
        | \mathcal{B} | &\leq e + t-1 = (t-m)  + e + m - 1 \\
        &\leq (t-m) + (\nu_m - ( t-m )) + m - 1 \\
        &= \nu_m + m -1.
    \end{align*}
    where the inequality $e \leq \nu_m - ( t-m )$ follows from the fact that $e<\nu_m$ and for each block of separation between the decreasing block of $\nu_m$ and the decreasing block of $e$ the minimum distance between $e$ and $\nu_m$ increases by 1.
    
    In either case the inequality $|\mathcal{B}| \leq |\mathcal{D}|$ holds, because again \[
    |d_{\nu_m} \cup \{\min( u\in d_{\nu_{r}} \: | \: u \notin d_{\nu_m})     \}_{1\leq r < m}| = \nu_m + m -1.
    \]
    This is a contradiction.

    As all cases were covered, the desired result follows.
\end{proof}

The next lemma shows that the words in the shifted plactic class of $\hat{y}_\nu$ also have this desired property. For this reason, we consider hereafter the words in $[\hat{y}_\nu]$, the barely Yamanouchi words, to be the most suitable generalization of Yamanouchi words to $\sPlactic$.

\begin{lemma}
    Let $w$ be such that $w\sim \hat{y}_\nu$ in $\sPlactic.$ Then, the longest hook subwords of $w$ have the same length as those of $\hat{y}_\nu$. That is,
    \[I_k(w) = I_{k}(\hat{ y}_\nu) \qquad \forall k \: : \: 1\leq k \leq \ell(\nu).  \]
\end{lemma}
\begin{proof}
    Since $w\sim \hat{y}_\nu$, it follows that $P_\mix(\hat{y}_\nu) = P_\mix(w)$. In particular, the shape of both tableaux is the same. The desired result follows from Theorem~\ref{thm:greeneGeneralization}.
\end{proof}

It could seem reasonable to expect that the converse condition of having the same content as $\hat{y}_\nu$ and with maximal hook subwords of lengths $\nu_1,$ $\nu_2$, $\ldots$, and $\nu_{\ell(\nu)}$, would be enough to characterize the shifted plactic class of $y_\nu$. However, in contrast to the plactic setting, this is not enough to determine a unique insertion tableau as illustrated in the example below.

\begin{example}\label{ex:notionExample}
    Let $\nu:=(4,2)$ and $w:=121324\in \cN^*$. The insertion tableau of $w$ is
    \[ P_\mix(w) = 
    \begin{ytableau}
            1&1&2&4 \\
            \none&2&3
    \end{ytableau} .\]
    Moreover, the hook subwords $1124$ and $234$ of $w$ form a maximal $2$-hook subword of $w$. In particular, $I_1(w)=4$ and $I_2(w)=3$. 

    Consider the barely Yamanouchi word $\hat{y}_\nu = 214321.$ Its insertion tableau has the same shape as $P_\mix(w)$, but is clearly a different tableau:
    \[ P_\mix(\hat{y}_\nu) = 
    \begin{ytableau}
            1&1&2'&4' \\
            \none&2&3'
    \end{ytableau} .\]
    Furthermore, using the hook subwords $4321$ and $213$, one sees that $I_1(\hat{y}_\nu)=I_1(w)$ and $I_2(\hat{y}_\nu)=I_2(w)$. 
    Thus, the condition of having hook subwords of the same length and using the same letters with the same multiplicity is not sufficient to uniquely determine the shifted plactic class of $\hat{y}_\nu.$
\end{example}

Thus far, our considerations have left us with a strong candidate for the Yamanouchi words in the shifted context; namely, those words in the shifted plactic class $[\hat{y}_\nu].$ We saw in Example~\ref{ex:notionExample}, however, that the increasing and decreasing structure of barely Yamanouchi words is shared by other shifted classes as well. 

It is not {\it a priori} obvious that $[\hat{y}_\nu]$ is preferable to such other shifted classes. However, in its favour, we will show that the class of barely Yamanouchi words is closed under right factors (Corollary~\ref{cor:rightisbarelyYam}), has a clean combinatorial characterization (Theorem~\ref{thm:inBarelyYamanouchiClassiffshiftedYamanouchi}), and is a subclass of classically Yamanouchi words (Corollary~\ref{cor:barelyYamisYam}). Analogous properties are critical to the approach of Lascoux and Sch\"utzenberger in the classical setting and these properties will similarly be critical to the proof of our shifted Littlewood--Richardson rule, Theorem~\ref{thm:main}.

\subsection{Shapes of barely Yamanouchi tableaux}

For the sake of readability, it is convenient to denote the decreasing sequence of letters
\[ n \qquad n-1 \qquad n-2 \qquad \ldots \qquad m  \]
by $\seq(n,m)$. For $m=1$ we also use the notation $\seq(n):= \seq(n,1)$.

The shrinking decomposition of the canonical word $\hat{y}_\nu$ can then be decomposed into $\ell(\nu)$ blocks, as
\[\hat{y}_\nu = \seq(\nu_{\ell(\nu)}) \cdot \seq(\nu_{\ell(\nu) - 1 }) \cdot \ldots \cdot \seq(\nu_1) \]

\begin{lemma}\label{lem:charScarcelyYamanouchiTableaux}
    Let $\nu$ be a strict partition. Then, for all $i$, we have that at any point of the insertion process after all the elements of $\seq(\nu_{\ell(\nu) - i})$ in $\hat{Y}_\nu \coloneqq P_\mix(\hat{y}_\nu)$ have been inserted:
    \begin{itemize}
        \item The element in $\seq(1,1) \cap \seq(\nu_{\ell(\nu)})$ is unprimed and in the first diagonal position.
        \item For all $i>0$, the elements in $\seq(i+1, 1) \cap \seq(\nu_{\ell(\nu) - i})$ form a vertical column of unprimed entries starting in the first row and ending in the $(i+1)$-st diagonal position; and this vertical column is adjacent to the vertical column formed by the entries of $\seq(i, 1) \cap \seq(\nu_{\ell(\nu)- i+1} )$.
        \item The elements in $\seq(\nu_{\ell(\nu)},2) \cap \seq(\nu_{\ell(\nu)})$ form a horizontal strip of primed entries in the first row. 
        \item For all $i>0$, the elements in $\seq(\nu_{\ell(\nu) - i}, i+2) \cap \seq(\nu_{\ell(\nu) - i})$ form a horizontal strip of primed entries, starting in a column after the greatest unprimed letter, in such a manner that the $j$-th element of the horizontal strip, from smallest to greatest, is strictly below the $j$-th element of $\seq(\nu_{\ell(\nu) - i-1}, i+1) \cap \seq(\nu_{\ell(\nu) - i+1})$, or to the right of the last element in $\seq(\nu_{\ell(\nu) - i-1}, i+1) \cap \seq(\nu_{\ell(\nu) - i+1})$ if it has no $j$-th element. 
        \item Before the first element of the block $\seq(\nu_{\ell(\nu) -i -1  })$ is inserted, the horizontal strip formed by the primed entries of $\seq(\nu_{\ell(\nu) -i  })$ starts in the column after the greatest unprimed entry of that block. 
    \end{itemize}
\end{lemma}
\begin{proof}
    We prove the statement by induction on $i$. First the elements in $\seq(\nu_{\ell(\nu)})$ are inserted from greatest to smallest. As the letters in $\seq(\nu_{\ell(\nu)})$ are strictly decreasing, each letter bumps the previous one from the diagonal position in the first row, resulting in the tableau
    
    \ytableausetup{boxsize=2em}
    \[\begin{ytableau}
        1 & 2' & 3' & \ldots & \nu_{\ell(\nu)}'
    \end{ytableau}\]
    The entry $1$ is not going to be displaced by posterior entries. Also, since all the primed entries of $\seq(\nu_{\ell(\nu)})$ lie in the first row, the insertion of primed or unprimed entries from a later block will only displace them horizontally without changing their relative order. 
    
    That is, for $i=0$ the lemma follows. Suppose that the lemma holds for all $i<k$ and consider $i=k.$ 

    Throughout this insertion, we have from the inductive hypothesis that the unprimed entries from a previous block form a connected column from the first row to a diagonal position, whose entries start in $1$ and increase by $1$ until they reach the diagonal. Accordingly, no unprimed entries coming from an earlier block are displaced by unprimed entries of $\seq(\nu_{\ell(\nu) -k  })$, and any primed entries formed while inserting  $\seq(\nu_{\ell(\nu) -k  })$ are bumped out of a diagonal position that has not been used before. We also have that all primed entries before the insertion of the first letter in $\seq(\nu_{\ell(\nu) -k  })$, are to the right of all unprimed entries, because the latter are to be found in the first columns going from the first row to a diagonal position, and nowhere else by the inductive hypothesis. 

    Consequently, as the elements of  $\seq(\nu_{\ell(\nu) -k  })$ are inserted, even where they displace primed elements, they do so bumping them to the next column, that is to the right, and so in such a manner that they never bump unprimed elements from previous blocks. As primed elements created out of letters of  $\seq(\nu_{\ell(\nu) -k  })$ during its insertion are created in a new diagonal, these primed entries too are created and possibly bumped to the right of all unprimed entries from earlier blocks. 

    Hence the vertical strips from previous blocks $\seq(\nu_{\ell(\nu) -k +i })$ for $i>0$ are unaltered. Then, as the letters of $\seq(\nu_{\ell(\nu) -k  })$ are inserted in the first row, it also follows that for $j$ arbitrary either $j_{\ell(\nu)-k}$ bumps $(j+1)_{\ell(\nu)-k}$, or $j_{\ell(\nu)-k}$ bumps an instance of $(j+1)'$. But in the second case $(j+1)'$ must displace $(j+1)_{\ell(\nu)-k}$ by columns, since they were both in the first row.
    
    Similarly, for arbitrary $(j+1)_{\ell(\nu)-k}$ in an ensuing row, either it is bumped by $j_{\ell(\nu)-k}$ or $j_{\ell(\nu)-k}$ bumps an instance of $(j+1)'$; and, as there is no entry such that $(j+1)' < x < j+1$ in $(\cN \cup \cN)^*$, we obtain again that $(j+1)'$ column-bumps $(j+1)_{\ell(\nu)-k}$. Thus, the entries of $\seq(\nu_{\ell(\nu) -k  })$ are inserted displacing the previous entries, forming all the while a connected vertical strip, which reaches a new diagonal position because no entries in this strip interact with unprimed entries from previous blocks. In particular, this diagonal position corresponds to the $k+1$-st diagonal, and the vertical strip of $\seq(\nu_{\ell(\nu) -k  })$  is adjacent to the one formed by the unprimed entries of $\seq(\nu_{\ell(\nu) -k +1  })$ because each forms a connected column starting from $1$ and whose entries increase by $1$ from top to bottom, and no primed letter $x'$ can be interposed between two such columns and respect the order of $(\cN\cup \cNprime)^*$.

    Moreover, if a letter $j_{\ell(\nu) - k}$ displaces a primed entry $(j+1)_{\ell(\nu)-k+i}'$, then, by virtue of the latter bumping $(j+1)_{\ell(\nu)-k}$ to the next row and the inductive hypothesis that $(j+2)_{\ell(\nu)-k+i-1}'$ was below $(j+1)_{\ell(\nu)-k+i}'$ before being displaced, it follows that when $(j+1)_{\ell(\nu)-k}$ is inserted to the next row, it bumps $(j+2)_{\ell(\nu)-k+i-1}'$, so that it is again below $(j+1)_{\ell(\nu)-k+i}'$, albeit in a different column. 

   Fix $i>0$. By the inductive hypothesis, before the insertion of elements in $\seq(\nu_{\ell(\nu) -k  })$, the primed entries of the block $\seq(\nu_{\ell(\nu) -k + i })$ are all in $\seq(\nu_{\ell(\nu) - k +i}, k-i+2) \cap \seq(\nu_{\ell(\nu) - k+i})$, and form a horizontal strip beginning in the column after the $k-i+1$-th diagonal. But then, for every $i>0$ and column $c$ after the $k+1$-th diagonal, the primed entries of $\seq(\nu_{\ell(\nu) -k + i })$ are all bounded above by $c-1$ at that point of the insertion. Since primed entries can only be bumped towards the right, this statement still holds after letters in $\seq(\nu_{\ell(\nu) -k  })$ are inserted.  

    Consider now the primed entries created in the insertion of $\seq(\nu_{\ell(\nu) -k  })$. Every time $j_{\ell(\nu)-k}'$ is inserted to a column, we have that $(j+1)_{\ell(\nu)-k}'$ is the only instance of $(j+1)'$: indeed, all other primed entries are less than or equal to $j$ in $(\cN\cup\cNprime)^*$ as we have shown. Hence, $j_{\ell(\nu)-k}'$ must bump $(j+1)_{\ell(\nu)-k}'$ for arbitrary $j.$ Thus, the primed entries of  $\seq(\nu_{\ell(\nu) -k  })$ form a horizontal strip that may not be connected but lies nonetheless on successive columns; and they do not interact with the primed entries of an earlier block. That is, in the insertion of  $\seq(\nu_{\ell(\nu) -k  })$ all the inductive properties to be established for earlier blocks still obtain. 

    Finally, note that as the primed entries of $\seq(\nu_{\ell(\nu) -k  })$ were all bumped from the $k+1$-st diagonal and lie in successive columns, they start in the column after the last unprimed entry of $\seq(\nu_{\ell(\nu) -k  })$, while the primed entries of $\seq(\nu_{\ell(\nu) -k +1 })$, having all been displaced one column to the right by the unprimed entries of $\seq(\nu_{\ell(\nu) -k  })$ and previously in the column after the $k$-th diagonal, form now a horizontal strip starting in the $(k+2)$-nd column. That is, both horizontal strips begin in the $k+2$-nd column, and as the unprimed entries of $\seq(\nu_{\ell(\nu) -k  })$ contain exactly one more letter than the unprimed entries of $\seq(\nu_{\ell(\nu) -k +1 })$, it follows that the $j$-th element of the horizontal strip of $\seq(\nu_{\ell(\nu) -k  })$, from smallest to greatest, is strictly below the $j$-th element of $\seq(\nu_{\ell(\nu) - k-1}, k+1) \cap \seq(\nu_{\ell(\nu) - k+1})$, or to the right of the last element in $\seq(\nu_{\ell(\nu) - k-1}, k+1) \cap \seq(\nu_{\ell(\nu) - k+1})$ if it has no $j$-th element.     
\end{proof}

\begin{corollary}\label{cor:shapeLemma}
    Let $w$ be an arbitrary word in the shifted plactic class of $\hat{y}_\nu$. Then, the shape of its mixed insertion $P_\mix(w)$ is also $\nu.$ In particular, the shape of the tableau $\hat{Y}_\nu \coloneqq P_\mix(\hat{y}_\nu)$ is $\nu.$
\end{corollary}

\begin{proof}
    First, we have by definition that $w$ is equivalent in $\sPlactic$ to $\hat{y}_\nu.$ Thus $P_\mix(\hat{y}_\nu) = P_\mix(w)$. In particular, these insertion tableaux have the same shape. 
     By Lemma~\ref{lem:charScarcelyYamanouchiTableaux}, this common shape is $\nu$.  
\end{proof}

\begin{remark}
    As a consequence of Lemma~\ref{cor:shapeLemma}, there is a unique shifted class with insertion of shape $\lambda$ for $\lambda$ a shifted partition.
\end{remark}

\section{The shifted lattice condition and shrinking decomposition of a word}\label{sec:shifted_lattice}

\subsection{Shifted lattice condition}
Now we provide a combinatorial characterization for the words in $[\hat{y}_\nu]$, having as a consequence that the class of such words is closed under right factors. The resulting similarity of the combinatorial description of these words with the classical case arms us with a further reason to regard barely Yamanouchi words as the right generalization of Yamanouchi words. 

\begin{definition}\label{def:shiftedLatticeCondition}
    Let $w\in \cN^*$. We say that $w$ is \newword{shifted lattice}, if on reading the word $w$ from right to left, the number of occurrences of $i$ in $w$ is always either equal to the number of instances of $i+1$ or is greater by exactly $1$ more occurrence.
\end{definition}

Note that in a shifted lattice word $w$, the maximum value appearing in $w$ appears exactly once.

\begin{lemma}\label{lem:consecutiveIs}
    Suppose that $w\in \cN^*$ is shifted lattice. Then, it is also the case that,
    \begin{itemize}
        \item for $1<i$, there is an instance of $i-1$ between any two instances of $i$ in $w$; and 
        \item for all $i$, there is an instance of $i+1$ between any two instances of $i$. 
    \end{itemize}
\end{lemma}
\begin{proof}
    We prove the second bullet point; the proof of the first is analogous.
    Let $i$ be a letter of $w$, and suppose towards a contradiction that there are two consecutive instances of $i$ in $w$ without an $i+1$ in between. If $r_w$ is the subword obtained by reading from right to left $w$ until before the two consecutive instances of $i$, and $j(w)$ the number of instances of $j$ in $w$ for an arbitrary $j\in \cN$, then 
    \[i(r_w) \geq (i+1)(r_w) \]
    because $r_w$ satisfies the shifted lattice condition. But then, on finding the two occurrences of $i$, we see that
    \[ i(i\cdot i \cdot r_w) \geq (i+1)(i\cdot i \cdot r_w) + 2,  \]
     contradicting that $w$ is shifted lattice.
\end{proof}

\begin{corollary}\label{cor:shiftedLatticeisYamanouchi}
 If $w$ is a shifted lattice word, then $w$ is a Yamanouchi word.
\end{corollary}
\begin{proof}
    This is clear from the condition $i(r_w)\geq (i+1)(r_w)$, discussed in the proof of Lemma~\ref{lem:consecutiveIs}.
\end{proof}

\begin{lemma}\label{lem:barelyYamanouchiareShiftedLattice}
    Let $\nu$ be a strict partition and $w\sim \hat{y}_\nu$. Then,
    \begin{itemize}
        \item the word $w$ satisfies $\ct(w)=\ct(\hat{y}_\nu)$ and
        \item the word $w$ is shifted lattice.
    \end{itemize}
\end{lemma}
\begin{proof}
    The first condition follows from the fact that the shifted equivalence does not change the content of a word. 

    Observe that the word $\hat{y}_\nu$ satisfies the shifted lattice condition and that $w$ can be obtained from $\hat{y}_\nu$ by a finite sequence of shifted plactic relations. To prove the second condition, we show by cases that shifted plactic equivalence preserves the shifted lattice condition. 

    \eqref{eq:SP1}: As $b$ and $d$ are changed in this relation, the only case that could alter the condition of being shifted lattice is when $b=i$ and $d=i+1.$ Moreover, as $b\leq c <d$, this case forces $c=i.$ In particular, the word $adbc$ has the form $a \quad i+1 \quad i \quad i$ while the word $abdc$ has the form $a \quad i \quad i+1 \quad i$. 
    
    By Lemma~\ref{lem:consecutiveIs}, a word that has $a \quad i+1 \quad i \quad i$ as a subword is not shifted lattice. We must show that the same is true for $a \quad i \quad i+1 \quad i$. If $a \quad i \quad i+1 \quad i$ were shifted lattice, then by Lemma~\ref{lem:consecutiveIs} there can only be two instances of $i$ without $i-1$ in between if $i=1$. 
This in turn forces $a=1$ as $a\leq b$. But then, $ab=11$ and there are two consecutive instances of $1$ in $abdc$ without an instance of $2$ in between, violating Lemma~\ref{lem:consecutiveIs}. Hence, $a \quad i \quad i+1 \quad i$ is not shifted lattice either and the relation \eqref{eq:SP1} is never used for $b=i$ and $d=i+1$ from a word that is shifted lattice.

    \eqref{eq:SP2}: In this relation, $b$ and $d$ are exchanged, so that the interesting case happens when $b=i$ and $d=i+1$.

    Now, the inequality $b<c\leq d$ under which the relation can be applied, implies that $c=i+1$. Accordingly, $acdb= a \quad i+1 \quad i+1 \quad i$ and $acbd = a \quad i+1 \quad i \quad i+1$. By Lemma~\ref{lem:consecutiveIs}, no word containing either of these can be shifted lattice. 

    \eqref{eq:SP3}: For this relation, the only case that we need to consider is when $a=i$ and $d=i+1.$ However, since this relation can be applied solely where the inequality $a\leq b <c < d$ is satisfied, this never takes place.

    \eqref{eq:SP4}: Similarly to the previous case, here the only relevant scenario happens for $a=i$ and $d=i+1$. But the inequality $a<b\leq c < d$ for which this relation can be applied does not allow for this scenario.

    \eqref{eq:SP5}: For this relation, we only need to worry about the case where $b=i$ and $d=i+1$. In this case, considering that the inequality $a<b<c\leq d$ needs to be satisfied, we see that $c=i+1$ must hold. But then by Lemma~\ref{lem:consecutiveIs} neither $cdba = i+1 \quad i+1 \quad i \quad a$ nor $cbda= i+1 \quad i \quad i+1 \quad a$ can be part of a shifted lattice word. 

    \eqref{eq:SP6}: We can suppose that $b=i$ and $d=i+1.$ Lemma~\ref{lem:consecutiveIs}, $dbca = i+1 \quad i \quad i \quad a$ cannot be part of a shifted lattice word, while $bdca= i \quad i+1 \quad i \quad a$ can only be in a shifted lattice word if $i=1$. However, $i=1$ is impossible because the inequality $a<b\leq c <d$ needs to be satisfied. 

    \eqref{eq:SP7}: We can assume $a=i$ and $d=i+1$, but the inequality $a<b\leq c \leq d$ then forces $bcda= i+1 \quad i+1 \quad i+1 \quad u$ and $bcad = i+1 \quad i+1 \quad i \quad i+1$, neither of which is possible in a shifted lattice word by Lemma~\ref{lem:consecutiveIs}.

    \eqref{eq:SP8}: We can suppose $a=i$ and $d=i+1.$ The inequality $a\leq b < c \leq d$ further implies that $b=i$ and $c=i+1$ so that for this case, $cdab = i+1 \quad i+1 \quad i \quad i$ and $cadb= i+1 \quad i \quad i+1 \quad i$ cannot be part of a shifted lattice word by Lemma~\ref{lem:consecutiveIs}. 
\end{proof}

\begin{corollary}\label{cor:barelyYamisYam}
    If $w$ is a barely Yamanouchi word, then $w$ is a Yamanouchi word. 
\end{corollary}
\begin{proof}
    By combining Lemma~\ref{lem:barelyYamanouchiareShiftedLattice} with Corollary~\ref{cor:shiftedLatticeisYamanouchi}.
\end{proof}

\subsection{Shrinking decompositions}
Every word $w\in \cN^*$ can be decomposed canonically into a sequence of decreasing blocks. Consider the rightmost maximum element $m$ in $w$, continue with the $m-1$ to the right of $m$ closest to $m$, on obtaining its position add it to the subword and choose in the same manner $m-2$ in relation to $m-1$, and so on, until a decreasing subword of maximal length is obtained. If $w$ is Yamanouchi or barely Yamanouchi, the last element of such a subword is an instance of $1$.

Remove now the letters of $w$ that were used in the construction of the decreasing subword, and construct a second decreasing subword by starting at the new maximum element and obeying the same procedure; after a finite number of repetitions, this process ends with a list $\{ (d_{\nu_j}^{(i)})_i \}_{1\leq i \leq j}$ of decreasing sequences that partition $w$. We call this list the \newword{shrinking decomposition} of $w$. 

The shifted class of a barely Yamanouchi word (indeed, of a shifted lattice word) can be characterized by its shrinking sequences as described in Lemma~\ref{lem:combinatorialCharShiftedYamanouchi}, and results in a similar vein will be sufficient for our purposes. For a new description of the shifted Littlewood--Richardson coefficients that exploit the notion of shrinking sequences in relation to semistandard increasing decomposition tableaux we recommend \cite{Shigechi}.

Before illustrating this construction for a particular word, we adopt the convention that in the shrinking decomposition $\{ (d_{\nu_j}^{(i)})_i \}_{1\leq j \leq \ell}$ of $w$, each of the decreasing subwords $(d_{\nu_j}^{(i)})_i$ are listed in the order that they are formed. That is, $d_{\nu_j}^{(i+1)}$ is to the left of $d_{\nu_j}^{(i)}$ in $w.$ 

\begin{example}
    For the word $8 \quad 3 \quad 2 \quad 1 \quad 5 \quad 4 \quad 7 \quad 6 \quad 3 \quad 2 \quad 1 \quad 5 \quad 4 \quad 3 \quad 2 \quad 1$, the associated shrinking decomposition is readily identified as being the list $\{ d_8, d_5, d_3\}$ where $d_j= j \quad j-1 \quad j-2 \quad \ldots \quad 1$ for $j=3,5,8.$
\end{example}

 We employ the notation $i(w)$ for the number of instances of $i$ in $w$ for all $i\in \cN$; and for a letter $u$ of $w$ the notation $w_u$ to signify the smallest subword of $w$ containing $u$ such that $w = w_1 \cdot w_u$ for some $w_1\in \cN^*.$

\begin{lemma}\label{lem:removingYamanouchi}
    Let $w$ be a shifted lattice word and $\{ (d_{\nu_j}^{(i)})_i \}_{1\leq j \leq \ell}$ its shrinking decomposition. Then the subword $\hat{w}$ of $w$ whose letters do not belong to $(d_{\nu_1}^{(i)})_i$ is also shifted lattice. 
\end{lemma}
\begin{proof}
    Suppose this is not the case. Then there is a letter $u$ of $\hat{w}$ and an integer $r$ such that
    \[ r(\hat{w}_u) > (r+1)(\hat{w}_u) + 1 \quad \text{or} \quad r(\hat{w}_u) < (r+1)(\hat{w}_u).  \] 
    Without loss of generality, take $u$ rightmost with respect to this property. 

    There are three possible cases for $u$. First, $u$ could be among the letters of $\hat{w}$ that happen after $r\in (d_{\nu_1}^{(i)})_i$ in $w.$ Since $(d_{\nu_1}^{(i)})_i$ is a decreasing subword, we would then have that 
    \[ r(\hat{w}_u) = r(w_u)  \quad \text{and} \quad  (r+1)(\hat{w}_u) = (r+1)(w_u).\]
    Then, the hypothesis that $w$ is shifted lattice implies that $\hat{w}$ is also shifted lattice at $u$, a contradiction. 

    Second, $u$ could be among the letters in $\hat{w}$ that occur between the letters $r,r+1\in (d_{\nu_1}^{(i)})_i$ in $w$. In that case, 
    \[ r(\hat{w}_u) = r(w_u) - 1 \quad \text{and} \quad  (r+1)(\hat{w}_u) = (r+1)(w_u).\]
    Note that the case $r(\hat{w}_u) > (r+1)(\hat{w}_u) + 1$ is therefore not possible, and it can only be that $r(\hat{w}_u) < (r+1)(\hat{w}_u).$ Since we took $u$ rightmost with the property of not being shifted lattice, it follows that $u=r+1$. But then, between $u$ and $r+1\in (d_{\nu_1}^{(i)})_i$ there is at least one $r$ in $w$ by Lemma~\ref{lem:consecutiveIs}. Moreover, since $(d_{\nu_1}^{(i)})_i$ is the first decreasing sequence to be removed in the process of construction of the shrinking decomposition of $w$, this contradicts that $(d_{\nu_1}^{(i)})_i$ is constructed greedily so that $r\in (d_{\nu_1}^{(i)})_i$ is the first $r$ to the right of $r+1\in (d_{\nu_1}^{(i)})_i.$

    The only remaining alternative for $u$ is to be before $r+1\in (d_{\nu_1}^{(i)})_i$ in $w$ so that 
    \[ r(\hat{w}_u) = r(w_u) - 1  \quad \text{and} \quad  (r+1)(\hat{w}_u) = (r+1)(w_u) - 1.\]
    and again that $w$ is shifted lattice by assumption contradicts that $\hat{w}$ is not shifted lattice at $u.$
\end{proof}

\begin{remark}\label{remark:removingYamanouchi}
    Notice that the previous lemma can be applied iteratively to yield that if $w$ is a shifted lattice word and $\{ (d_{\nu_j}^{(i)})_i \}_{1\leq j \leq \ell}$ its shrinking decomposition, then the subword $\hat{w}$ consisting of those letters of $w$ that are not letters of $\{ (d_{\nu_j}^{(i)})_i \}_{k\leq j \leq \ell}$ is also shifted lattice for all $k$ with $1\leq k \leq \ell.$
\end{remark}

\begin{lemma}\label{lem:combinatorialCharShiftedYamanouchi}
    Let $w$ be a shifted lattice word with content $\ct(w)=\ct(\hat{y}_\nu)$ for $\nu$ a strict partition, and $\{ (d_{\nu_j}^{(i)})_i \}_{1\leq j \leq \ell(\nu)}$ its shrinking decomposition. Then the letter $m \in (d_{\nu_j}^{(i)})_i$ is to the left of the letter $m + 1 \in (d_{\nu_{j-1}}^{(i)})_i$ in $w$. In particular, the letter $m \in (d_{\nu_j}^{(i)})_i$ is also to the left of the letter $m \in (d_{\nu_{j-1}}^{(i)})_i$ in $w$.
\end{lemma}
\begin{proof}
    Suppose by contradiction that the opposite is the case. Then there exist $j$, and letters $u=m_j$ and $v_1=(m+1)_{j-1}$, such that $u$ is to the right of $v_1.$ As $(d_{\nu_j}^{(i)})_i$ is a decreasing subword whose elements decrease by one until reaching $1$, we have that there also exists a letter $v_2=m_{j-1}$ and that it is to the right of $v_1$.
    
    Remove all the letters of the word $w$ that belong to any of the decreasing subsequences $\{ (d_{\nu_k}^{(i)})_i \}_{1\leq k \leq j-2}$. By Remark~\ref{remark:removingYamanouchi}, the resulting word $\hat{w}$ is still shifted Yamanouchi. We identify the subwords in $\{ (d_{\nu_k}^{(i)})_i \}_{j-1\leq k \leq \ell(\nu)}$ of $w$, as well as their elements, with those of the shrinking decomposition of $\hat{w}$, as none of the letters therein were removed. 
    
    By Lemma~\ref{lem:consecutiveIs}, there is another letter $r_1=m+1$ between the letters $u$ and $v_2$ in $\hat{w}$. Moreover, it is to the right of $v_1.$ Using again Lemma~\ref{lem:consecutiveIs}, this implies the existence of a letter $m+2$ to the right of $(m+2)_{j-1}$; and so on, until a letter $M+1$ is obtained, to the right of the greatest letter $M_{j-1}$ defined by $M:= \max(r \in (d_{\nu_{j-1}}^{(i)})_i).$ This contradicts that the subword $(d_{\nu_{j-1}}^{(i)})_i$ is the largest decreasing subword in the shrinking decomposition of $\hat{w}$; as should have followed from the hypothesis of $\{ (d_{\nu_j}^{(i)})_i \}_{1\leq j \leq \ell(\nu)}$ being the shrinking decomposition of $w.$
\end{proof}

\begin{remark}
    In fact, Lemma~\ref{lem:combinatorialCharShiftedYamanouchi} can be upgraded to an `if and only if' statement characterizing shifted lattice words in terms of shrinking sequences. However, we omit the proof of the reverse direction as it is complicated and superfluous to our purposes.
\end{remark}

\section{Interlacing tableaux}\label{sec:interlacingTableaux}

The properties displayed by shifted lattice words in the previous results are shared by a broader family of words, and as they will be responsible for most of the insertion characteristics relevant to our purposes, we choose to prove the next theorems in greater generality than necessary for shifted lattice words; unifying the proofs for both classes. 

\begin{definition}\label{def:interlacing}
    A word $w\in \cN^{*}$ with shrinking decomposition $\left\{ \left(d_{\nu_j}^{(k)} \setminus d_{\mu_j}^{(k)} \right)_k  \right\}_{1\leq j \leq \ell}$ for some strict partitions $\mu \subseteq \nu$ is \newword{interlacing} if
    \begin{itemize}
        \item between any two letters with integer value $i>1$, there is an instance of $i-1$ and $i+1$; 
        \item between any two instances of $1$, there is an instance of $2$.
    \end{itemize}
    We say a tableau $T$ is \newword{interlacing} if it is the mixed insertion of an interlacing word.
\end{definition}

We will show in Corollary~\ref{cor:bYam_is_interlacing} that all barely Yamanouchi words are interlacing.

\begin{lemma}\label{lem:Imalemma}
    Suppose $w$ is interlacing. Then entry $k_j$ is right of $k_h$ for all $h > j$.
\end{lemma}
\begin{proof}
    Suppose not. Fix $j$ least such that $k_j$ is left of $k_{h}$ for some $k$ and some $h>j$. Choose $k$ greatest so that this occurs. 

    By the interlacing property, there is a $(k+1)_r$ between $k_j$ and $k_{j+1}$ in $w$. By the strict partition content of the shrinking decomposition $(k+1)_j$ exists. By construction, it is left of $k_j$. Hence, by maximality of $k$, we have $j > r$.

    By construction of shrinking sequences $k_r$ exists and must be strictly between $(k+1)_r$ and $k_{j+1}$. But $k_r$ left of $k_{j+1}$ contradicts the choice of $j$ least.
\end{proof}

\begin{lemma}\label{lem:redundant_condition}
    Suppose $w$ is interlacing. Then for all $i$ and $j$, the letter $i_j$ is to the left of $(i+1)_{j-1}$ if $(i+1)_{j-1}$ exists.
\end{lemma}
\begin{proof}
    Suppose $w$ were a counterexample. Then there is $(i+1)_{j-1}$ left of $i_j$ for some $i$ and $j$. By the definition of the shrinking decomposition, there is $i_{j-1}$ strictly between $(i+1)_{j-1}$ and $i_j$. However, $i_{j-1}$ left of $i_j$ contradicts Lemma~\ref{lem:Imalemma}.
\end{proof}

\subsection{Rows of interlacing tableaux}

\begin{lemma}\label{lem:rowandColBoundedLemma}
    Let $T$ be a semistandard shifted Young tableau. Then, for all $i$ and $j$ the letter $i_j$ belongs to a row $r\leq i$, whereas if primed, the letter $i_j'$ lies in a row $f<i.$
\end{lemma}
\begin{proof}
Let $\sf{d_1}$ and $\sf{d_2}$ denote diagonal positions of different rows in $T$, and $x$ and $y$ two entries located in $\sf{d_1}$ and $\sf{d_2}$, respectively. Without loss of generality, let $\sf{d_1}$ be North of $\sf{d_2}$. The order in rows and columns imposed on a semistandard Young tableau forces $x<y$, and the fact that they occupy diagonal cells implies that $x,y\in \cN.$ Thus, the minimum possible entry in row $i$ under the ordering of $\cN \cup \cNprime$ is an unprimed instance of $i$.

Consequently, as $i'<i$, a primed instance of $i$ may only be placed in a row $f< i. $ \qedhere
\end{proof}

We say that a word $w$ has \newword{strict content} if $\ct(w)=\lambda$ for a strict partition $\lambda$.

\begin{lemma}\label{lem:orderSameSeq}
    Let $w$ be an interlacing word with strict content. If, in the mixed insertion process of $w$, two unprimed letters with the same value $q$ never appear in the same row $r$ for any $r<q$, then for each row $r$ the letter $i_j$ is inserted into row $r$ before $(i-1)_j$ is.

    Even more, for any $m$, if two unprimed instances of $q$ never appear in the same row $r$ for any $r < m$, then for each row $r\leq m$ the letter $i_j$ is inserted into row $r$ before $(i-1)_j$ is. 
\end{lemma}
\begin{proof}
    The first statement is a special case of the second since we may assume $m \leq q$. Hence we prove the second.
    
    We induct on $r$. The base case $r=1$ is a consequence of the way that shrinking sequences are defined. Fix $m$. Towards a contradiction, let $r=k<m \leq q$ be the smallest indexed row in which the lemma fails in the order of insertion, and without loss of generality, further let $i$ be the smallest integer with $i_j$ being inserted after $(i-1)_j$ to row $k$ and let $j$ be maximal in relation to $i.$
    
    Since in row $k-1$ the property held, $i_j$ was its member before $(i-1)_j$. But we then have that $i_j$ was the only unprimed entry of value $i$ in row $k-1$; and when $(i-1)_j$ was inserted, either it bumped $i_j$, or there was an entry $i'$ before $i_j$ in row $k-1$. In the former situation, $i_j$ was either bumped from a diagonal position, in which case $(i-1)_j$ now occupies this position and will not be inserted to posterior rows; or $i_j$ was bumped to the next row, so that in both cases the lemma holds. In the latter case, $i'$ is displaced to the next column by $(i-1)_j$, and there it bumps $i_j$ because there are no entries $x$ such that $i'<x < i.$ Irrespective of the cause, $i_j$ is then again bumped before $(i-1)_j$, a contradiction.
\end{proof}

\begin{lemma}\label{lem:precedingEntry}
    Let $w$ be an interlacing word with strict content. If, in the mixed insertion process of $w$, two unprimed letters with the same value $q$ never belong to the same row $r$ for any $r<q$, then for each $i$ and each row $r\leq i-1$ the letter $(i-1)_{j}$ is inserted to row $r$ before $i_{j-1}$ is, or $(i-1)_j$ has been bumped from a diagonal position and is a primed entry. 

    Even more, if $m$ is chosen arbitrarily and for all $i>1$ and all $r < \min(i, m)$ the entry $i_j$ never appears in the same row $r$ as another unprimed instance of $i$, then $(i-1)_j$ is inserted before $i_{j-1}$ to all rows $s \leq \min(i-1, m)$, or $(i-1)_j$ is already primed when $i_{j-1}$ is inserted to $s$.
\end{lemma}
\begin{proof}
Again, the first statement follows from the second, since we may assume $m \leq q$. We prove the second statement by induction on $r$.

    The base case $r=1$ follows from Lemma~\ref{lem:redundant_condition}. Suppose towards a contradiction that the lemma holds for all rows $r<k$, but that in row $k\leq i-1$ the letter $i_{j-1}$ is inserted before $(i-1)_{j}$.

    Then $(i-1)_j$ was inserted to row $k-1$ before $i_{j-1}$, and when the latter was displaced to row $k$, the letter $(i-1)_j$ remained in row $k-1$. If $i_{j-1}$ was row-bumped to row $k$, it must have been bumped by an instance of $i-1$, as otherwise $(i-1)_j$ would have been bumped as well. However, in that case two instances of $i-1$ must have belonged to the same row, a contradiction. 
    
    Therefore, $i_{j-1}$ must be bumped to the next row by a primed entry. There are two cases according to whether there is an entry between $(i-1)_j$ and $i_{j-1}$ in row $k-1$. Note that if such as entry exists, it must be an $i'$.
    
    If there were no entries between $(i-1)_j$ and $i_{j-1}$ in row $k-1$, then $i_{j-1}$ is bumped by an instance of $i'$, coming from the previous column, to wit, in the position immediately below $(i-1)_j$. But then it is impossible for $i'$ to have been bumped, by rows or columns, because there is no entry $x$ with $(i-1)_j < x <i'$. 
    
    On the other hand, if there was an $i'$ between $(i-1)_j$ and $i_{j-1}$ in row $k-1$, we have that either $i'$ bumped $i_{j-1}$ or a different entry did so. In the former situation, $i'$ had to be row-bumped by an instance of $i-1$, which contradicts again that there are no repeated entries in row $k$. In the latter situation, a different entry $x$ in the same column as $i'$ had to bump $i_{j-1}$, taking its position to the right of $i'.$ Accordingly, we must have that $i'\leq x < i$. However, since primed entries of the same integer value cannot belong to the same row, this also is impossible.
\end{proof}

\begin{lemma}\label{lem:rowOrderForSameElementDiffSeq}
    Fix $m>0$. Let $w$ be an interlacing word with strict content, and suppose that during the mixed insertion of $T=P_\mix(w)$ no two letters of the same value $q$ are repeated in a row $f<\min(m,q)$, and the letter $i_j$ is inserted to row $r \leq \min(m,i)$. Then, 
    \begin{enumerate}
        \item for all $k>0$, either $i_{j+k}$ was primed when $i_j$ was inserted to row $r$, or $i_{j+k}$ is inserted to row $r$ before $i_j;$ and
        \item if $i_j$ is primed, then $i_{j+k}$ is already primed at that point.
    \end{enumerate} 
\end{lemma}
\begin{proof}
    Suppose that $i_{j+k}$ was not primed when $i_{j+k-1}$ is inserted to $r<i.$ Then, by Lemma~\ref{lem:orderSameSeq} $(i+1)_{j+k-1}$ is inserted to $r$ before $i_{j+k-1}$, so that nor was $i_{j+k}$ primed when $(i+1)_{j+k-1}$ was inserted to $r$. Hence, from Lemma~\ref{lem:precedingEntry}  the letter $i_{j+k}$ was likewise inserted to $r$ before $(i+1)_{j+k-1}$ was. Thus, $i_{j+k}$ is inserted to $r$ before $i_{j+k-1}$ is. 

    If we now let $r=i$, and give thought to the order of insertion in $r-1$, it follows from our considerations that if repeated, $i_{j+k}$ is left of $i_{j+k-1}$ in $r-1$, while if not $i_{j+k}$ is displaced to $r$ before $i_{j+k-1}$ as desired. In the first case, 
$i_{j+k-1}$  cannot be displaced, row or column-wise without causing the bumping of $i_{j+k}$. But then $i_{j+k}$ is also bumped to $r$ before $i_{j+k-1}$. In either case the desired order holds for $r=i$ too, and this exhausts all cases by Lemma~\ref{lem:rowandColBoundedLemma}.
    
    The same argument shows that for all $e>0$, if $i_{j+e}$ was not primed when $i_{j+e-1}$ was inserted to row $r$, then $i_{j+e}$ is inserted to row $r$ before $i_{j+e-1}.$ Moreover, for all $e>0$, the letter $(i+1)_{j+e}$ is primed before $i_{j+e}$ because of Lemma~\ref{lem:orderSameSeq}, and the letter $i_{j+e+1}$ is primed before $(i+1)_{j+e}$ because of Lemma~\ref{lem:precedingEntry} and the fact that $i_{j+e+1}$ does not become primed in row $i$. This last fact can be seen as a consequence of there being no entries $x<i$ in row $i$ from Lemma~\ref{lem:rowandColBoundedLemma}.  Thus, we obtain that if $i_{j+k}$ was unprimed when $i_j$ was inserted, then so were $i_{j+e}$ for $1\leq e \leq k-1$. Then, iterating for $e$ over $1\leq e \leq k-1$, it follows that the order of insertion
    \[i_{j+k} \rightarrow  (i+1)_{j+k-1}  \rightarrow  i_{j+k-1}  \rightarrow  (i+1)_{j+k-2} \rightarrow i_{j+k-1} \rightarrow \cdots \rightarrow  i_j  \]
    is observed throughout the insertion process in row $r.$ In particular, $i_{j+k}$ is inserted before $i_j$ is. This proves (1). Noting that this order also corresponds to the order in which entries are primed, this establishes $(2)$ as well. 
\end{proof}

\begin{lemma}\label{lem:repeatingLemma}
    Let $w$ be an interlacing word with strict content. Then in the mixed insertion process of $w$, for each $i$ and for each row $r < i$, two distinct letters $i_j$ and $i_k$ never both appear in the row $r$.
\end{lemma}
\begin{proof}
    Fix $i$ and, without loss of generality, assume $k < j$.
    Suppose by contradiction that the lemma fails and without loss of generality, that it fails for the first time after the last letter of $w$ is inserted. Let the failure occur in row $r <i.$
    
    If $r=1$, then the last letter of $w$ is $i_k$, and then we have that between $i_j$ and $i_k$ in $w$ there is an $i-1$ because $i>1$ and $w$ is interlacing. Since $i_k$ is the first letter to repeat an integer value in a row, this hence implies that on being inserted, $i-1$ row-bumps $i_j$. Otherwise, there would be an $i'$ before $i_j$, in which case $i-1$ row-bumps $i'$, and $i'$ displaces $i_j$ because they are both in row $1$. In either case, we arrive at a contradiction.
    Thus we may assume that the property is not broken in row $1$. Inductively suppose that the lemma also holds for all rows $r<m$ where $m<i.$ 

    Suppose for the sake of contradiction that the property is broken for the first time on row $m.$ We can assume that $(i-1)_j$ is not a primed entry when $i_j$ and $i_{j-1}$ are repeated for the first time, because as $i_j$ is inserted before $(i-1)_j$ to every row by Lemma~\ref{lem:orderSameSeq}, it must also become primed before. 
    
    Hence, in row $m$, $i_j$ had to be inserted before $(i-1)_j$, which in turn is inserted before $i_{j-1}$ by Lemma~\ref{lem:orderSameSeq}, Lemma~\ref{lem:precedingEntry}, and Lemma~\ref{lem:rowOrderForSameElementDiffSeq}. That is, if two entries $i_j$ and $i_k$ are  repeated in that row for the first time, then $i_j$ is inserted before $i_k$ and $k=j-1$. 

    Moreover, if they are to be repeated for the first time in row $m$, then the facts that $i_j$ is inserted before $(i-1)_j$ and that $(i-1)_j$ is inserted before $i_{j-1}$ together imply that $(i-1)_j$ is inserted to row $m$ displacing $i_j$, a contradiction.
    \end{proof}

\begin{lemma}\label{lem:bigRowLemma}
    Let $w$ be an interlacing word with strict content. Then,
    \begin{enumerate}
        \item  in the mixed insertion process of $w$, for all $i$ and all rows $r < i$, two distinct letters $i_j$ and $i_k$ never occupy the same row $r$;  
        \item for all $i$ and all $r\leq i-1$, the letter $(i-1)_{j}$ is inserted to row $r$ before $i_{j-1}$ is, or is already a primed entry; and
        \item for all $r$, the letter $i_j$ is inserted to row $r$ before $(i-1)_j$ is, or $i_j$ is primed.
    \end{enumerate}
\end{lemma}
\begin{proof}
   The first item is the content of Lemma~\ref{lem:repeatingLemma}. Now, since from there we have that for all rows $r$, two letters $i_j$ and $i_k$ never occupy the same row as long as $r<i$. The second and third items then follow from Lemmas~\ref{lem:repeatingLemma} and~\ref{lem:orderSameSeq}, respectively. 
\end{proof}

\subsection{Columns of interlacing tableaux}

Now we provide the key properties satisfied by interlacing words as regards column insertion, morally, that the order of insertion is also transferred to the primed entries.

\begin{lemma}
    Let $w$ be an interlacing word with strict content. Then, for every column $c$, and all $i$ and $j$, we have that if $i_{j-1}'$ and $i_j'$ are inserted to $c$, then $i_j'$ is inserted before $i_{j-1}'.$ 
\end{lemma}
\begin{proof}
    For the sake of convenience we work with $(i+1)_j'$ and $(i+1)_{j-1}'$ instead of $i_{j-1}'$ and $i_{j}'$. 
    
    Suppose that both $(i+1)_j$ and $(i+1)_{j-1}$ are eventually primed in the insertion of $w$.  By Lemma~\ref{lem:rowOrderForSameElementDiffSeq} and Lemma~\ref{lem:rowandColBoundedLemma}, while unprimed, $(i+1)_j$ is inserted to every row before $(i+1)_{j-1}$. Accordingly, $(i+1)_j$ also becomes primed before $(i+1)_{j-1}$ and there are three cases for the relative positions of $(i+1)_j'$ and $(i+1)_{j-1}$.
    
    {\sf Case 0 (On becoming primed, $(i+1)_j$ is inserted to the column of $(i+1)_{j-1}$
    ):}
    In that case, $(i+1)'_j$ displaces $(i+1)_{j-1}$ so that $(i+1)_{j-1}$ ends in a position weakly west of $(i+1)_j'$.

    {\sf Case 1 (The letter $(i+1)_{j}'$ is NorthWest of $(i+1)_{j-1}$ ):}
    This scenario is impossible by the order imposed on the rows and columns of the tableau. Note that as a result of this, not only is such a situation absurd when $(i+1)_j$ becomes primed, but also for further steps of the insertion as long as $(i+1)_j$ is primed and $(i+1)_{j-1}$ unprimed.

    {\sf Case 2 (The letter $(i+1)_{j}'$ is weakly south of $(i+1)_{j-1}$ ):}
    For the letter $(i+1)_{j-1}$ to become primed, it must eventually occupy a diagonal position $\sf{d}$. Moreover, in its trajectory to $\sf{d}$, either $(i+1)_{j-1}$ passes by a row that $(i+1)_j'$ belongs to or it does not.

    {\sf Case 2.1 (The letter $(i+1)_{j-1}$ is inserted to a row $r$ containing $(i+1)_j'$ before reaching $\sf{d}$):}
    Then $(i+1)_{j-1}$ is inserted to $r$ so that it is adjacent to $(i+1)_j'$. Note that for this to take place we must have that $\sf{d}$ is not in $r$, i.e., that $(i+1)_{j-1}$ reacher $\sf{d}$ in a later row. Thus, for $(i+1)_{j-1}$ to continue in its bumping path it is displaced from $r.$ Furthermore, whether it is row-bumped or column-bumped, it is inevitable that $(i+1)_j'$ is bumped too as a result. Hence, $(i+1)_{j-1}$ is West of $(i+1)_{j}'$ and remains so after being displaced from row $r.$

    {\sf Case 2.2 (All of the bumping trajectory of $(i+1)_{j-1}$ from the moment $(i+1)_{j}$ becomes primed, until $(i+1)_{j-1}$ arrives to $\sf{d}$ is North of $(i+1)_j'$ ):}
    In particular, it then follows that $\sf{d}$ is North of $(i+1)_j'$. But then, $(i+1)_{j-1}$ is NorthWest of $(i+1)_j'$ when in $\sf{d}$, a contradiction by order considerations.

    Accordingly, in all cases $(i+1)_{j-1}$ is weakly west of $(i+1)_j'$. If West of $(i+1)_j'$, then $(i+1)_{j-1}$ is primed weakly west of $(i+1)_j'.$ Let us then consider the case where $(i+1)_{j-1}$ becomes primed in the same column $c$ as $(i+1)_j'$. For $(i+1)_{j-1}$ to be bumped from $\sf{d}$, it needs to be row-bumped, and as there are no entries $(i+1)'<x<i+1$, it follows that for that to happen all instances of $(i+1)'$, with $(i+1)_j'$ included, must first be displaced to the next column. But then, $(i+1)_{j-1}'$ is again column-inserted for the first time weakly left of $(i+1)_j'.$

    With this order of insertion in place, it must also hold that if $(i+1)_j'$ and $(i+1)_{j-1}'$ appear in the same column $c$, and an entry $n'$ with $n<i+1$ is inserted to $c$, then $(i+1)_{j-1}'$ is not the element displaced to the next column. Similarly,  $(i+1)_{j-1}'$ cannot be row-bumped, while in the same column as $(i+1)_j'$ because it would have to be bumped by some $e < (i+1)'$ which would then create the invalidly ordered column configuration
     \ytableausetup{boxsize=3.5em}
    $\scalebox{0.7}{\ytableaushort{{(i+1)_j'}, e}}$.

    In either case, $(i+1)_{j-1}'$ cannot be bumped to the next column until after $(i+1)_j'$ has been; in other words, $(i+1)_{j-1}'$ is inserted to every column before $(i+1)_j'$ as desired.
\end{proof}

\begin{lemma}\label{lem:GodzillaLemma}
    Let $w$ be an interlacing word with strict content. Suppose that throughout the mixed insertion of $w$, no two primed entries of the same value occur in the same column at the same time.  Then, for all $i$ and $j$, $i_j$ becomes primed before $(i+1)_{j-1}$. Moreover, if the latter becomes primed, it does so in a column weakly to the left of $i_j'.$
\end{lemma}
\begin{proof}
    Let $f$ be the row in which $i_j$ becomes primed or, if $i_j$ is never primed, let $f$ be the empty row immediately below the tableau. From Lemma~\ref{lem:precedingEntry} we have that $i_j$ is inserted to every row $r\leq f$ before $(i+1)_{j-1}$ unless $r\geq i$. Furthermore, Lemma~\ref{lem:rowandColBoundedLemma} asserts that, for all $i$, $i_j$ is never inserted below row $i$. Therefore, $i_j$ never becomes primed in row $i$, since if that were to happen $i_j$ would need to be row-bumped by an unprimed entry $x<i$, which is impossible, as $x$ is by Lemma~\ref{lem:rowandColBoundedLemma} never inserted below row $x.$ Hence, if $i_j$ is to become primed, it must do so in a row $f\leq i-1$, and moreover $i_j$ must be inserted to row $f$ before $(i+1)_{j-1}$ is.

    \medskip
    \noindent
    {\sf Case 1 ($(i+1)_{j-1}$ is inserted to row $f$ before $i_j$ becomes primed):}    

    There are two subcases: either $i_j$ and $(i+1)_{j-1}$ are adjacent immediately after insertion of $(i+1)_{j-1}$ to row $f$, or else there is an instance of $(i+1)'$ in between. 

    \medskip
    \noindent
    {\sf Case 1.1 ($i_j$ and $(i+1)_{j-1}$ are adjacent):}  
    For $i_j$ to become primed, it must be bumped from a diagonal position and either displaces $(i+1)_{j-1}$ or an instance of $(i+1)'$ in the same column as $(i+1)_{j-1}.$ If the former, then $(i+1)_{j-1}$ is bumped to the diagonal position immediately below $i_j'$. Suppose towards a contradiction that $i_j'$ remains in the same column while $(i+1)_{j-1}$ is displaced to the next column acquiring a prime. As $(i+1)_{j-1}$ can only be bumped from its diagonal position by rows, it must be so bumped by an entry $i'<x<i+1$. That is, we have that $x=i$. Moreover, this $x$ must be on the same row as $i_j'$ before occupying the diagonal position below it. If $x$ is displaced by rows to the next row, it has to be bumped by an instance of $y\in \cN$ where $y<x$, but then also $y<i'_j$ holds and $i_j'$ instead of $x$ is displaced. If on the other hand $x$ is displaced by columns by an entry $y'$, then we must have $i'<y'<i$ in $\cN\cup \cNprime$. But this is impossible, so we reach a contradiction and $i_j'$ is inserted to the next column before $(i+1)_{j-1}'$ is.    
    
    If instead $i_j$ is bumped from a diagonal position and displaces an instance of $(i+1)'$ in the same column as $(i+1)_{j-1}$, then regardless of whether $(i+1)_{j-1}$ is row-bumped or column-bumped later in the insertion process, and it must be if it is to become primed, it will reach the diagonal position immediately below it. Consider the configuration before $(i+1)_{j-1}$ becomes a primed entry (if it does not become a primed entry the statement is vacuously true).

      \ytableausetup{boxsize=4.3em}
    \[\scalebox{0.7}{\begin{ytableau}
        ~     & ~     & i_j' \\
        \none & ~     & x    \\
        \none & \none & (i+1)_{j-1} 
    \end{ytableau}}\]

    To become primed, $(i+1)_{j-1}$ is bumped by rows by an entry $y<i+1$. Hence if $(i_j)'$ is not displaced by columns, we have that $i_j' < x < y < (i+1)_{j-1}$ in $\cN \cup \cNprime$ where the first and second inequalities are strict because there are no primed repeated entries in the same column by hypothesis. This is a contradiction, and so $(i_j)'$ must be displaced to the column after $(i+1)_{j-1}$ before $(i+1)_{j-1}$ becomes primed. The desired result follows. 

    \medskip
    \noindent
    {\sf Case 1.2 (there is an instance of $(i+1)'$ between $i_j$ and $(i+1)_{j-1}$):} 
    Let $i_j$ appear in box $\bbb$, so that $(i+1)_{j-1}$ appears in box $\bbb^{\rightarrow \rightarrow}$ while there is an instance of $(i+1)' \in \bbb^\rightarrow$. Note that $\bbb$ is a diagonal box.
    
    When $i_j$ becomes primed it column-bumps $(i+1)'$ because of the hypothesis that no two primed entries are repeated in the same column. The letter $(i+1)'$ then displaces $(i+1)_{j-1}$ to the diagonal position $\bbb^{\rightarrow \downarrow}$, and this situation has already been analyzed in {\sf Case 1.1}. 

     \medskip
    \noindent
    {\sf Case 2 ($i_j$ becomes primed before $(i+1)_{j-1}$ is inserted to row $f$):} 
      In this case, note that $(i+1)_{j-1}$ must be inserted to row $f$ if it is to be primed later, or else the lemma statement is vacuously true.

    Just after $i_{j}$ becomes primed, we have that $i_j'$ is weakly southwest of $(i+1)_{j-1}$, or $i_j'$ is NorthWest of $(i+1)_{j-1}$, or $(i+1)_{j-1}$ has not been inserted yet to the tableau. To prove this, note that $i_j'$ cannot be East of $(i+1)_{j-1}$ because before being bumped from a diagonal position, the letter $i_j$ is west of $(i+1)_{j-1}$; this is the only way for the increasing order in rows and columns of the tableau not to contradict the order between both letters. 
    Hence, taking into account that in being displaced from its diagonal position $i_j$ is inserted into the next column, we have that either $i_j'$ bumps $(i+1)_{j-1}$ or it displaces a different entry, arriving at a configuration for which $(i+1)_{j-1}$ is still weakly east of $i_j'$. In the former case, $(i+1)_{j-1}$ must occupy the position immediately below $i_j'$, while in the latter it is weakly east of $i_j'.$ Both cases yield that $i_j'$ is weakly west of $(i+1)_{j-1}$.

    \medskip
    \noindent
    {\sf Case 2.0 ($i_j'$ is in the same column as $(i+1)_{j-1}$):} 
    Suppose that after $i_j$ is primed, $(i+1)_{j-1}$ lies in the same column $c$ as $i_j'.$ It must then remain there until bumped from the diagonal box $\bbb$ of column $c$. Say $\bbb$ is in row $r$. We then either have one of four configurations, which we proceed to analyze in turn.

First, suppose we have the local configuration
     \ytableausetup{boxsize=4.3em}
    \[\scalebox{0.7}{\begin{ytableau}
        \ldots  & ~     &    \vdots \\
        \none   & \ldots & i_j'  \\
        \none   & \none & (i+1)_{j-1}     
    \end{ytableau}}\]
    with $i_j' \in \bbb^\uparrow$.
Here, $(i+1)_{j-1}$ is only bumped from $\bbb$ without displacing $i_j'$ if $(i+1)_{j-1}$ is displaced by an instance of $i.$ For that to happen, that instance of $i$ has to be bumped from the row $r-1$ to row $r$. But whenever we have \ytableausetup{boxsize=2.0em,centertableaux} $\ytableaushort{{i_j'}{i}}$, $i$ cannot be bumped, column- or row-wise, to the next row without displacing $i_j'.$ Thus in this case the lemma follows.
\ytableausetup{nocentertableaux}

        Secondly, suppose we have the configuration
        \ytableausetup{boxsize=4.3em}
        \[\scalebox{0.7}{\begin{ytableau}
            \ldots  & ~     & ~ &    \vdots \\
            \none & \ldots  & ~ &    i_j'   \\
            \none & \none   & \ldots & i  \\
            \none & \none   & \none & (i+1)_{j-1}     
        \end{ytableau}}\]
        with $i_j' \in \bbb^{\uparrow \uparrow}$. Here, for $(i+1)_{j-1}$ to be displaced from $\bbb$, we have that $i$ also has to be bumped from its current cell $\bbb^\uparrow$. But $i$ is only bumped from $\bbb^\uparrow$ by an instance of $i'$, if $i_j'$ is to remain unmoved. This scenario is ruled out by the hypothesis that no two primed entries with the same value occur in the same column. Hence, in this case the lemma statement also holds. 

        Thirdly, suppose we have the configuration 
        \ytableausetup{boxsize=4.3em}
        \[\scalebox{0.7}{\begin{ytableau}
            \ldots  & ~     & ~ &    \vdots \\
            \none & \ldots  & ~ &    i_j'   \\
            \none & \none   & \ldots & (i+1)'  \\
            \none & \none   & \none & (i+1)_{j-1}     
        \end{ytableau}},\]
        again with $i_j' \in \bbb^{\uparrow \uparrow}$.
        Here, if $(i+1)_{j-1}$ is to be bumped from $\bbb$, the instance of $(i+1)' \in \bbb^\uparrow$ must be bumped first because of the strict ordering of the columns. For this to happen without displacing $i_j'$ as well, the $(i+1)' \in \bbb^\uparrow$ must be bumped by an instance of $i$. Then we are done by the case of the second local configuration above, where $i$ is the only entry between $i_j'$ and $(i+1)_{j-1}$. 

        Fourthly and finally, suppose we have the configuration   
        \ytableausetup{boxsize=4.3em}
        \[\scalebox{0.7}{\begin{ytableau}
            \ldots & ~     & ~     & ~ &   \vdots  \\
            \none &\ldots  & ~     & ~ &    i_j' \\
            \none &\none & \ldots  & ~ &    i   \\
            \none &\none & \none   & \ldots & (i+1)'  \\
            \none &\none & \none   & \none & (i+1)_{j-1}     
        \end{ytableau}},\]
        with $i_j' \in \bbb^{\uparrow\uparrow\uparrow}$.
        Here, for $(i+1)_{j-1}$ to be displaced from $\bbb$, the $(i+1)' \in \bbb^\uparrow$ has to be displaced by an entry $x< (i+1)'$. But then $i_j'$ would also be column-displaced.

     \medskip
    \noindent
    {\sf Case 2.1 ($i_j'$ is NorthWest of $(i+1)_{j-1}$):} 

     \ytableausetup{boxsize=4.3em}
    \[\scalebox{0.7}{\begin{ytableau}
        \ldots  & ~     &    ~   &   ~     & i_j' & ~ \\
        \none   & \ldots     & ~ &    ~    & \vdots      & ~      \\
        \none   & \none & \ldots &      ~     & \ldots & ( i+1)_{j-1} \\
        \none   & \none & \none  & \bbb     & ~      &   ~
    \end{ytableau}}.\]

    As $i_j$ was inserted to all rows before row $f$, it follows that $(i+1)_{j-1}$ cannot become primed in a row $r\leq f-1.$
     
    Also, if on being column inserted $i_j'$ had bumped $(i+1)_{j-1}$ then $i_j'$ would be southwest of it, a contradiction. Then, since $i_j$ was just displaced from the diagonal box $\bbb$ in row $f$, it follows that $(i+1)_{j-1}$ is never row-displaced to $\bbb$. That is, if $(i+1)_{j-1}$, is to become primed, it must do so in a row $r$ with $r \geq f+1$. 
    
    Note also that by assumption both $i_j'$ is NorthWest of $(i+1)_{j-1}$, and $i_j$ has just become primed so that $(i+1)_{j-1}$ is at most in row $f-1$. Hence, the row index of $i_j'$ is at most $f-2$. Accordingly, either $(i+1)_{j-1}$ is at a later point of the insertion process in the same column as $i_j'$ and we are done by {\sf Case 2.0}, or else we eventually have the configuration
    \ytableausetup{boxsize=4.3em}
    \[\scalebox{0.7}{\begin{ytableau}
        \ldots  & ~     &    ~   &   i_j'   & ~         & ~ \\
        \none   & \ldots & ~     &    i     & ~ & ~      \\
        \none   & \none & \bbb      &   (i+1)' & (i+1)_{j-1} & ~ \\
        \none   & \none & \none  &   (i+1)      & ~      &   ~
    \end{ytableau}}.\]

    However, note that in this case either $(i+1)_{j-1}$ never becomes primed or it is bumped to the next row, where, taking into account that there cannot be repeated unprimed entries, $(i+1)$ must be displaced prior to the insertion of $(i+1)_{j-1}$. This is absurd unless $i_j'$ is bumped by columns first. Furthermore, as we showed that $(i+1)_{j-1}$ can never be strictly southwest of $i_j'$, this means that after the column displacement of $i_j'$, the letter $(i+1)_{j-1}$ is in the same column as $i_j'$, as so again we are done by {\sf Case 2.0}.     

 \medskip
    \noindent
    {\sf Case 2.3 ($(i+1)_{j-1}$ has not been inserted into the tableau):}
    If the letter $(i+1)_{j-1}$ is not yet in the tableau, it follows that for $(i+1)_{j-1}$ to become primed, it must be row-bumped from the first row, and eventually be inserted to the same row as $i_j'$. Indeed, as all the elements $y$ above $i_j'$ in its column satisfy $y<i'$, they also satisfy $y<i+1$, and then $(i+1)_{j-1}$ is always inserted to the right of $y$ in its row, until it is bumped to the row of $i_j'.$

    Once in that row, either $(i+1)_{j-1}$ is row bumped to a position SouthWest of $i_j'$ (contradicting the analysis of {\sf Case 2}), or $(i+1)_{j-1}$ is bumped to the position immediately below $i_j'$ as considered in {\sf{Case 2.0}}, or we have a configuration of the form
    \ytableausetup{boxsize=4.3em}
        \[\scalebox{0.7}{\begin{ytableau}
            \ldots  & ~     & ~ & ~ &  \vdots \\
            \none & \ldots & ~ & i_j' &  (i+1)_{j-1} \\
            \none & \none  & \ldots & i & ~   
        \end{ytableau}}.\]
Here, when $(i+1)_{j-1}$ is inserted to the next row, $i$ is not displaced in the process. But then we are done by {\sf Case 2.1} 
     Note that the existence of an entry $(i+1)'$ between $i_j'$ and $(i+1)_{j-1}$ would not affect this argument. If there had been an instance of $(i+1)$ in place of $i$ below $i_j'$ the condition of no repetition in the rows would guarantee that $(i+1)_{j-1}$ is again bumped to the position immediately below $i_j'.$

        \medskip
    \noindent
    {\sf Case 2.2 ($i_j'$ is SouthWest of $(i+1)_{j-1}$):} 
    Since $(i+1)_{j-1}$ must be bumped until it becomes primed, lest it never does and then the lemma holds trivially, it must be bumped from a diagonal position $r>f$. Indeed, as $i_j$ was bumped from the diagonal in row $f$, all the diagonal entries from rows $d\leq f$ are occupied by entries strictly smaller than $i$ in $\cN\cup\cNprime$.

    But then, either $(i+1)_{j-1}$ eventually lands in the same column as $i_j'$ as analysed in {\sf{Case 2.0}}; or else $i_j'$ is so bumped, possibly more than once, so that it is now NorthWest of $(i+1)_{j-1}$ as seen in {\sf{Case 2.1}}.
\end{proof}

\begin{lemma}\label{lem:smallSeqWestofBigSeqPrimed}
     Let $w$ be an interlacing word with strict content. Suppose that throughout the mixed insertion of $w$, no two primed entries of the same value occur in the same column at the same time. Then, $(i+1)_{j-1}'$ is always weakly west of $i_j'$ if both exist.
\end{lemma}
\begin{proof}
    From Lemma~\ref{lem:GodzillaLemma}, we have that when first primed, the letter $(i+1)_{j-1}'$ is weakly west of $i_j'$. For the sake of contradiction, suppose that $(i+1)_{j-1}'$ and $i_j'$ are both in column $c$, but that the letter $(i+1)_{j-1}'$ is bumped out while $i_j'$ is not. Then, $(i+1)_{j-1}'$ is either bumped by columns or by rows. 

    If $(i+1)_{j-1}'$ is bumped by columns, then the strictness of the order in $c$ ensures that $(i+1)_{j-1}'$ is displaced by an instance of $i'.$ This would result in two instances of $i'$ appearing in row $c$, a contradiction.
    
    Otherwise $(i+1)_{j-1}'$ is bumped by rows. Say that $(i+1)_{j-1}'$ is in row $r$. Since $(i+1)_{j-1}'$ and $i_j'$ are both in column $c$, we have $r \geq 2$. Order considerations force $i$ to be the bumping entry. Moreover, if $i$ is to displace $(i+1)_{j-1}'$, this $i$ must have been bumped in turn from row $r-1$. Hence, before this bumping, row $i-1$ contains both an $i$ and an $i'$ (namely, $i_j'$). But in such a row, the only way that the $i$ can be bumped is if it is bumped by that $i'$.  So, if $(i+1)_{j-1}'$ is bumped by rows, then $i_j'$ is also bumped into column $c+1$.
\end{proof}

\begin{lemma}\label{lem:sameEntryPrimedDiffSeq}
    Let $w$ be an interlacing word with strict content and suppose that $i_j'$ is inserted to column $c$ at some point during the mixed insertion of $w$. Then, for all $k>j$, the letter $i_k'$ is also inserted to $c$ and before $i_j'$ is.
\end{lemma}
\begin{proof}
    From Lemma~\ref{lem:rowOrderForSameElementDiffSeq}, if $i_j$ becomes primed, then $i_k$ also becomes primed and at an earlier time. We will first show that $i_j$ becomes primed weakly west of $i_k'$.
    
    Let $e$ denote the column of $i_k'$ when $i_j$ is placed in a diagonal position ${\sf d}$ in column $d$. If $d > e$, we have a configuration of the form
    \ytableausetup{boxsize=2.5em}
        \[\begin{ytableau}
            \ldots  & ~     & \vdots &   ~ & ~ \\
            \none & \ldots & i_k' & \ldots      & x  \\
            \none & \none  & \vdots & ~ & \vdots \\
            \none & \none  & \none  & ~ & \vdots \\
            \none & \none  & \none  & \none&  i_j  
        \end{ytableau}.\]
    and because of the inequalities $i_k' < x < i_j$, we arrive at a contradiction. 

    Hence $d \leq e$. It suffices to consider the case where $d=e$, as otherwise the claim is immediate. 

    For this scenario, once $i_j$ is displaced by an entry $x<i$ to become primed, either $i_k'$ was bumped to column $d+1$ prior to the insertion of $x$, or else $i_k'<x$ in $\cN \cup \cNprime.$ The latter is absurd since then $i_k'<x<i$, and in the former $i_j$ again acquires its prime weakly west of $i_k'.$

    Thus, for all cases $i_j'$ is first inserted weakly west of $i_k'.$ Therefore, it also follows that for an arbitrary column $c$ with both $i_j'$ and $i_k'$, the entry $i_k'$ is displaced first. Indeed, by induction $i_k'$ is North of $i_j'$ in $c$, and then $i_j'$ is never row-bumped without having that $i_k'$ was displaced first; whereas if $i_j'$ is column-bumped, more northerly entries with the same value are displaced first, so that again $i_k'$ is bumped to column $c+1$ before $i_j'$ is.
\end{proof}

\begin{lemma}\label{lem:eastSameSeqLemma}
    Let $w$ be an interlacing word with strict content. Without restriction throughout the insertion of $T=P_\mix(w)$, the letter $(i+1)_j'$ is always East of $i_j'$.
\end{lemma}
\begin{proof}
    From Lemma~\ref{lem:bigRowLemma} we know that for all rows $r$, the letter $(i+1)_j$ is inserted before $i_j$ is.

    Consider the moment of the insertion process when $(i+1)_j$ becomes primed in row $f.$ It is displaced by an element $x<i+1$ and for $i_j$ to acquire a prime as well, it is bumped eventually to row $f$. Furthermore, either $x=i$ or $x<i.$ In the former case, Lemma~\ref{lem:bigRowLemma} implies that the letter $i_{j+1}$ is either primed or was inserted to $f$ before $(i+1)_j$ was, and then Lemma~\ref{lem:rowOrderForSameElementDiffSeq} implies that $x=i_j$. Consequently, in being bumped, $x$ is either inserted by columns to the same column $c$ as $(i+1)_j'$, or is inserted to a column left of $c$; in both scenarios it is primed West of $(i+1)_j'$ as shown hereafter. If $x$ is inserted to $c$ and there are no repeated instances of $(i+1)'$ in $c$, then $x$ displaces $(i+1)_j'$; otherwise, it is a consequence of Lemma~\ref{lem:sameEntryPrimedDiffSeq} that $x=i_j'$ is not inserted to $c$ until after $i_{j+k}'$ is for all $k>0$. That is, after the letter $(i+1)_{j+k}'$ has been displaced from $c$, so that $i_j'$ bumps $(i+1)_j'$ when inserted to $c$.

On the other hand, if $x<i$, then before $i_j$ is inserted to $f$ we have a configuration of the form 

    \ytableausetup{boxsize=2.5em}
        \[\begin{ytableau}
            \ldots  & ~     & ~ & \vdots \\
            \none & x & y & z  \\
            \none & \none  & \ldots & ~   
        \end{ytableau}.\]

    so that if $(i+1)_j'$ is on the column of $y$, then $y\geq i+1$ and it is bumped to the next column before the insertion of $i_j.$ In general, if $(i+1)_j'$ is on column $c$, then either $(i+1)_j'$ is displaced by $i_j$, or $i_j$ is inserted to a column $k<c$. 

    Since further row-bumpings can only move $i_j$ west of column $k$, or until it becomes primed, that is; it follows that when $i_j$ acquires a prime it does so either displacing $(i+1)_j'$ due to Lemma~\ref{lem:sameEntryPrimedDiffSeq}, or in a column West of $(i+1)_j'.$ 

    Regardless of the identity of $x$, and the manner in which $(i+1)_j$ and $i_j$ become primed, our analysis yields that $i_j$ is first inserted to the tableau west of $(i+1)_j'.$ If further insertions were then to alter the relative positions of $i_j'$ and $(i+1)_j'$, it would be necessary for the latter to be in the same column as the former at one point of the insertion. But this is absurd by Lemma~\ref{lem:sameEntryPrimedDiffSeq} as $i_j'$ and $(i+1)_j'$ can only be both in the same column if $i_j'$ is inserted to the column of $(i+1)_j'$, thereby displacing it to the next column, a contradiction.     
\end{proof}

\begin{lemma}\label{lem:noTwoRepeatedCols}
    Let $w$ be an interlacing word with strict content. Then, throughout the mixed-insertion of $w$, no two primed letters of the same value are repeated in the same column.
\end{lemma}
\begin{proof}
    Towards a contradiction, suppose that at a point in the insertion of $w$ two primed entries are repeated and say that the column where this first occurs is column $c$. By Lemma~\ref{lem:sameEntryPrimedDiffSeq}, the repeated entries have consecutive subscripts; hence, denote them by $i_j'$ and $i_{j+1}'$. Again by Lemma~\ref{lem:sameEntryPrimedDiffSeq}, the letter $i_j'$ is bumped to $c$ after $i_{j+1}'$ is. 

    Consider the bumping of $i'_j$ from column $c-1$ to $c.$ Note $i_j$ must be primed immediately before this bump by Lemma~\ref{lem:GodzillaLemma} since $i_{j+1}'$ is already primed at that point. When $i'_j$ was in $c-1$, it is a consequence of Lemma~\ref{lem:smallSeqWestofBigSeqPrimed} that it was west of $(i-1)_{j+1}'$. Hence, Lemma~\ref{lem:eastSameSeqLemma} implies that $(i-1)_{j+1}'$ was in column $c-1$ with $i_j'.$ Accordingly, for $i_j'$ to be column-bumped without displacing $i_{j+1}'$, it must be column-bumped by an entry $x'$ with $(i-1)'\leq x' < i'.$ Note that $x'\neq (i-1)'$, lest there would be two instances of $(i-1)'$ in the same column, contradicting that $i_j'$ and $i_{j+1}'$ were the first primed pair to be repeated. As that was the only possible value for $x'$, we conclude $i_j'$ was not column-bumped.   

     Accordingly, $i_j'$ had to be row-bumped by an entry $y$ to column $c.$ This is only possible without also bumping $(i-1)_{j+1}'$ when $y=i-1.$ Let $r$ be the row $i_{j+1}'$ is located in. There are two cases to consider relative to whether $(i-1)_{j+1}'$ and $i_{j+1}'$ are adjacent.

     If they are adjacent, an instance of $y$ had to be bumped by an entry $z$, possibly primed, from row $r$ to $r+1.$ Such a $z$ would hence satisfy $z<y=i-1$ and be in row $r.$ But then, order considerations and the fact that a row-bumped entry must occupy a position weakly southwest of its original position, contradict that $(i-1)_{j+1}'$ and $i_{j+1}'$ are adjacent.

     If on the other hand $(i-1)_{j+1}'$ and $i_{j+1}'$ are not adjacent, then $i_{j+1}'$ is Northeast of $(i-1)_{j+1}'$ by order considerations and the fact that there cannot be two instances of $i'$ in the same row. Again, for $i_j'$ to be row-bumped, an entry $z<y$ must exist, this time in row $r+1$. But this contradicts the order in the column of $i_{j+1}'$.    
\end{proof}

\section{Characterization of constructed tableaux}\label{sec:construct}
A crucial step in many of the arguments of Section~\ref{sec:interlacingTableaux} consisted in noticing first that the order of insertion in the first row of $i_j$ and $(i-1)_j$ on one side, and $(i-1)_j$ and $i_{j-1}$ on the other, comes from the word being interlacing; and then that the same order is preserved when these letters are first inserted to subsequent rows. In fact, the corresponding order for these letters as they become primed, or inserted column-wise, still follows the same behaviour as seen in the second part of Section~\ref{sec:interlacingTableaux}. These properties were used for example to establish the fundamental result that no two unprimed entries repeat in a row, or two primed entries in a column.

But what if we had a tableau whose elements can be grouped in sequences in such a manner that they satisfy the properties of the past sections, and we continue inserting letters in a manner consistent with the grouping? Our hope is that throughout the insertion process the tableau will still satisfy the properties of Section~\ref{sec:interlacingTableaux}, e.g., that $i_j$ is bumped to any row before $(i-1)_j$ is. 

Tableaux whose elements can be grouped in that way, or to be more precise, that appear like they were inserted from an interlacing word, were defined in Definition~\ref{def:constructedDef}; we will explore these tableaux in detail throughout the next couple of sections. 

Recall from Section~\ref{sec:constructed_intro} the definitions of Serrano--Pieri strips, constructible, and constructed tableaux. A tableau $T$ constructible from $\alpha<\beta$ has a decomposition into Serrano--Pieri strips $\gamma^{(s)}=(\xi/\pi) \olessthan (\theta/\eta)$, one for each of the intervals $(\alpha_j,\beta_j]$. The obvious indexing of these Serrano--Pieri strips is by the integer $j$; however, it turns out to be more convenient to index them by $\ell(\beta) + 1 - j$ and so we do the latter. We write $\gamma^{(s)}=(\xi^{(s)}/\pi^{(s)}) \olessthan (\theta^{(s)}/\eta^{(s)})$ for the Serrano--Pieri strip using the interval $(\alpha_{\ell(\beta) + 1 - s},\beta_{\ell(\beta) + 1 - s}]$. 



We now develop some combinatorial properties of constructed tableaux, similar to properties that we have previously established for interlacing tableaux.

We first observe that if one of the Serrano--Pieri strips in a constructible tableau can be extended (so that the tableau is not constructed), then it can be extended using added boxes from greater indexed Serrano--Pieri strips. We now prove some results about the positions of entries in constructed tableaux.

\begin{lemma}\label{lem:primePlacementLemma}
    Let $\mu<\nu$ be strict partitions and $T$ constructed from $\mu$ and $\nu.$ Then, for any column $c$ and for all $j<i$, we have that the primed entries of $(\mu_j,\nu_j]$ in $c$ occur below the unprimed entries of $(\mu_i,\nu_i]$ in $c$.
\end{lemma}
\begin{proof}
    Suppose that $T$ is constructible from $\mu < \nu$ and that $y'\in \gamma^{(n)}$ is a primed entry above an unprimed letter $x\in \gamma^{(m)}$ for some $m<n$. Without loss of generality, we may assume that $x$ is the smallest element of $\xi^{(m)}/ \pi^{(m)}$ such that a primed letter in $\theta^{(n)}/\eta^{(n)}$ is above it. As $\mu<\nu$ and both are strict partitions, there is an element $w\in \xi^{(m)}/ \pi^{(m)}$ with $w<y'$ or else $y' = 1'$. But $y' = 1'$ is impossible since there cannot be $1' \in T$. Hence, there is such a $w$ and since $x > y'$, we may take $w = y-1$.

    We now analyze the position of $w$. First, observe that $w$ is in a row North of $x$ because $\xi^{(m)}/ \pi^{(m)}$ is a vertical strip. Then, $\gamma^{(m)}$ can be extended to $\gamma'^{(m)}$ by taking $\xi'^{(m)}/ \pi'^{(m)}$ to be the subset of the vertical strip $\xi^{(m)}/ \pi^{(m)}$ of elements at most $w$, and letting $\theta'^{(m)}/\eta'^{(m)}$ consist of the letters in $\theta^{(n)}/\eta^{(n)}$ that are at least $y'.$ Since $y = w+1$, then entries of $\gamma'^{(m)}$ form an interval. Because of our choice of $x$ and $y'$, we have $\xi' \subseteq \eta'$. Thus, $T$ is not constructed from $\mu$ and $\nu$. 
\end{proof}

\begin{lemma}\label{lem:rowCharConstructedTableaux}
    Let $T$ be constructed from strict partitions $\mu<\nu.$ Then in $T$:
    \begin{enumerate}
        \item no two distinct unprimed entries $i_j$ and $i_k$ occur in the same row $r$ with $r<i$;
        \item there are no repeated entries in any column; and
        \item assuming the relevant entries exist, 
        \begin{itemize}
            \item either the letter $(i-1)_{j+1}$ is weakly south of $i_j$, or
            \item $(i-1)_{j+1}$ is in row $i-1$ and $i_j$ in row $i$, or 
            \item $(i-1)_{j+1}$ is a primed entry.
        \end{itemize}
    \end{enumerate}
\end{lemma}
\begin{proof}
(1): 
Say that value $i$ has an \newword{early repeat} if there are two distinct unprimed entries $i_k, i_k$ in some row $r$ with $r < i$. We want to show that $T$ has no early repeats. Hence, we suppose that $i$ is the least early repeat in $T$ and derive a contradiction to $T$ constructed from $\mu < \nu$.

Suppose that entries $i_j$ and $i_k$ are an early repeat in row $r < i$ such that $i_k$ is immediately left of $i_j$. Define $s,t$ by $i_j\in\xi^{(t)}/\pi^{(t)}$ and $i_k\in\xi^{(s)}/\pi^{(s)}$ with $s<t$.
By the definition of constructible tableaux (Definition~\ref{def:constructibleDef}) and the fact that $\mu$ is strict, either the vertical strip $\xi^{(t)}/\pi^{(t)}$ starts with an entry greater than $\xi^{(s)}/\pi^{(s)}$ does or else both start with $1$. 
    
    In the first case, we claim that the $s$th Serrano--Pieri strip in $T$ can be extended, contradicting $T$ constructed from $\mu < \nu$.

Since $\xi^{(s)}/\pi^{(s)}$ is a vertical strip, the letter $(i-1)_k$ is Northeast of $i_k$. There are two cases, either $(i-1)_k$ is east of $i_j$ or $(i-1)_k$ is West of $i_j$. If $(i-1)_k$ is east of $i_j$, define $\gamma'^{(s)}$ by letting $\xi'^{(s)}/\pi'^{(s)}$ consist of all the entries of $\xi^{(s)}/\pi^{(s)}$ at most $(i-1)$ and all letters of $\xi^{(t)}/\pi^{(t)}$; and $\theta'^{(s)}/\eta'^{(s)}$ be $\theta^{(t)}/\eta^{(t)}$. But then $\gamma'^{(s)}$ provides an extension of $\gamma^{(s)}$, a contradiction. 

Otherwise, $(i-1)_k$ is West of $i_j$. Consider the entry $x$ in the same column as $i_j$ and the same row as $(i-1)_k$. The order in $\cN \cup \cNprime$ forces $x = i-1$ or $x= i'.$ However, $x=i-1$ is impossible by the choice of $i$ as the least early repeat, so $x = i'$. Moreover, by Lemma~\ref{lem:primePlacementLemma} we can further say that $x=i_h \in \gamma^{(m)}$ for some $h<j.$ However, it then follows that the strip $\gamma'^{(s)}$ defined by $\xi'^{(s)}/\pi'^{(s)}$ being the subset of all entries in $\xi^{(s)}/\pi^{(s)}$ less than or equal to $i-1$, and $\theta'^{(s)}/\eta'^{(s)}$ as the horizontal strip consisting of all entries at least $i'$ in $\theta^{(m)}/\eta^{(m)}$, extends $\gamma^{(s)}$, a contradiction.

    If instead $\xi^{(s)}/\pi^{(s)}$ and $\xi^{(t)}/\pi^{(t)}$ both have $1$ as their smallest entry, it follows from both being vertical strips that the entries $i_j\in\xi^{(t)}/\pi^{(t)}$ and $i_k\in\xi^{(s)}/\pi^{(s)}$ can only appear south of row $i$. Accordingly, (1) holds in either case.

       \medskip
\noindent
   (2): It suffices to note that the order imposed on the columns of a shifted semistandard Young tableau already forces there not to be repeated unprimed elements. Thus, it is enough to see that the lemma holds also for primed letters.
    
    Let $i_j'\in \theta^{(s)}/\eta^{(s)}$, $i_k'\in \theta^{(r)}/\eta^{(r)}$ be two such repeated entries with $j<k$ in column $c.$ Without loss of generality, we may assume that $i_j'$ and $i_k'$ are in consecutive rows. Since $\mu<\nu$ are strict partitions, we have that $(i+1)_j'$ exists and so $(i+1)_j'\in \theta^{(s)}/\eta^{(s)}$. If $(i+1)_j'$ is north of $i_k'$, then it, together with the greater elements of $\theta^{(s)}/\eta^{(s)}$, can be used to extend the horizontal strip $\theta^{(r)}/\eta^{(r)}$ and hence the Serrano--Pieri strip $\gamma^{(r)}$. Thus, without loss of generality, entries that exceed $i'$ in $ \theta^{(s)}/\eta^{(s)}$ are in the same row as $i_j'$.

    The letter $(i+1)_j'$ is directly below either a primed entry or an unprimed one. If the former, order considerations force the primed entry to be $(i+1)'_\ell \in \theta^{(t)}/\eta^{(t)}$ for some $t$ and $\ell.$ If $\ell<k$, then $\theta^{(t)}/\eta^{(t)}$ extends $\theta^{(r)}/\eta^{(r)}$ and $\gamma^{(r)}$ too as a consequence. Hence, $k< \ell$ and $(i+1)'_\ell \in \theta^{(t)}/\eta^{(t)}$ implies that $(i+1)'_k \in \theta^{(r)}/\eta^{(r)}$ exists. Then, it follows from $k>j$ that $(i+2)'_j\in \theta^{(s)}/\eta^{(s)}$ exists. The same argument shows that if $(i+2)'_j$ is then also directly below a primed entry, then that entry is an instance of $(i+2)'$ and $(i+3)'_j$ exists. Applying iteratively this argument, we obtain that there is a least $e>0$ such that $(i+e)_j'$ is immediately below an unprimed entry.

    Consider the unprimed entry directly above $(i+e)_j'$. By order considerations, this entry is an instance of $i+e-1$. Moreover, if this entry belongs to $\xi^{(u)}/\pi^{(u)}$, then $u<s$ because primed entries are never placed before unprimed entries of a previous sequence by Lemma~\ref{lem:primePlacementLemma}. But then, $\gamma^{(u)}$ can be extended to $\gamma'^{(u)}$ by removing all letters in the vertical strip $\xi^{(u)}/\pi^{(u)}$ that are greater than $i+e-1$ and then choose $\theta'^{(u)}/\eta'^{(u)}$ as the part of the horizontal strip $\theta^{(s)}/\eta^{(s)}$ consisting of all entries that are at least $(i+e)'.$ 

    Accordingly, we always obtain that one of the Serrano--Pieri strips comprising $T$ can be extended, a contradiction.

    \medskip
    \noindent
(3):
If $(i-1)_{j+1}$ or $i_j$ is primed, we are done, so we assume that $(i-1)_{j+1}\in\xi^{(s)}/\pi^{(s)}$ and $i_j\in\xi^{(s+1)}/\pi^{(s+1)}$.
  Suppose now $(i-1)_{j+1}$ is North of $i_j$. We must show that $(i-1)_{j+1}$ is in row $i-1$ and $i_j$ in row $i$.
  
  If $(i-1)_{j+1}$ is east of $i_j$, then we can extend the $s$-th Serrano--Pieri strip by letting $\xi'^{(s)}/\pi^{(s)}$ consist of the entries in $\xi^{(s+1)}/\pi^{(s+1)}$ that are at least $i$ together with those in $\xi^{(s)}/\pi^{(s)}$ that are at most $(i-1)$, and $\theta'^{(s)}/\eta'^{(s)}= \theta^{(s+1)}/\eta^{(s+1)}$, contradicting $T$ constructed. 

   It remains to consider the case where $(i-1)_{j+1}$ is NorthWest of $i_j$. By Lemma~\ref{lem:rowandColBoundedLemma} combined with part (1) of the current lemma, either $(i-1)_{j+1}$ is in row $i-1$, or else it is in a row $g<i-1$ and no additional instance of $i-1$ lies in row $g$. In the former case, we must also have $i_j$ in row $i$, and so we are done. We will show that the second case cannot in fact occur by deriving a contradiction. 
   
   In the second case, by the previously proved parts (1) and (2) of this lemma, we would have a situation of the form 
   \ytableausetup{boxsize=4.35em}
        \[ \scalebox{0.7}{\begin{ytableau}
             (i-1)_{j+1} &  i_k'  \\
              ~      &  i_j
        \end{ytableau}},\]
        where the lower left box may or may not exist. Either way, from Lemma~\ref{lem:primePlacementLemma}, the $i'_k$ immediately above $i_j$ satisfies $k>j.$ Say that $i_k'\in \gamma^{(n)}$. 
      
      We now analyze the scenario where the box does not exist and so $(i-1)_{j+1}$ and $i_j$ are on the main diagonal. In this scenario, $i_k'$ is the smallest entry in $\theta^{(n)}/\eta^{(n)}$ because $\theta^{(n)}/\eta^{(n)}$ is a horizontal strip. 

    Since $(i-1)_{j+1} \in \gamma^{(s)}$, $n\leq s$, and $\mu<\nu$ are strict partitions, the entry $(i-1)_k$ exists and so $(i-1)_k\in \gamma^{(n)}$. Then, $(i-1)_k$ must be weakly east of $i_j$ for its place in $T$ and the order imposed therein to be consistent with the location of $(i-1)_{j+1}$. Accordingly, $\gamma^{(n)}$ can be extended by defining $\xi'^{(n)}/\pi'^{(n)}$ to be $\xi^{(n)}/\pi^{(n)}$ together with all elements $\xi^{(s+1)}/\pi^{(s+1)}$ that are at least $i$ and defining $\theta'^{(n)}/\eta'^{(n)}$ to be $\theta^{(s+1)}/\eta^{(s+1)}$. This contradicts $T$ constructed.

   In the remaining scenario, $(i-1)_{j+1}$ and $i_j$ lie strictly above the main diagonal and so by part (1), we have the local configuration 
   \ytableausetup{boxsize=4.35em}
        \[\scalebox{0.7}{\begin{ytableau}
            \ldots  & ~     & (i-1)_{j+1} &   i_k'  \\
            \none  & \ldots &  i_\ell'         &  i_j
        \end{ytableau}}.\]
        Say that $i'_\ell \in \gamma^{(m)}$.

    There are now two possibilities: either $\ell<k$ or $k<\ell.$ If $\ell<k$, then by construction the horizontal strip $\theta^{(n)}/\eta^{(n)}$ contains primed entries greater than $i_k'$ and they can be used to extend the Serrano--Pieri strip $\gamma^{(m)}$, contradicting $T$ constructed.

    If instead $k<\ell$, then since $\mu < \nu$ are strict partitions, we obtain that $(i+1)_\ell'$ exists and so $(i+1)_\ell' \in \theta^{(m)}/\eta^{(m)}$. Note that $(i+1)_\ell'$ cannot be in the same column as $i_j$ or else $\theta^{(m)}/\eta^{(m)}$ would not be a horizontal strip. Also, if $(i+1)_\ell'$ is weakly north of $i_k'$, then it could be used to extend the horizontal strip $\theta^{(n)}/\eta^{(n)}$, a contradiction. Hence $(i+1)_\ell'$ is in the same row as $i_\ell'.$
    Moreover, the elements in $\xi^{(s+1)}/\pi^{(s+1)}$ less than $i_j$ are North of $(i-1)_{j+1}$ because they form a vertical strip and there are no repeated unprimed elements in a row by part (1). 
    
    If there was an unprimed letter $x_r$ East of $(i-1)_{j+1}$ in its row, then $r<j$ by Definition~\ref{def:constructibleDef}. Hence, if $(i+1)_\ell'$ is directly below $x_r$, and $x=i+1$ as a consequence, we can extend $\gamma^{(s+1)}$ by letting $\xi'^{(s+1)}/\pi'^{(s+1)}$ consist of all entries in $\xi^{(s+1)}/\pi'^{(s+1)}$ at most $(i-1)_{j}$ together with all unprimed letters indexed by $r$ and of value at least $x$; and letting $\theta'^{(s+1)}/\eta'^{(s+1)}$ consist of all primed entries indexed by $r.$
    
   Therefore $(i+1)_\ell'$ is directly below a primed entry $y'.$ The inequality $i_k' < y' \leq (i+1)_\ell'$ also determines $y' = (i+1)'$, so that because of its position either $y'$ extends the Serrano--Pieri strip $\gamma^{(n)}$ or $y'=(i+1)_k'$. In the former case, we are done; the latter case, coupled with $\mu < \nu$ being strict partitions, implies that $(i+2)_\ell'\in \theta^{(m)}/\eta^{(m)}.$ Applying this argument iteratively, we conclude that for some $e>0$, the letter $(i+e)_\ell'$ is the greatest element of $\theta^{(m)}/\eta^{(m)}$ and the primed element immediately above $(i+e)_\ell'$, another instance of $(i+e)'$ extends $\gamma^{(n)}$, a contradiction.  
\end{proof}

\section{Insertions into constructed tableaux}\label{sec:insertions}

After developing a robust understanding of constructed tableaux, we proceed to consider whether these properties can be conserved when new letters are inserted from the right in an appropriate order. 

\begin{lemma}\label{lem:transitionLemmaConstructed}
    Let $T$ be constructed from strict partitions $\mu < \nu$ and let $m \leq \ell(\nu)$.
    Suppose that, for all $i$ and all rows $r< m$, no two unprimed entries of value $i$ are repeated before row $i$ at any point during the insertion process $T\leftarrow \mu_{\ell(\mu)} \coloneqq T\cdot \mu_{\ell(\mu)}$. Then, for all $i_j$ occurring in a row $r \leq \min\{m, i-1\}$ and throughout the insertion process, either $(i-1)_{j+1}$ is West of $i_j$ in the same row, or it is in a subsequent row, or it is a primed entry. 
\end{lemma}
\begin{proof}
    By Lemma~\ref{lem:rowCharConstructedTableaux}, the vertical strips in $T$ are organized so that $(i-1)_{j+1}$ is in the same row or in a row below $i_j$, unless the letter $(i-1)_{j+1}$ is primed, or the letter $i_j$ is in row $i$ and $(i-1)_{j+1}$ in row $i-1$. Hence, if the lemma statement fails, it is because the insertion of $\mu_{\ell(\mu)}$ changes the relative order between two such entries $i_j$ and $(i-1)_{j+1}.$ We prove by induction on rows that such order remains unchanged.

    Consider row $1$. Either $\mu_{\ell(\mu)}-1 \in (0, \nu_{\ell(\mu)+1}]$, in which case $(\mu_{\ell(\mu)}-1)_{\ell(\mu)+1}$ was in $T$ prior to the insertion of the letter $\mu_{\ell(\mu)}$; or else $(\mu_{\ell(\mu)}-1)_{\ell(\mu)+1}$ does not exist and the lemma holds vacuously. In either case, the lemma holds for $r=1.$

    Suppose then that the lemma is likewise true of all rows $r< k < m$, and consider the case $r=k.$ Suppose towards a contradiction that in row $k\leq i-1$ the letter $i_{j-1}$ is inserted before $(i-1)_j$. The remainder of the inductive argument is now identical to the inductive part of the proof of Lemma~\ref{lem:precedingEntry}. \qedhere
\end{proof}

\begin{lemma}\label{lem:sameSeqConstructed}
    Let $T$ be a tableau constructed from strict partitions $\mu < \nu$, let $m \leq \ell(\nu)$, and let $2 \leq i \in \mathbb{Z}$. If for all rows $r < \min\{m,i\}$, no two unprimed entries of integer value $i$ are repeated north of row $r$  during the insertion process $T\leftarrow \mu_{\ell(\mu)} \coloneqq T\cdot \mu_{\ell(\mu)}$, then for each row $f< m $, for all $(i-1)_j$ occurring in $f$, and throughout the insertion process, the letter $i_j$ is in a row below $(i-1)_j$; unless $i_j$ is primed, in which case this condition may fail.
\end{lemma}
\begin{proof}
    We may assume that $i_j$ is unprimed. 
    
    All tableaux, such as $T$, constructed from strict partitions are partitioned into a collection of Serrano--Pieri strips $\gamma^{(s)}= (\xi^{(s)}/\pi^{(s)}) \olessthan (\theta^{(s)}/\eta^{(s)})$. By Definition~\ref{def:constructedDef}, the vertical strips $\xi^{(s)}/\pi^{(s)}$ are positioned so that $i_{j}\in \xi^{(s)}/\pi^{(s)}$ is in a row below $(i-1)_j \in \xi^{(s)}/\pi^{(s)}$. Hence, if the conclusion of the lemma does not hold, it is because the insertion of $\mu_{\ell(\mu)}$ changes the relative order between the entries $i_j$ and $(i-1)_{j}.$ We proceed by induction on the rows.
    
    When $\mu_{\ell(\mu)}$ is inserted to row $1$, it either displaces an instance of $(\mu_{\ell(\mu)}+1)'$ or not. If the latter, it displaces $\mu_{\ell(\mu)}+1$ because there are no repeated entries by rows by Lemma~\ref{lem:rowCharConstructedTableaux}, or it was South of row $1$; both satisfy the lemma. If the former, we have that there are no repeated entries in the columns by Lemma~\ref{lem:rowCharConstructedTableaux}, so that when $(\mu_{\ell(\mu)}+1)'$ is bumped to the next column it must displace $\mu_{\ell(\mu)}+1$ as there are no entries $x$ such that $(\mu_{\ell(\mu)}+1)'<x<\mu_{\ell(\mu)}+1.$ Irrespective of the cause, $\mu_{\ell(\mu)}+1$ is bumped to row $2.$ Thus $\mu_{\ell(\mu)}$ and $\mu_{\ell(\mu)}+1$ are in the expected order and the lemma holds for $T\leftarrow \mu_{\ell(\mu)}$ with $i_j = \mu_{\ell(\mu)}.$ As no other letters are inserted to the first row, for unrestricted $i$ and $j$, if the letter $i_j$ is in row $1$, then the lemma holds for $(i+1)_j$ and $i_j$.
    
    Suppose now that the lemma is satisfied on all rows $f<k<m$ and take now $f=k$. The inductive step is interchangeable with the reasoning for the inductive argument of Lemma~\ref{lem:orderSameSeq} and we omit the details.
\end{proof}

\begin{lemma}\label{lem:sameSeqPrimedConstructed}
    Let $T$ be constructed from strict partitions $\mu<\nu$ and fix $m \in \mathbb{Z}_{>0}$. Suppose that, throughout the mixed insertion $T\leftarrow \mu_{\ell(\mu)}$, no two unprimed entries $k$ are repeated in a row $r<\min(m,k)$. Then, for any $i \geq 2$ and $j \geq 1$, if the letter $i_j$ becomes primed north of row $m$, then it does so before $(i-1)_j$ becomes primed. In particular, if $(i-1)_j$ is primed in $T$, then $i_j$ is also primed in $T$.
\end{lemma}
\begin{proof}
    Let $\gamma^{(s)}$ be the Serrano--Pieri strip containing the possibly primed letters $(i-1)_j$ and $i_j.$ By Definition~\ref{def:constructibleDef} if $i_j$ is unprimed, then $i_j\in \xi^{(s)}/\pi^{(s)}$ and also $(i-1)_j \in \xi^{(s)}/\pi^{(s)}$ so that {\it a fortiori} $(i-1)_j$ is unprimed as well. Hence, there are three starting configurations that can occur in $T$: both $i_j$ and $(i-1)_j$ are unprimed in $T$, or $(i-1)_j$ is unprimed but $i_j$ is not, or both are primed. In particular, the lemma statement holds before the insertion of $\mu_{\ell(\mu)}$.
    
    By contradiction, suppose that $(i-1)_j$ becomes primed in row $e<m$. Then, it occupies a diagonal position $\sf{d}$ therein and from Lemma~\ref{lem:sameSeqConstructed}, the relative position of $i_j$ from $(i-1)_j$ in $T$ is preserved, i.e., $i_j$ remains South of $(i-1)_j$ when $(i-1)_j$ is in $\sf{d}$. However, order considerations forbid this scenario, unless $i_j$ is already primed. Therefore, if $(i-1)_j$ becomes primed before $i_j$, it must do so south row $m.$ But by the hypothesis, combined with Lemma~\ref{lem:sameSeqConstructed}, it cannot be South of row $m$. Hence $(i-1)_j$ must become primed in row $m$.
    
    Consider the case where $(i-1)_j$ was in the diagonal position $\sf{d}$ of row $m$ prior to the insertion of $\mu_{\ell(\mu).}$ Note that $(i-1)_j \in  \xi^{(s)}/\pi^{(s)} $; then $\xi^{(s)}/\pi^{(s)}$ would have been hindered from being a vertical strip by the fact that $i_j$ must belong to it as well. We thus may assume that $(i-1)_j$ is inserted to row $m$ but not there originally. 
    
    Recalling that from Lemma~\ref{lem:sameSeqConstructed}, whenever $(i-1)_j$ occupies a row $r<m$, the letter $i_j$ is in a row below it or is primed, the letter $i_j$ must necessarily have been in row $m$ before $(i-1)_j$. As it is primed north of row $m$ by assumption, it must therefore become primed in row $m.$ Additionally, again letting $\mathsf{d}$ be the diagonal box of row $m$, $(i-1)_j$ must have been inserted to that row thereby bumping an entry $y$ of value $i$ or greater from $\sf{d}$. If $y \neq i_j$, then $i_j$ in turn would have been bumped by an entry $x<i_j$. But then, the entry $y$ in $\sf{d}$ satisfies $y\leq x < i_j$. Furthermore, $y=i$ if $(i-1)_j$ is to be in $\sf{d}$ later. Hence $i\leq x <i$, a contradiction. Thus $i_j$ becomes primed before $(i-1)_j.$
\end{proof}

Before continuing with our study of right multiplication, we need another lemma on the structure of constructed tableaux.
\begin{lemma}\label{lem:constructedAncillaLemma}
    Let $T$ be constructed from strict partitions $\mu<\nu.$ We have that, for all $i$ and $j$, in $T$:
    \begin{enumerate}
        \item if $(i+1)_{j-1}$ is primed, then so is $i_j;$
        \item for all $0<s<t$, if $\theta^{(s)}/\eta^{(s)} \neq \emptyset$ and $\theta^{(t)}/\eta^{(t)} \neq \emptyset$, then $\min \theta^{(s)}/\eta^{(s)} < \min \theta^{(t)}/\eta^{(t)}$; 
        \item let $i'\in \gamma^{(m)}$ and $(i+k)'\in \gamma^{(m-d)}$ for some $m,d>0,$ we have that $i'$ and $(i+k)'$ are not in the same row; and
        \item if $i_j$ and $(i+1)_{j-1}$ are both primed in $T$, then $(i+1)_{j-1}'$ is west of $i_j'$. 
    \end{enumerate}
\end{lemma}
\begin{proof}
(1):
    Suppose that $(i+1)_{j-1}$ is primed but $i_j$ is not. By Lemma~\ref{lem:primePlacementLemma}, if $i_j\in \gamma^{(s)}$ and $(i+1)_{j-1}'\in \gamma^{(t)}$, then modifying $\xi^{(s)}/\pi^{(s)}$ by letting it end with $i_j$ and replacing $\theta^{(s)}/\eta^{(s)}$ with the elements of $\theta^{(t)}/\eta^{(t)}$ that are greater than or equal to $(i+1)_{j-1}'$, we obtain an extension of $\gamma^{(s)}.$ This proves (1).

    \medskip
    \noindent
    (2): Fix $0<s<t.$  Let $i_j' = \min \theta^{(s)}/\eta^{(s)}$ and $u_r' = \min \theta^{(t)}/\eta^{(t)}$. Towards a contradiction, assume $i\geq u$. Since $\mu<\nu$ are strict partitions and $i_j\geq u_r$, we have $u_r-1 \in \xi^{(s)}/\pi^{(s)}$. But then, it is a consequence of Lemma~\ref{lem:primePlacementLemma} and Definition~\ref{def:constructibleDef} that the primed entries of $\theta^{(t)}/\eta^{(t)}$ are all placed below primed or unprimed entries of Serrano--Pieri strips of smaller indices; a fortiori, below all entries of $\gamma^{(s)}.$ Thus, modifying $\theta^{(s)}/\eta^{(s)}$ to contain all those entries of $\theta^{(t)}/\eta^{(t)}$ that are greater than or equal to $u_r'$ yields an extension of $\gamma^{(s)}$, a contradiction.

    \medskip
    \noindent
    (3): If  $i'\in \gamma^{(m)}$ and $(i+k)'\in \gamma^{(m-d)}$ are in the same row, then the hypothesis that $\mu<\nu$ are strict, together with (2), implies that there must exist a different instance of $i'$ such that $i'_r\in \gamma^{(m-d)}$. Furthermore, $i_r'$ is South of $(i+k)'$ because of Definition~\ref{def:constructibleDef} and the fact that there are no repeated primed entries in the same row. Hence $\gamma^{(m-d)}$ can be extended using primed entries of $\gamma^{(m)},$ a contradiction.

    \medskip
    \noindent
    (4): Suppose that $(i+1)_{j-1}'\in \gamma^{(m)}$ is East of $i_j'\in \gamma^{(k)}$ and let $i_j'$ be in row $r$. When $(i+1)_{j-1}'$ is weakly north of $i_j'$, it is clear that $\gamma^{(k)}$ can be extended, so we restrict ourselves to the case where $(i+1)_{j-1}'$ is South of $i_j'.$ Order considerations force that $(i+1)_{j-1}'$ is in either row $r+1$ or $r+2$. The latter is not possible because there can be no repeated primed entries in the columns (Lemma~\ref{lem:rowCharConstructedTableaux}) and configurations that avoid using another instance of $(i+1)'$ are absurd. 

    Hence taking $(i+1)_{j-1}'$ to be in row $r+1$, it is also the case that $i_j'$ and $(i+1)_{j-1}'$ are in successive columns. Indeed, for their distance to be greater, there would need to be repetitions of unprimed entries with value $i$ in both rows $r$ and $r+1$, a contradiction by Lemma~\ref{lem:rowCharConstructedTableaux}. That is, $i_j'$ and $(i+1)_{j-1}'$ are as
    \ytableausetup{boxsize=4em}
        \[\scalebox{0.7}{\begin{ytableau}
            \ldots & ~      & ~     & \vdots  & \vdots   \\
            \none  & \dots  &  ~    & i_j'    &    x  \\
            \none  & \none  & \dots &  y   &  (i+1)_{j-1}'
        \end{ytableau}}.\]

    Once again the order imposed on a semistandard shifted tableau restricts the number of potential scenarios, forcing $y$ to be an instance of $i$, denoted $y=i_f$, and imposing the value of the entry $x$ to be $i$ as well, denoted by $i_m.$ 

    If $f<m$, it follows that $(i-1)_f$ exists and then by Lemma~\ref{lem:sameSeqPrimedConstructed} as $i_m$ is unprimed in $T$, the letter $(i-1)_m$ is also unprimed therein. Similarly, by (1), it is immediate that $i_{m-1}$ is likewise unprimed. Coupled with Lemma~\ref{lem:rowCharConstructedTableaux} we conclude that $i_{m-1}$ is weakly north of $i_m$. Furthermore, for every $e>0$ it follows from the same argument that $i_{m-e}$ is weakly north of $i_m$. This contradicts that $f<m$. 
    Thus, $f>m.$  Say $i_f\in \gamma^a$.

    If $(i+1)_m$ is primed, then $\gamma^a$ can be extended using $i_f$ and $(i+1)_m'$ because of Lemma~\ref{lem:primePlacementLemma}. Thus $(i+1)_m$ must be unprimed.

    Let us consider then the position of the unprimed letter $(i+1)_m.$ For it to form a vertical strip along with $i_m$ and possibly other unprimed entries, it must be Southwest of $i_m.$ Comparing now the positions of $i_{m+v}$, $(i+1)_{m+v}$ and $i_{m+e}$ for $0\leq v < e\leq r$ using Lemma~\ref{lem:rowCharConstructedTableaux} and the fact that all are unprimed in that range because of Lemma~\ref{lem:sameSeqPrimedConstructed} and part (1), so that the comparison is legitimate, it is seen that $(i+1)_m$ is weakly north of $i_f$ so that by order considerations it cannot be in the same column as $(i+1)_{j-1}'$.
    
    Accordingly, $(i+1)_m$ is SouthWest of $i_m.$ Furthermore, its relative location to $i_f$ forbids it from being weakly northwest of $i_f.$ Hence $(i+1)_m$ is Southwest of $i_f$, and thus can be used to extend $\gamma^a$, a contradiction. We conclude that $(i+1)_{j-1}'$ cannot be East of $i_j'.$
\end{proof}

\begin{lemma}\label{lem:diffSeqPrimedConstructed}
    Let $T$ be constructed from strict partitions $\mu<\nu.$ If for a fixed $m$, no two unprimed entries of value $v$ are repeated in a row $r< \min(m,v)$, then throughout the insertion $T\leftarrow \mu_{\ell(\mu)}$ if a letter $i_{j+1}$ becomes primed in row $f\leq m$, it does so before $i_{j}$ becomes primed. 
    In particular, we have that if the letter $i_j$ is primed in $T$, then $i_{j+1}$ is also primed in $T$. 
\end{lemma}
\begin{proof}

 If the entry $i_{j+1}$ is unprimed in $T$, then by Lemma~\ref{lem:sameSeqPrimedConstructed} so is $(i-1)_{j+1}$ in $T$. Moreover, if $(i-1)_{j+1}$ is unprimed, then so is $i_j$, this time by Lemma~\ref{lem:constructedAncillaLemma}(1).
Consequently, if $i_{j+1}$ is unprimed in $T$, then so is $i_j$. This proves the ``in particular'' statement.

Suppose by contradiction that $i_{j}$ becomes primed in a diagonal position $\sf{d}$ on row $e$ while $i_{j+1}$ remains unprimed. Since $i_{j+1}$ becomes primed in row $f$, we know that it is in a row $w<f \leq m$ when $i_{j}$ is in row $e$, as otherwise both $i_j$ and $i_{j+1}$ are in diagonal positions of different rows, an absurd situation by order considerations. Moreover, as $i_{j+1}$ remains unprimed, we have by Lemma~\ref{lem:sameSeqPrimedConstructed} that $(i-1)_{j+1}$ is likewise unprimed when $i_j$ becomes primed. But then, $(i-1)_{j+1}$ was also unprimed in $T$ before the insertion of $\mu_{\ell(\mu)}.$ Therefore, $(i-1)_{j+1}$ was weakly south of $i_j$ or else $i_j$ was in row $i$ and $(i-1)_{j+1}$ in row $i-1.$ The second possibility is ruled out by the fact that $i_j$ is becoming primed at a later point of the insertion and Lemma~\ref{lem:rowCharConstructedTableaux}. Hence, in $T$ the letter $i_{j+1}$ was south of $i_j$ in $T$. But then by Lemma~\ref{lem:rowandColBoundedLemma},  $i_{j+1}$ was South of $i_j$ in $T$.

Now, for the relative order between $i_{j+1}$ and $i_j$ to change before $i_j$ arrives to ${\sf{d}}$, both need to first repeat in a row $u\leq w$, and then $i_j$ needs to be bumped out of $u$ before $i_{j+1}$ is. Note that such a bumping cannot be a row-bumping, since every entry $x<i$ inserted to $u$, will displace $i_{j+1}$ first, it being the leftmost $i$ in row $u.$ An entry $y'$ bumping $i_{j}$ by columns would however also be contradictory because no $y'$ can satisfy $i_{j+1}\leq y' < i_j.$ Thus we reach an absurdity and the order between $i_{j+1}$ and $i_j$ cannot change.

Accordingly, when $i_j$ is in $\sf{d}$ before becoming primed, order considerations compel $i_{j+1}$ to be in the same row $e$. That is, $e=w.$ But then, since $i_{j+1}$ is left of $i_j$ when in the same row by the argument above, it is this entry that lies in the diagonal, not $i_j.$
\end{proof}

\begin{lemma}\label{lem:noRepetConstructedTableaux}
   Let $T$ be constructed from strict partitions $\mu<\nu$.  Then, throughout the insertion $T\leftarrow \mu_{\ell(\mu)}$ and for all $i>1$, no two letters $i_j$ and $i_k$ occur simultaneously in the same row $r<i$.
\end{lemma}
\begin{proof}
    By Lemma~\ref{lem:rowCharConstructedTableaux}(1), the lemma holds for $T$ before the insertion.

    Let $j \coloneqq \ell(\nu)-\ell(\mu) + 1$. Since the interval $(\mu_{\ell(\mu)}, \nu_{\ell(\mu)}]$ codifies the shrinking sequence with greatest index whose Serrano--Pieri strip is missing an instance of $1$, and this Serrano--Pieri strip is also missing $\mu_{\ell(\mu)}$; we assign to the inserted entry the index $j$. That is, rigourously, it corresponds to the letter $(\mu_{\ell(\mu)})_j$, where it is important not to confuse the shrinking sequence notation with the notation indicating one of the parts of $\mu.$
    
    We prove the lemma for $T\leftarrow \mu_{\ell(\mu)}$ by induction on $r$. For the base case $r=1$, note that if there were a letter $(\mu_{\ell(\mu)})_k$ in row $1$ with $k\neq j$ then $(\mu_{\ell(\mu)}-1)_k$ would also be part of $T$ because $\mu$ is a strict partition and the indexing of letters in a shrinking sequence is done in the opposite order than that of Serrano--Pieri strips by convention. Let $s$ be such that $(\mu-1)_k \in \pi^{(s)}/\xi^{(s)}$. Since the entries of $\pi^{(s)}/\xi^{(s)}$ form a vertical strip, $(\mu_{\ell(\mu)})_k$ must be South of $(\mu_{\ell(\mu)}-1)_k$ and hence not in row $1$, a contradiction. Thus, as the insertion of an entry in row $1$ may displace further unprimed entries southward but never cause another unprimed entry to be displaced to row $1$, it follows that on being inserted $\mu_{\ell(\mu)}$ does not cause the repetition of two unprimed entries in row $1.$ As no such entries were repeated in $T$ by Lemma~\ref{lem:rowCharConstructedTableaux}(1), this establishes the lemma for row $r=1$. 

    Suppose now that the lemma holds for all rows $r<m$ and for the sake of contradiction that two letters $i_h$ and $i_k$ get repeated in row $m$ for $h\neq k$, so that this is the first repetition in the insertion of $\mu_{\ell(\mu)}.$ Without loss of generality, there must have been a point in that insertion in which $i_h$ was in row $m$, and $i_k$ in row $m-1.$

    Note that by Lemma~\ref{lem:diffSeqPrimedConstructed} if $a<b$, and both $i_a$ and $i_b$ are unprimed in $T$, then $i_x$ is also unprimed for all $a< x \leq b$ before the insertion of $\mu_{\ell(\mu)}.$ Then, for such $x$ we also have that $(i-1)_{x+1}$ is weakly south of $i_x$ or in the immediately preceding row. Combining that with the fact that $i_{x+1}$ is always South of $(i-1)_{x+1}$ in $T$ because of Lemma~\ref{lem:sameSeqPrimedConstructed}, it follows that $i_{x+1}$ is weakly south of $i_x$ for all $x$ with $a<x\leq b$. Iterating, we further have that $i_b$ is weakly south of $i_a$ and by the inductive hypothesis we can further say that $i_b$ is South of $i_a.$

    Consider now the insertion of $\mu_{\ell(\mu)}.$ Because of the inductive hypothesis, $i_b$ remains South of $i_a$ throughout the insertion whenever $i_a$ is in a row $r< m$: the relative order between both letters is only allowed to change after the two entries have repeated in a row. Therefore, that $i_h$ is in row $m$ and $i_k$ in row $m-1$ implies that $h>k$ and even more, $h=k+1.$

    Accordingly, Lemma~\ref{lem:transitionLemmaConstructed} implies that either $(i-1)_h$ is South of $i_h$, a contradiction because of Lemma~\ref{lem:sameSeqConstructed}; or else $i_h$ is in row $i$ and $(i-1)_h$ is in row $i-1.$ The latter is likewise inconsistent, as $(i-1)_h$ should still be weakly south of $i_k$ when it is displaced to row $m$ by Lemma~\ref{lem:transitionLemmaConstructed}, but in that case, as $i_h$ is the sole unprimed entry of that value, it would get displaced. 
\end{proof}

\begin{lemma}\label{lem:weaklyWestCompletionOfConstructed}
    Let $T$ be constructed from strict partitions $\mu<\nu$. Fix $i \geq 1, j\geq 2$. Suppose that $(i+1)_{j-1}$ is primed at one point of the insertion process and that throughout the process no two entries are repeated in any column. Then, at one point of the insertion $T\leftarrow \mu_{\ell(\mu)}$, the letter $i_j$ becomes primed before $(i+1)_{j-1}$. Moreover, from then on, $(i+1)_{j-1}'$ is always weakly west of $i_j'$.
\end{lemma}
\begin{proof}
    If $i_j$ and $(i+1)_{j-1}$ are unprimed in $T$, all the arguments of Lemma~\ref{lem:GodzillaLemma} apply {\it mutatis mutandis}. Otherwise, Lemma~\ref{lem:constructedAncillaLemma} implies that $i_j$ is primed and $(i+1)_{j-1}$ is not. This scenario is similar to {\sf Cases 2} and {\sf 3} of Lemma~\ref{lem:GodzillaLemma}, but we spell out the differences needed, as there are a few subtleties. Note that if $(i+1)_{j-1}$ is West of $i_j'$ then it is SouthWest of it by order consideration, and its bumping trajectory, weakly southwest of its current position, ensures that if it becomes primed, then it does so in a column weakly left of $i_j'.$ Thus the initial analysis of {\sf Case 2} holds again and the only possible situations are described by {\sf Cases 2.0--2.3}. In this scenario, all cases but {\sf Case 2.1} are {\it mutatis mutandis} the same as the proof for tableaux arising from the insertion of an interlacing word. 
    
    Finally, we consider the analogue of {\sf Case 2.1} in more detail. Here, the position of $i_j'$ is NorthWest of $(i+1)_{j-1}$ by hypothesis. Four different alternatives arise when considering whether $(i+1)_{j-1}$ will be eventually inserted to the column $c$ of $i_j'$. 

    Let $\sf{d}$ be the diagonal position of column $c$. If there is an instance of $i$ in $c$ but not in $\sf{d}$, or there is an instance of $(i+1)'$ in column $c$, then as an unprimed letter must occupy $\sf{d}$, an entry $x\geq i+1$ has to be in this position in $T$. Moreover, order considerations imply that $(i+1)_{j-1}$ is north of row $c$. It can only be in row $c$ if $c=i+1$, in which case $(i+1)_{j-1}$ is not primed at any point of the insertion and the lemma holds vacuously. Therefore, $(i+1)_{j-1}$ is North of row $c$, so even when $x=i+1$, Lemma~\ref{lem:noRepetConstructedTableaux} forces $(i+1)_{j-1}$ to be eventually inserted to $c$.

    If instead there is an instance of $i$ in column $c$ and $i$ is in $\sf{d}$, then either $i = c$ and $(i+1)_{j-1}$ is never primed, or else $(i+1)_{j-1}$ is never displaced to the next row so that it is again never primed. 

    Finally, if both $i_j$ and $(i+1)_{j-1}$ are primed, then Lemma~\ref{lem:constructedAncillaLemma}(4) establishes the desired conclusion. 

    Now the proof of Lemma~\ref{lem:smallSeqWestofBigSeqPrimed} guarantees than once $i_j$ and $(i+1)_{j-1}$ are primed with $(i+1)_{j-1}$ weakly west of $i_j'$, their relative order is preserved. 
\end{proof}

\begin{lemma}
    Let $T$ be constructed from strict partitions $\mu<\nu$. Suppose that, for all $i$ and throughout the mixed insertion $T\leftarrow \mu_{\ell(\mu)}$, two instances of $i$ never appear simultaneously in the same row $r<i$. 
    
    Let $i_j, i_k$ be letters with $j<k$. If, at some point of the mixed insertion, $i_j$ becomes primed, then $i_k$ is primed earlier. Moreover, after both are primed, $i_k'$ is weakly east of $i_j'.$ Moreover, at any point of the insertion where $i_j'$ and $i_k'$ are in the same column, we have $i_k'$ North of $i_j'$.
\end{lemma}
\begin{proof}
    By Lemmas~\ref{lem:diffSeqPrimedConstructed} and~\ref{lem:noRepetConstructedTableaux}, if $i_j$ is primed in $T$, then so is $i_k$; furthermore, if $i_k$ becomes primed during the mixed insertion, it does so before $i_j$. Accordingly, it always holds that if $i_j$ is primed, then $i_k$ is primed too. 

     Suppose first that $i_k$ and $i_j$ are primed in $T$. Since $\mu < \nu$ are strict partitions, $(i+1)_j$ exists in $T$. Hence, Lemma~\ref{lem:constructedAncillaLemma}(4) coupled with the conditions imposed on the primed entries of the same Serrano--Pieri strip in Definition~\ref{def:constructibleDef} imply that $i_j'$ is West of $i_k'$ in $T$. This relation is preserved throughout the mixed insertion process by the same reasoning as in the proof of Lemma~\ref{lem:sameEntryPrimedDiffSeq}.
    
    Suppose instead that $i_j$ is unprimed in $T$. We can assume that, after some point of the mixed insertion, both $i_j$ and $i_k$ are primed, lest the lemma be vacuously true.
   In this scenario, there must exist a diagonal position $\sf{d}$ reached by $i_j$ during the insertion of $\mu_{\ell(\mu)}$. Say that the column of $i_k'$ at the point that $i_j$ reaches $\mathsf{d}$ is column $e$ and let $d$ be the column of $\sf{d}$. There are three cases for the relative positions of the columns $d$ and $e$. 

    {\sf Case 1 ($e<d$):}
In this scenario $i_k'$ is West of $i_j$. However, $i_j$ is in a diagonal position, so $i_k'$ is NorthWest of $i_j$, contradicting the order of the entries inside a shifted tableau.

{\sf Case 2 ($d=e$):}
The letter $i_j$ can only be displaced by an unprimed entry $x$. Noting that $i_k'<x<i_j$ is impossible for such an entry, it follows that $i_k'$ is displaced to the next column before $i_j$ becomes primed. In particular, if $i_j'$ is inserted to the same column as $i_k'$ it is in a position below it. 

{\sf Case 3 ($d<e$):}
Because of the location of $i_j$ and $i_k'$ in this case, it is clear that $i_j$ becomes primed weakly west of $i_k'.$ Also, we have again that if $i_j$ is inserted to the same column as $i_k'$ then it is in a position below it. 

Thus, regardless of the location of the columns $d$ and $e$, the letter $i_j$ becomes primed weakly west of $i_k'$, and in a position below it if on the same column. 
But then it also follows that for an arbitrary column $c$ with both $i_j'$ and $i_k'$, the entry $i_k'$ is displaced first. Indeed, by induction $i_k'$ is North of $i_j'$ in $c$, and then $i_j'$ is never row-bumped without having that $i_k'$ was displaced first; whereas if $i_j'$ is column-bumped, more northerly entries with the same value are displaced first, so that again $i_k'$ is bumped to column $c+1$ before $i_j'$ is. Hence, the lemma also holds for $i_j$ unprimed in $T$. 
\end{proof}

\begin{lemma}\label{lem:eastSameSeqCompletionCompletionConstructed}
    Let $T$ be constructed from strict partitions $\mu<\nu$. Then, for all $i$ and $j$ and throughout the insertion $T\leftarrow \mu_{\ell(\mu)}$, the letter $(i+1)_j'$ is East of $i_j'$.
\end{lemma}
\begin{proof}
    The proof is completely analogous to the reasoning for Lemma~\ref{lem:eastSameSeqLemma}.
\end{proof}

\begin{lemma}\label{lem:noTwoRepeatedColsCompletionConstructed}
    Let $T$ be constructed from strict partitions $\mu<\nu$. Then, throughout the insertion $T\leftarrow \mu_{\ell(\mu)}$, no  letters are repeated in any column. 
\end{lemma}
\begin{proof}
    The proof is entirely analogous to the proof of Lemma~\ref{lem:noTwoRepeatedCols}.
\end{proof}

\begin{lemma}\label{lem:puttingAllTogetherCompletionConstructed}
    Let $T$ be a tableau constructed from segments $\bigcup_{i=1}^{\ell(\nu)}(\mu_i, \nu_i]$ for some strict partitions $\mu <\ \nu$.
    Then, the tableau
    \[ \hat{T} \coloneqq T \cdot \seq(\mu_{\ell(\mu)}, 1) \cdot \seq(\mu_{\ell(\mu)-1}, 1) \cdot \seq(\mu_1, 1 ) \]
satisfies all the properties established for $T\leftarrow \mu_{\ell(\mu)}$ in Section~\ref{sec:insertions} up to this point. 
\end{lemma}
\begin{proof}
    Since $T$ is constructed, Lemmas~\ref{lem:noRepetConstructedTableaux} and~\ref{lem:noTwoRepeatedColsCompletionConstructed} apply and show that the consequents of Lemmas~\ref{lem:transitionLemmaConstructed}--\ref{lem:noTwoRepeatedColsCompletionConstructed} hold for $T\leftarrow \mu_{\ell(\mu)}$. 

    Moreover, all the lemmas related to completions of constructed tableaux use a weaker starting object than it seems. Despite assuming that the initial object is a constructed tableaux, the Serrano--Pieri structure of $T$ is never used unless to establish that the relative positions of $i_j$, $(i-1)_j$, and $(i-1)_{j-1}$ are the appropriate on one side; and that of $i'_j$, $(i-1)'_{j}$, and $(i+1)_{j-1}'$ on the other side. But if $T$ is constructed from strict partitions $\mu<\nu$ and $\mu_{\ell(\mu)}$ is inserted on the right, it is the content of Lemmas~\ref{lem:transitionLemmaConstructed}--\ref{lem:noTwoRepeatedColsCompletionConstructed} that the relative positions of those letters is the expected one. Hence, iterating those lemmas for the letters $\seq(\mu_{\ell(\mu)}, 1) \cdot \seq(\mu_{\ell(\mu)-1}, 1) \cdot \seq(\mu_1, 1 )$ we obtain the desired result for all letters originally in $T.$ 

    It remains then to see that when inserting $\seq(\mu_{\ell(\mu)}, 1) \cdot \seq(\mu_{\ell(\mu)-1}, 1) \cdot \seq(\mu_1, 1 )$ the order between the inserted letters is also in accordance with that order. For that, it is enough to note that if the letters in $\seq(\mu_{\ell(\mu)}, 1 )$ are identified with the $j$-th shrinking sequence then the letters of $\seq(\mu_{\ell(\mu) - i}, 1 )$
 would belong to the $(j-i)$-th shrinking sequence; so that, for all $i>1$ and $1\leq k\leq j$ the letter $i_k$ is inserted before $(i-1)_k$, and the letter $i_k$ before $(i+1)_{k-1}.$ \qedhere
\end{proof}

The following lemma is useful for building constructed tableaux by hand; we will moreover use it in the calculations of Section~\ref{sec:examples} and in the complexity analysis of Section~\ref{sec:complexity}.
\begin{lemma}\label{lem:constructedSmallChar}
Suppose that $T_\lambda$ is constructed from strict partitions $\tau < \psi$. Let also $i$ be an entry of minimum value satisfying that the number of copies of $i$ in $\cup_i (\mu_i, \nu_i]$ is greater than the number of instances of $i$ in $T_\lambda$; and let $j$ be maximal so that $i_j\notin T_\lambda$. Place $i_j$ so that $T_\lambda \cup \{i_j\}$ is constructible from $\tau' < \psi'$ for some strict partitions $\tau'$ and $\psi'$, without relabelling of its shrinking index. If the placement of $i_j$ results in a valid shifted tableau $T_\lambda'$, we have that $T_\lambda'$ is constructed from $\tau'<\psi'$ if and only if:
\begin{enumerate}
    \item $i_j$ is unprimed and it is weakly north of the unprimed letter $(i-1)_{j+1}$, or
    \item $i_j$ is unprimed and it is in row $i$; in which case $(i-1)_{j+1}$ is likewise unprimed and in row $i-1$; or 
    \item $i_j$ is primed and it is weakly west of the primed letter $(i-1)_{j+1}$, or
    \item $i_j$ is unprimed and $(i-1)_{j+1}$ is primed.
\end{enumerate}
\end{lemma}
\begin{proof}
By Lemma~\ref{lem:rowCharConstructedTableaux} and Lemma~\ref{lem:constructedAncillaLemma} all $4$ conditions are met if $T_\lambda'$ is constructed from $\tau' < \psi'$.

Now, let us see that the 4 conditions guarantee none of the shrinking sequences can be extended when $i_j$ is added, primed or unprimed. 

(1): If $i_j$ and $(i-1)_{j+1}$ are both unprimed, clearly $i_j$ being north of $(i-1)_{j+1}$ ensures that at least the $(j+1)$-st sequence cannot be extended. Suppose then that (1) holds and that the $(j+k)$-th sequence can be extended for some $k>1.$ This means that there is an unprimed letter $(i-1)_{j+k}$ northeast of $i_j$. Recall that $i_j$ is north of $(i-1)_{j+1}$ and that $j$ was chosen so as to be maximal with respect to $i_j\notin T_\lambda$. Hence $i_{j+1}\in T_\lambda$ and $i_j$ is north of $i_{j+1}$ too. Moreover, order considerations force $i_j$ to be northeast of $i_{j+1}$, yielding that $(i-1)_{j+k}$ is northeast of $i_{j+1}$ as well. This implies that the $(j+k)$-th shrinking sequence can be extended by the $(j+1)$-st sequence, contradicting that $T_\lambda$ was a constructed tableau before the addition of $i_j.$

(2): Assume now $(2)$, and towards a contradiction that $i_j$ is southwest of $(i-1)_{j+1}$. First, notice that again the maximality of $j$ implies that $i_{j+1}\in T_\lambda$. As $(i-1)_{j+1}$ is in row $i-1$, the letter $i_{j+1}$ must be in row $i$ by Lemma~\ref{lem:rowandColBoundedLemma}. Since $i_j$ was the last letter to be added, it is East of $i_{j+1}$ in the same row. Its being southwest of $(i-1)_{j+1}$ then guarantees that there is at least one other letter West of $(i-1)_{j+1}$ in the same row in a non-diagonal position. In fact, a stronger statement holds: There exists a letter northeast of $i_{j+1}$ in the same row as $(i-1)_{j+1}$ different from $(i-1)_{j+1}$. But this letter must have a value smaller than or equal to $i-1$ and be in the $(i-1)$-st row. That is, it is another instance of $i-1$, in the $(j+r)$-th shrinking sequence for $r>1$. Hence, $i_{j+1}$ extends the $(j+r)$-th sequence contradicting that $T_\lambda$ is constructed. 

Thus $i_j$ is southEast of $(i-1)_{j+1}$. If there were an entry $(i-1)_{j+k}$ northeast of $i_j$ for some $k>1$, it would then follow that $(i-1)_{j+k}$ does not share its row with $(i-1)_{j+1}$ since $T_\lambda$ is constructed. Moreover, being North of $(i-1)_{j+1}$, it would also lie northeast of $(i-1)_{j+1}$, so that $i_{j+1}$ would extend the $(j+k)$-th sequence, a contradiction. 

(3): A similar analysis to that of (1) proves that no existing sequences are extended.

(4): In this case, $(i-1)_{j+k}$ is also primed for all $k> 1$. Thus, the addition of $i_j$ does not extend a shrinking sequence of $T_\lambda.$ 
\end{proof}

\section{Interlacing tableaux are exactly constructed tableaux}\label{sec:interlacing=constructed}
Now that most of the relevant results for completions of constructed tableaux and the characterization of interlacing tableaux are in place, we are prepared to establish that they correspond to each other. More precisely, that if strict partitions $\mu<\nu$ are fixed, the class of tableaux $T$ constructed from $\mu$ and $\nu$ is exactly the class of interlacing tableaux with that content.

\subsection{Right factors of barely Yamanouchi words are barely Yamanouchi}

\begin{lemma}\label{lem:shiftedIsInterlacing}
    Let $w$ be a shifted lattice word. Then $w$ is also interlacing.
\end{lemma}
\begin{proof}
    This follows by Lemmas~\ref{lem:consecutiveIs} and~\ref{lem:combinatorialCharShiftedYamanouchi}. 
\end{proof}

In what follows we sometimes abbreviate ${\ell(\nu)}$ by $\ell.$

\begin{lemma}\label{lem:insertionOfShrinkingSequence}
    Let $\nu$ be a strict partition, and $w$ a shifted lattice word with $\ct(w)=\ct(\hat{y}_\nu)$. Moreover, suppose that for all $i \geq r\geq 1$ and $j\geq r-1$, either the unprimed letter $i_{\ell(\nu) - j}$ is inserted to row $r$ during the mixed insertion of $w$ or $i_{\ell(\nu) - j}$ becomes primed North of row $r$. Then, for all $i \geq 1$ and $j\geq 0$ and throughout the mixed insertion of $w$:
   \begin{enumerate}
       \item the letter $i_{\ell(\nu)-j}$ is never South of row $j+1;$ and
       \item if $i>j+1$, then the letter $i_{\ell(\nu)-j}$ becomes primed in a row $b\leq j+1$, while if $i\leq j+1$, then $i_{\ell(\nu)-j}$ remains unprimed throughout the insertion.
   \end{enumerate}     
\end{lemma}
\begin{proof}
    (1): Fix $i\geq 1$ and $j \geq 0$. Since $w$ is shifted lattice, it is interlacing by Lemma~\ref{lem:shiftedIsInterlacing}. By Lemma~\ref{lem:GodzillaLemma}, $i_j$ becomes primed, if at all, before $(i+1)_{j-1}$. Similarly, we claim that $i_j$ becomes primed, if it does, before $(i-1)_j$. Indeed, as long as both are unprimed, $i_j$ is inserted to all rows before $(i-1)_j$ because of Lemma~\ref{lem:bigRowLemma}; but then, if $(i-1)_j$ is primed in a diagonal position $\sf{d}$ of row $t$, the letter $i_j$ had to be inserted in $\sf{d}$ too when it was in row $t$ by Lemma~\ref{lem:rowandColBoundedLemma}, so that $i_j$ is displaced out of $\sf{d}$ when $(i-1)_j$ is inserted to row $t.$ 

    Accordingly, we have that if $i_j$ is unprimed, so are $(i+e)_k$ for all $k<j$ and $-i+1\leq e \leq 1$. Then, as laid out in Lemma~\ref{lem:bigRowLemma}, we then have that for all $e>0$ and $k>j$, if inserted, the letters of $w$ are inserted to any fixed row in the order:
         \begin{align}
        i_{\ell-j} \rightarrow (i-1)_{\ell-j} \rightarrow \dots \rightarrow 1_{\ell-j} \label{eq:row_order}\\
        i_{\ell-j} \rightarrow (i-e)_{\ell-k} \label{eq:more_row_order}
    \end{align} 
   where \eqref{eq:more_row_order} follows from the fact that either $i_{\ell-j} \rightarrow (i+1)_{\ell-j-1}$ or $i_{\ell-j}$ is in the row preceding that of $(i+1)_{\ell-j-1}$; and both situations yield $i_{\ell-j}\rightarrow i_{\ell-j-1}$, which coupled with the repeated application of \eqref{eq:row_order} establishes the desired order of insertion. 

   By hypothesis, $i_{\ell-j}$ is inserted to row $j+1$ or it was primed in a row $f<j+1.$ If the latter, the desired conclusion follows, taking $b=f.$ Considering the former case, let us prove for arbitrary $n$ and by induction on $y$, that $n_{\ell-y}$ is never inserted to a row South of row $y+1.$ For the base case $y=0$, note that \eqref{eq:more_row_order} implies that $n_{\ell}$ is inserted to the first row before $(n-e)_{\ell-k}$ for $e\geq 0$ and $k>0$, while \eqref{eq:row_order} forces $n_{\ell}$ to be inserted to row $1$ before $(n-e)_\ell$ for all $e>0$. Accordingly, $n_\ell$ is inserted to the diagonal position of row $1$ and hence never South of row $1.$ 

   For the inductive step, suppose that the statement holds for all $y<m$ and consider $n_{\ell-m}.$ If it is inserted at some point to row $m+2$, it must first be inserted to row $m+1$. Again, $n_{\ell-m}$ is inserted before $(n-e)_{\ell-m-k}$ for $e,k\geq 0$ to row $m+1$ by \eqref{eq:row_order} and \eqref{eq:more_row_order}. Furthermore, no letter $(n-e)_{\ell-k}$ for $e\geq0$ and $k<m$ is in row $m+1$, or in that row but unprimed, by the inductive hypothesis. Thus, $n_{\ell-m}$ is inserted to the diagonal of row $m+1$ and never South of it. 

\medskip
\noindent
   (2): We again proceed by induction, this time on $j$. For $j=0$, our inductive argument for item (1) proves that $1_\ell$ is inserted to the diagonal of row $1$. If $i_\ell$ is such that $i>1$ we then have that it becomes primed, while if $i=1$ it remains unprimed because $T$ is a semistandard shifted tableau. Suppose then that the statement holds for all $j<k$ and consider the letter $n_{\ell-k}.$ If $n>k+1$ and $n_{\ell-k}$ is not primed in a row $b<k+1$, then it must be inserted to row $k+1$ by hypothesis.

   Since for all $j<k$ the letter $(j+1)_{\ell-j}$ does not become primed by the inductive hypothesis, it follows from the lemma hypothesis that it is inserted to row $j+1$; furthermore, it is inserted to the diagonal position of that row by the reasoning for the inductive step within item $(1)$. Observe also that we have the insertion order
   \[ 1_\ell \rightarrow 2_{\ell-1} \rightarrow \dots \rightarrow (k+1)_{\ell-k}  \]
   because, again, none of the letters in this list are primed so that Lemma~\ref{lem:bigRowLemma} applies. Accordingly, when $(k+1)_{\ell-k}$ is inserted to a row North of row $k+1$, it occupies a non-diagonal position. Equivalently, it does not become primed North of row $k+1.$ But once in row $k+1$ we have by the reasoning of item $(1)$ that it will be inserted to the diagonal position. Finally, Lemma~\ref{lem:rowandColBoundedLemma} ensures that in this position it will not become primed. \qedhere
\end{proof}

\begin{lemma}\label{lem:entriesOfShiftedLatticeAreInsertedToRows}
    Let $\nu$ be a strict partition, $w$ a shifted lattice word with $\ct(w)=\ct(\hat{y}_\nu)$. Then, for all $r>0$,  $i \geq r$, and $j\geq r-1$, the letter $i_{\ell(\nu) - j}$ is inserted to row $r$ during the mixed insertion of $w$ or becomes primed in a row $f<r$.
\end{lemma}
\begin{proof}
    Let us prove the statement by induction on $r$. For $r=1$, we know that for all $i,j>0$ the letter $i_{\ell(\nu)-j}$ is inserted to the first row regardless of whether it is later bumped to a subsequent row. 

    Assume then that the statement holds for all $r<k$, and let $i\geq k$ and  $j\geq k-1$ be fixed. Since $i\geq k-1$ and $j\geq k-2$ as well, the inductive hypothesis implies that $i_{\ell-k}$ was inserted to row $k-1$ or became primed in a row $f< k-1.$ In the latter case, the desired conclusion follows, so let us assume that it was inserted to row $k-1.$ 
    By the same argument as that made for the proof of Lemma~\ref{lem:insertionOfShrinkingSequence}, $i_{\ell-k}$ is not inserted to the diagonal position of row $k-1$ and since $i_{\ell-k}$ is not primed, neither is $(i-1)_{\ell-k}$. Moreover, by Lemma~\ref{lem:bigRowLemma}, there are no repetitions of unprimed instances of $i$ in row $k-1$ as $i\geq k > k-1$. Hence, when $(i-1)_{\ell-k}$ is inserted to row $k-1$, which it is because of the inductive hypothesis as $i-1\geq k-1$, it is displacing $i_{\ell-k}$, or else $i_{\ell-k}$ had already been bumped to row $k$. In both cases the inductive step holds. 
\end{proof}

\begin{theorem}\label{thm:inBarelyYamanouchiClassiffshiftedYamanouchi}
    Let $\nu$ be a strict partition. Then, $w\sim \hat{y}_\nu$ if and only if
    \begin{itemize}
        \item $\ct(w)=\ct(\hat{y}_\nu)$ and
        \item the word $w$ is shifted lattice. 
    \end{itemize}
\end{theorem}
\begin{proof}
    $\implies)$ This direction is Lemma~\ref{lem:barelyYamanouchiareShiftedLattice}.
    
\medskip
\noindent

    $\impliedby$) By Lemma~\ref{lem:insertionOfShrinkingSequence} and Lemma~\ref{lem:entriesOfShiftedLatticeAreInsertedToRows} the unprimed entries in the tableaux are all $i_{\ell-j}$ such that $i\leq j+1$, while the primed entries are those satisfying $i>j+1$.

    Fix $j\geq 0$ and consider the location of the letters $(j+1)_{\ell-j}, j_{\ell-j}, \ldots, 1_{\ell-j}.$ We claim that $(e+1)_{\ell-j}$ is North of $e_{\ell-j}$ throughout the insertion of $w$ and for all $1\leq e \leq j$. First, note that by Lemma~\ref{lem:shiftedIsInterlacing} and Lemma~\ref{lem:bigRowLemma}, $(e+1)_{\ell-j}$ is inserted to all rows before $e_{\ell-j}.$ Hence, the only alternative is for $(e+1)_{\ell-j}$ and $e_{\ell-j}$ to be in the same row $f$ for some $f$ by Lemma~\ref{lem:bigRowLemma}. However, we must then have that $f=e+1$ since that is the only way that $e_{\ell-j}$ is inserted to row $f$ without bumping $(e+1)_{\ell-j}.$ But this yields a contradiction by Lemma~\ref{lem:rowandColBoundedLemma} since $e_{\ell-j}$ could not have been inserted to that row in the first place.

    Accordingly, $(j+1)_{\ell-j}$ is North of $j_{\ell-j}$, which is North of $(j-1)_{\ell-j}$, and so on, so that for all $m\leq j+1$, the letter $(m)_{\ell-j}$ is in row $m$. Indeed, by Lemma~\ref{lem:insertionOfShrinkingSequence} an unprimed entry $i_{\ell-j}$ never lies South of row $j+1$ so that this is the only possible arrangement.

    Note that our argument so far determines the unprimed entries and their rows. Observe also that the characterization of primed and unprimed entries given before implies that the unprimed entries of row $e$ consist of $\ell+1-e$ copies of the entry $e$. Since the primed entries in a row $e$ are a subset of the entries $i_{\ell-j+1}$ with $j\geq e$ and $i> e$ by Lemma~\ref{lem:insertionOfShrinkingSequence}(2) and the characterization, this means that the unprimed entries are all left of all the primed entries. Thus, the position of unprimed entries is completely determined, and matches that of the unprimed entries in $\hat{Y}_\nu.$
    
    Let us now turn our attention to the primed entries and describe their positions as identical to that of primed entries in $\hat{Y}_\nu$ by induction on $j$. For the base case, it is sufficient to note that $i_{\ell}$ is primed and in row $1$ for $i>1$ by Lemma~\ref{lem:insertionOfShrinkingSequence}. The primed entries of the $\ell$-th shrinking sequence can only be bumped eastwards, and by unprimed entries, if the order of a semistandard tableau is to be respected.

    Let the inductive hypothesis hold for all $j<k$. Again, Lemma~\ref{lem:insertionOfShrinkingSequence} determines the primed entries of the $(\ell-k)$-th shrinking sequence to be the letters $i_{\ell-k}$ with $i>k+1$; moreover, the primed entries are also forced to be in a row $r\leq k+1$ so that in particular, having values greater than $k+1$ they occur right of all unprimed entries. As the primed entries of the $(\ell-j)$-th sequence agree with their position in $\hat{Y}_\nu$ for $j<k$ by the inductive hypothesis, the position of the primed entries $i'_{\ell-k}$ in the $\ell-k$-th shrinking sequence is uniquely determined by $i_{\ell-k}'$ having to be weakly west of $(i-1)_{\ell-k+1}'$ as proved in Lemma~\ref{lem:smallSeqWestofBigSeqPrimed}; and their being a horizontal strip, as established from Lemma~\ref{lem:eastSameSeqLemma} combined with the fact that the bumping path of primed entries is northeastwards and the position of $i'_{\ell-k}$ in a column to which $(i+1)'_{\ell-k}$ has been inserted before is South of the position $(i+1)'_{\ell-k}$ had when it was there. Indeed, for $i'_{\ell-k}$ to satisfy these conditions, it must be directly below $(i+1)'_{\ell-k+1}$ if it exists and in the first row otherwise; in concordance with Lemma~\ref{lem:charScarcelyYamanouchiTableaux} and hence with $\hat{Y}_\nu.$ 

    Since $P_\mix(w)=\hat{Y}_\nu$ we have $w \sim \hat{y}_\nu.$ \qedhere
\end{proof}

Recall that our goal is to find a new shifted Littlewood--Richardson rule. By Theorem~\ref{thm:placticLRrule}, to compute $b_{\lambda, \mu}^{\nu}$ it is enough to count the number of pairs of tableaux $(T_\lambda, T_\mu)$ of shapes $\lambda$ and $\mu$ respectively, such that 
\begin{equation}\label{eq:placticMult}
    T_\lambda \cdot T_\mu = \hat{Y}_\nu 
\end{equation}
However, if we write $T_\lambda = P_\mix(w_1)$ and $T_\mu= P_\mix(w_2)$, we have that $w_1\cdot w_2 \sim \hat{y}_\nu$ and $w_2$ also follows the shifted lattice condition by Definition~\ref{def:shiftedLatticeCondition}. We can ensure that the content of $w_2$ is of a special form by the next lemma.

\begin{lemma}\label{lem:rightHasContent}
    If $e$ is a right factor of a barely Yamanouchi word, then $\ct(e)=\ct(\hat{y}_\mu)$ for some strict partition $\mu.$ 
\end{lemma}
\begin{proof}
    Read from right to left the word $e$. By Definition~\ref{def:shiftedLatticeCondition}, we have that for all $i$ the number of total occurrences of $i$ in $e$ is equal to the number of instances of $i+1$, or bigger by exactly one additional $i$. Let $k = |\{i : \#i(e)>\#(i+1)(e) \}|$, where $\#i(e)$ denotes the total number of letters with value $i$ in $e$. For $1 \leq c \leq  k$, define $c_i$ by the rule \[ c_i \coloneqq \min \{ j > c_{i-1} \: : \:  \#j(e)>\#(j+1)(e) \},\]
   where we treat $c_0 = 0$. Then, we have that $\mu \coloneqq (c_k, c_{k-1}, c_{k-2}, \ldots, c_1)$ is a strict partition. 

   Let $c_i$ be arbitrary. By definition, $\# c_i (e) > \#(c_i + 1)(e)$ so that by induction on $i$, with base case $i=k$, at least one of the sequences in the shrinking decomposition of $e$ starts with $c_i$. After removing all such sequences from $e$ we are left with a word that is again shifted lattice because of Lemma~\ref{lem:removingYamanouchi}. Repeating sequentially this process until $i=1$ we exhaust all the letters of $e$, so that $\ct(e) = \ct(\hat{y}_\mu)$.
\end{proof}

\begin{corollary}\label{cor:rightisbarelyYam}
    If $e$ is a right factor of a barely Yamanouchi word, then $e$ is a barely Yamanouchi word.
\end{corollary}
\begin{proof}
    By Theorem~\ref{thm:inBarelyYamanouchiClassiffshiftedYamanouchi} and Lemma~\ref{lem:rightHasContent}, noting that the class of shifted lattice words is transparently closed under taking right factors.
\end{proof}

\subsection{A shifted Littlewood--Richardson rule}

We now establish a shifted Littlewood--Richardson rule, which we will build upon to prove the main result, Theorem~\ref{thm:main}.

\begin{theorem}\label{thm:placticLRSophisticatedVersion}
    Let $\lambda$, $\mu$, and $\nu$ be strict partitions. Then $b_{\lambda,\mu}^\nu$ is equal to the number of tableaux $T_\lambda$ of shape $\lambda$ such that
    \[T_\lambda\cdot \hat{Y}_\mu = \hat{Y}_\nu.\]
\end{theorem}
\begin{proof}
    By Equation~\eqref{eq:placticMult}, we have that $b_{\lambda,\mu}^\nu$ is equal to the count of pairs $(T_\lambda,T_\mu)$ such that
    \[T_\lambda \cdot T_\mu = \hat{Y}_\nu. \]
    Suppose that $P_\mix(w_2)\coloneqq T_\mu$. Then, by Corollary~\ref{cor:rightisbarelyYam}, $w_2$ is barely Yamanouchi and mixed inserts to $\hat{Y}_\mu$ as this is the only barely Yamanouchi tableau of shape $\mu.$ 
    
    We can then say that $b_{\lambda, \mu}^{\nu}$ is the number of tableaux of shape $\lambda$ such that
\[ T_\lambda \cdot \hat{Y}_{\mu}  =  \hat{Y}_\nu   \]
for the fixed barely Yamanouchi tableaux $\hat{Y}_{\mu}$ and $\hat{Y}_\nu.$
\end{proof}

A natural approach at this point is to count all possible left factors by finding their mixed reading word, concatenating it with a fixed word for $\hat{Y}_\mu$, and determining if the resulting word is shifted lattice. However, finding that mixed reading word proves quite difficult in general. 
Instead, we take advantage of the fact that we can also express $b_{\lambda, \mu}^\nu$ as the number of words $w$ with insertion of shape $\lambda$ such that the equality
$ w \cdot \hat{y}_\mu = \hat{y}_\nu  $
holds in $\sPlactic$.

\subsection{Left factors of barely Yamanouchi words are interlacing}
We now turn to analyzing such left factors.

\begin{proposition}\label{prop:barelyYamhaspartitionshape}
    If $w$ is barely Yamanouchi with shrinking decomposition $\{ (d^{(k)}_{\nu_j})_k \}_{1 \leq j \leq \ell}$, then $\nu$ is a strict partition.
\end{proposition}
\begin{proof}
    By Lemma~\ref{lem:rightHasContent}, $w$ has strict content. By combining Lemmas~\ref{lem:consecutiveIs} and~\ref{lem:barelyYamanouchiareShiftedLattice}, $w$ satisfies the two bullet points of the definition of ``interlacing''. In particular, $w$ has a unique maximal letter $M$. The first shrinking sequence starts with this letter $M$ and has length $M$, since the word $w$ is classically Yamanouchi by Corollary~\ref{cor:barelyYamisYam}. After deleting the first shrinking sequence, the result is still shifted lattice by Lemma~\ref{lem:removingYamanouchi} (see also, Remark~\ref{remark:removingYamanouchi}). Again by Lemma~\ref{lem:consecutiveIs} this word has a unique (and strictly smaller) maximal letter $M'$. The next shrinking sequence starts with this letter $M'$ and has length $M'$, since the restricted word is also classically Yamanouchi by Corollary~\ref{cor:shiftedLatticeisYamanouchi}. The corollary then follows by iterating this process.
\end{proof}

\begin{corollary}\label{cor:bYam_is_interlacing}
    If $w$ is a barely Yamanouchi word, then $w$ is interlacing.
\end{corollary}
\begin{proof}
    The two bullet points of the definition of ``interlacing'' hold by Lemmas~\ref{lem:consecutiveIs} and~\ref{lem:barelyYamanouchiareShiftedLattice}. The corollary then follows from Proposition~\ref{prop:barelyYamhaspartitionshape}.
\end{proof}

\begin{proposition}\label{prop:leftImpliesInterlacing}
    If $w$ is a left factor of a barely Yamanouchi word, then $w$ is interlacing. 
\end{proposition}
\begin{proof}
    Let $wu$ be a barely Yamanouchi word. By Corollary~\ref{cor:bYam_is_interlacing}, $wu$ is interlacing. Hence, $w$ also satisfies the two bullet points in the definition of ``interlacing.'' 
     It remains to check that the shrinking decomposition of $w$ gives a difference of two strict partitions. 
     
     By Corollary~\ref{cor:rightisbarelyYam}, $u$ is barely Yamanouchi. Thus, Proposition~\ref{prop:barelyYamhaspartitionshape} gives that the shrinking sequences of both $u$ and $wu$ give strict partitions. 
    By Lemma~\ref{lem:redundant_condition}, every shrinking sequence of $wu$ consists of a shrinking sequence of $w$ concatenated onto a shrinking sequence of $u$. The lemma follows.
\end{proof}

\subsection{Completions of constructed tableaux are barely Yamanouchi}

    To improve upon Theorem~\ref{thm:placticLRSophisticatedVersion}, we now study constructed tableaux in more detail.
    
\begin{lemma}\label{lem:diagEntriesIndicesConstructed}
    Let $\mu<\nu$ be strict partitions and $T$ be constructed from $\mu$ and $\nu.$ If the diagonal entries of $T$ are $d^{(1)}_{s_1}, d^{(2)}_{s_2}, \dots, d^{(\ell(\nu))}_{s_{\ell(\nu)}}$ from top to bottom, then we have 
    \[
    s_1 > s_2 > \dots > s_{\ell(\nu)}.
    \]
\end{lemma}
\begin{proof}
    Suppose that the lemma does not hold. Then there is a smallest row $e$ such that $s_e < s_{e+1}$. By order considerations, we have $d^{(e)} < d^{(e+1)}$. Hence, the entries 
    \begin{equation}\label{eq:list}
          (d^{(e)}+1)_{s_e}, (d^{(e)}+2)_{s_e}, \dots, d^{(e+1)}_{s_e}, (d^{(e+1)}+1)_{s_e}  
    \end{equation}
    also exist in $T$ by $\mu$ and $\nu$ being strict partitions. Since $d^{(e)}_{s_e}$ is on the diagonal, all the entries of \eqref{eq:list} are primed in $T$. By order considerations, $(d^{(e+1)}_{s_e})'$ must appear east of $d^{(e+1)}_{s_{e+1}}$. Hence, $(d^{(e+1)}+1)_{s_e})'$ is East of $d^{(e+1)}_{s_{e+1}}$. Therefore, the Serrano--Pieri strip containing $d^{(e+1)}_{s_{e+1}}$ can be extended by using $(d^{(e+1)}+1)_{s_e})'$ and all larger entries of that shrinking sequence.
\end{proof}

\begin{lemma}\label{lem:insertionOfShrinkingSequenceConstructedCase}
    Let $\mu <\nu$ be strict partitions and $T$ be constructed from $\mu$ and $\nu.$ Suppose that $w$ is a shifted lattice word with $\ct(w)=\ct(\hat{y}_\nu)$. Moreover, suppose that for all $r>0$,  $i \geq r$, and $j\geq r-1$, at some point of the insertion 
     \[ T \leftarrow \seq(\mu_{\ell(\mu)}, 1) \leftarrow \seq(\mu_{\ell(\mu)-1}, 1) \leftarrow \dots \leftarrow \seq(\mu_1, 1 ) \]
    (possibly before any entry has been inserted), 
    the letter $i_{\ell(\nu) - j}$ is weakly south of row $r$, or is a primed letter in a row $f<j+1$. Then, for all $i,j>0$:
   \begin{itemize}
   \item the letter $i_{\ell(\nu)-j}$ is never inserted South of row $j+1;$
   \item if $i>j+1$, then the letter $i_{\ell(\nu)-j}$ becomes primed at some point of the insertion in a row $b\leq j+1$; and
       \item if $i\leq j+1$, then $i_{\ell(\nu)-j}$ remains unprimed throughout the insertion.
   \end{itemize}  
   \end{lemma}
\begin{proof}
    Suppose the hypotheses hold and fix $i,j > 0$.
    Let us first see that in $T$, the letter $i_{\ell-j}$ is never South of row $j+1$ whether primed or unprimed. We start by considering the position in $T$ of $i_{\ell-j}$ when unprimed.

    Let $x_c$ be the entry in the diagonal position of the same row of $T$ as $i_{\ell-j}$, where possibly $x_c=i_{\ell-j}.$ Since $T$ is constructible (Definition~\ref{def:constructibleDef}), we have $\ell-j \leq c$. Moreover, for any row North of $x_c$, the entry in its diagonal, which we denote by $y_b$, also satisfies that $c<b$ by Lemma~\ref{lem:diagEntriesIndicesConstructed}. Accordingly, if we let $k \coloneqq \ell - c$, it follows that the southernmost row $x_c$ can occupy is $k+1$ as the shrinking sequence subindex must decrease by at least one for each row and the best starting entry would be an entry in the $\ell$-th sequence. But then, $\ell-j<c$ also forces $k<j$. That is, $i_{\ell-j}$ occupies a row weakly north of row $j+1.$

    Now let us prove that the primed letter $i_{\ell-j}'$ must also be weakly north of row $j+1$ in $T$. Let $x_c$ be the diagonal entry lying in the same row as $i_{\ell-j}'$. If $\ell-j \leq c$, the argument for unprimed $i_{\ell-j}$ applies again and $i_{\ell-j}'$ is weakly north of row $j+1.$ Otherwise $c< \ell-j$. In this case, let $u_{\ell-m}'$ be the leftmost primed entry in the row of $i_{\ell-j}'$ (and of $x_c$). A direct application of Lemma~\ref{lem:constructedAncillaLemma}(3) shows that $\ell-m >\ell-j$, so that $c<\ell-m$ as well. Similarly, for all unprimed entries $e_d$ between $x_c$ and $u_{\ell-m}'$ in the same row, we have $d<c$ by Definition~\ref{def:constructibleDef}, and in particular $d<\ell-m$ also holds. Therefore, for all such unprimed entries $e_d$, where possibly $e_d=x_c$, there must exist a primed entry $(u-1)_d'$ because of Lemma~\ref{lem:constructedAncillaLemma}(1). The letter $(u-1)_d'$ is also weakly west of $u_{\ell-m}'$ by Lemma~\ref{lem:constructedAncillaLemma}, so that we require as many copies of $(u-1)'$ as unprimed entries left of $u_{\ell-m}'$ and in its same row, to lie weakly west of $u_{\ell-m}'$. Considering that $(u-1)_d'$ is further prohibited from being weakly northwest of $e_d$ by order considerations, it becomes apparent that at least two distinct copies of $(u-1)'$ must lie in the same column. This yields a contradiction by Lemma~\ref{lem:rowCharConstructedTableaux}(2).  

    With the aim of making the present lemma analogous to Lemma~\ref{lem:insertionOfShrinkingSequence} it is also important to show that for all $i,j>0$ the letter $i_{\ell-j}$ is not primed in $T$ for $i \leq j+1$. Once that is established, it is immediate that if in $T$ the letters $i_{\ell-j}, (i-1)_{\ell-j}, \dots, (i-j-1)_{\ell-j}$ all exist for some $i>j+1$, then $i_{\ell-j}$ is primed. Indeed, an unprimed instance of $i_{\ell-j}$ would need to be in a vertical strip of height at least $j+1$, contradicting that $i_{\ell-j}$ cannot be South of row $j+1.$

    We establish this statement by induction on $j$. For the base case, note that $1_\ell$ cannot be primed because $T$ is a semistandard shifted tableau. Suppose then that the statement holds for all letters $i_{\ell-j}$ with $j<k$ and $i\leq j+1$. If $i_{\ell-k}$ is unprimed for all $i$, we can conclude the inductive step, so assume $i_{\ell-k}'\in T$ for some $i\leq k+1$ and let us consider its position in $T$.
     
     Say that $\gamma^{(s)}$ is the Serrano-Pieri strip containing $i_{\ell-k}'$. Then, as $i\neq 1$ because $i_{\ell-k}$ is primed, and since $T$ is constructed from strict partitions $\mu<\nu$ where an instance of $i$, primed or not, is part of $T$; it also follows that $(i-1)_{\ell-k+1}$ is an entry of $T$. Moreover, the inductive hypothesis forces $(i-1)_{\ell-k+1}$ to be unprimed. However, this contradicts Lemma~\ref{lem:constructedAncillaLemma}(1).

    The rest of the argument is analogous to Lemma~\ref{lem:insertionOfShrinkingSequence} since the order of priming is identical by Lemma~\ref{lem:sameSeqPrimedConstructed} and Lemma~\ref{lem:weaklyWestCompletionOfConstructed}, and the order in which entries appear to a row is likewise analogous by Lemma~\ref{lem:puttingAllTogetherCompletionConstructed} Lemma~\ref{lem:transitionLemmaConstructed} and Lemma~\ref{lem:sameSeqConstructed}. 
\end{proof}

\begin{lemma}
    Let $T$ be constructed from strict partitions $\mu <\nu$ and $w$ be a shifted lattice word with $\ct(w)=\ct(\hat{y}_\nu)$. Then for all $r>0$,  $i \geq r$, and $j\geq r-1$, at some point of the insertion \[ T \leftarrow \seq(\mu_{\ell(\mu)}, 1) \leftarrow \seq(\mu_{\ell(\mu)-1}, 1) \leftarrow \dots \leftarrow \seq(\mu_1, 1 ), \] possibly before any entry has been inserted,  
    the letter $i_{\ell(\nu) - j}$ is weakly south of row $r$, or is a primed letter in a row $f<j+1$. 
\end{lemma}
\begin{proof}
    The proof is completely analogous to Lemma~\ref{lem:entriesOfShiftedLatticeAreInsertedToRows}.
\end{proof}

\begin{theorem}\label{thm:constructedIsLeft}
    Let $T$ be constructed from strict partitions $\mu < \nu$. There exists a word $w\in \cN^*$ such that
    \[ T \cdot w = \hat{y}_\nu \]
\end{theorem}
\begin{proof}
    The proof is analogous to the converse implication of Theorem~\ref{thm:inBarelyYamanouchiClassiffshiftedYamanouchi} by virtue of Lemma~\ref{lem:eastSameSeqCompletionCompletionConstructed} and Lemma~\ref{lem:noTwoRepeatedColsCompletionConstructed}.
\end{proof}

\subsection{Interlacing tableaux are left factors of barely Yamanouchi tableaux}

We previous established in Proposition~\ref{prop:leftImpliesInterlacing} that left factors of barely Yamanouchi words are interlacing. We now establish a converse.

\begin{theorem}\label{thm:interlacingIsLeft}
    Let $w$ be an interlacing word such that $\ct(w)=\ct(T)$ for some tableau $T$ constructed from strict partitions $\mu<\nu$. Then, $w$ is a left factor of the barely Yamanouchi word $\hat{y}_\nu$.
\end{theorem}
\begin{proof}
    Obtain the shrinking decomposition of $w$ and let $M$ be the greatest element of the first shrinking sequence. Let $w'$ be the word obtained by removing all letters from the first shrinking sequence. If $w'$ has more than one letter of maximum value, denote two of those by $m_i$ and $m_j$. Without loss of generality, assume $m_i$ appears left of $m_j$. In that case, it then follows from Definition~\ref{def:interlacing} that between $m_i, m_j$ in $w$ there is an $m+1$, which must be an element of $w \setminus w'$. Since $m_j$ is not part of the first shrinking sequence, the entry of the first shrinking sequence with value $m$ must appear strictly left of $m_j$. But then by Definition~\ref{def:interlacing}, $w$ has an instance of $m+1$ that is strictly right of the $m+1$ in the first shrinking sequence, contradicting the maximality of $m$ in $w'$.
     Hence, the maximum element of each shrinking sequence forms a strictly decreasing list, or equivalently, a strict partition. 

    If the minimum letter $b_i$ in a given shrinking sequence has value $b>1$, then the next shrinking sequence must have a strictly smaller minimum value. Otherwise, designating by $a_{i+1}$ the smallest element in the next sequence, there are two cases in which we obtain contradictions.

    {\sf Case 1 ($b<a$):}
    Let $B$ be the largest value in the $i$th shrinking sequence. Then $a < B$ by the above argument.
    Therefore, $(a+1)_i$ exists and by Definition~\ref{def:interlacing}, the letter $a_{i+1}$ is left of $(a+1)_i$, which is also left of $a_i$. Therefore, $a_{i+1}$ is left of $a_i$ and there is an instance of $a-1$ between both; in particular, right of $a_{i+1}.$ Let $(a-1)_j$ be the leftmost such $a-1$. If $j<i$, then using the same argument $a_{i+1}$ is left of $a_j$ and there is another instance of $a-1$ to the right of $a_{i+1},$ contradicting our choice of $(a-1)_j$ leftmost. Since $j \neq i$, we therefore have $j = i+1$, contradicting $a_{i+1}$ smallest in its shrinking sequence.

    {\sf Case 2 ($a=b$):}
    By Definition~\ref{def:interlacing}, we have $a_{i+1}$ left of $a_i$.
    Hence, there is an instance of $a-1$ between them. By the same arguments used in {\sf Case 1}, we may show that $a-1$ belongs to the $i+1$-st shrinking sequence; however, this contradicts the choice of $a_{i+1}$ minimal.
    \smallskip

    Hence, the shrinking decomposition can be represented by two strict partitions $\mu < \nu$; that is, the unprimed elements of the $k$-th sequence belong to the interval $(\mu_k, \nu_k]$. Designate now by $w^{\twiddle}$ the word \[ w^{\twiddle} \coloneqq \seq(\mu_{\ell(\mu)}) \cdot \seq(\mu_{\ell(\mu)-1}) \cdot \dots \cdot \seq(\mu_1). \]
    We want to prove that $w\cdot w^{\twiddle}$ is shifted Yamanouchi. Suppose toward a contradiction that it is not. Then, as $w^{\twiddle} = \hat{y}_\mu$, it is immediate that $w^{\twiddle}$ is itself shifted Yamanouchi, so that for some position in its left factor and $i>0$, it holds that $i(w\cdot w^{\twiddle})< (i+1)(w\cdot w^{\twiddle})$. In particular, we may assume without loss of generality that in this position $i(w\cdot w^{\twiddle}) = (i+1)(w\cdot w^{\twiddle}) -1$ for the first time in the word from right to left. 

    Accordingly, in this position we must have a letter $(i+1)_j$. We know that right of $(i+1)_j$ there must exist at least one instance of $i$ in $w\cdot w^{\twiddle}$ because of the way that $w^{\twiddle}$ is defined, and recalling that $(i+1)_j$ has to be located in $w.$ Consider then the first letter of value $i$ right of $(i+1)_j$ and denote it by $i_k$. Between $(i+1)_j$ and $i_k$ there are no entries of value $(i+1)$, as otherwise there would be an even earlier instance of $i$ because of Definition~\ref{def:interlacing}. Therefore, we obtain that in the position of $i_k$, $i(w\cdot w^{\twiddle}) = (i+1)(w\cdot w^{\twiddle}).$ But then in the position immediately following $i_k$, which exists if that equality is to be had at $i_k$, we obtain that $i(w\cdot w^{\twiddle}) < (i+1)(w\cdot w^{\twiddle})$ contradicting the choice of $(i+1)_j.$ 
\end{proof}

\subsection{Proof of Theorem~\ref{thm:main}}
We are now ready for the main theorem. First, we need the following lemma.

\begin{lemma}\label{lem:sameShrinkingLemma}
    Let $\nu$ be a strict partition and $w$ be such that $[w]=[\hat{y}_\nu]$, i.e., a barely Yamanouchi word. Then the shrinking decompositions of $w$ and $\hat{y}_\nu$ are the same. 
\end{lemma}
\begin{proof}
    Consider the shrinking decomposition of $\hat{y}_\nu$. Since $\nu$ is a strict partition, there is a unique entry of maximal value in $\hat{y}_\nu$, denoted $M_1$. After the shrinking sequence starting with $M_1$ is removed from $\hat{y}_\nu$, there is again a new unique entry of maximal value, that we refer to as $M_2$. By the same token, there are letters $M_3,M_4,\ldots,M_n$ presiding over their respective shrinking sequences, so that
    \[\bigsqcup_i \seq(M_i) \]
    exhausts all letters of $\hat{y}_\nu$, and is its shrinking decomposition. 

    Consider now the analogous process for $w$. The forward implication of Theorem~\ref{thm:inBarelyYamanouchiClassiffshiftedYamanouchi} showed that no chain of shifted equivalences, however long, exchanges an instance of $i$ with an instance of $i+1$ when adjacent to each other. But then, even though in different positions, the relative locations of any two letters in the same shrinking sequence of $\hat{y}_\nu$ will be the same after an arbitrary shifted equivalence. That is, if $i_j$ appeared in $\hat{y}_\nu$ before $(i-1)_j$, it will do so again after any application of a shifted relation. 

    Consequently, the process of decomposing $\hat{y}_\nu$ from the sequence $(M_i)_i$ applies verbatim to $w$. Hence, the shrinking decomposition of $w$ is identical to that of $\hat{y}_\nu. $
    \end{proof}

The following establishes Theorem~\ref{thm:main}, as well as the fact that constructed tableaux are exactly the same as interlacing tableaux.

\begin{theorem}\label{thm:main_fullversion}
    Let $\lambda,\mu,$ and $\nu$ be strict partitions with $|\lambda|+|\mu|=|\nu|$ and $\mu<\nu$. Then, the tableaux $T_\lambda$ of shape $\lambda$ such that 
    \[T_\lambda \cdot \hat{Y}_\mu = \hat{Y}_\nu\]
    are exactly the tableaux of shape $\lambda$ constructed from the partitions $\mu<\nu.$ 

    In particular, $b_{\lambda, \mu}^{\nu}$ equals the number of tableaux of shape $\lambda$ constructed from $\mu<\nu.$
\end{theorem}
\begin{proof}
The second sentence follows from the first by Theorem~\ref{thm:placticLRSophisticatedVersion}. It remains to establish the first sentence of the theorem.

    If $T$ is constructed from $\mu < \nu$, then $T$ is a left factor of $\hat{Y}_\nu$ because of Theorem~\ref{thm:constructedIsLeft}. Hence, we also have
    \[T\cdot \hat{Y}_\mu = \hat{Y}_\nu.\]

    Now we consider the converse direction.
    By Proposition~\ref{prop:leftImpliesInterlacing}, all left factors $w$ of the barely Yamanouchi word $\hat{y}_\nu$ are interlacing. Hence, all tableaux of shape $\lambda$ such that $T_\lambda \cdot \hat{Y}_\mu = \hat{Y}_\nu$ are interlacing tableaux. We will now show that all such interlacing tableaux are tableaux constructed from $\mu<\nu$. 
    First note that $T_\lambda \cdot \hat{Y}_\mu = \hat{Y}_\nu$ implies in particular that $\ct(T_\lambda \cdot \hat{Y}_\mu) = \ct(\hat{Y}_\nu).$ 

    Now, let $P_\mix(w)=T_\lambda$. From the hypotheses, we have that $\mu<\nu$; let $k$ be minimal such that $\mu_k<\nu_k$. We know that $[w]\cdot [\hat{y}_\mu]=[\hat{y}_\nu]$ and that the $j$-th shrinking sequence of $\hat{y}_\nu$ is $\nu_j \: \nu_j-1 \: \ldots \: 1$ for all $j$. Since $\hat{y}_\nu$ and $w\cdot \hat{y}_\mu$ must have the same shrinking decomposition by Lemma~\ref{lem:sameShrinkingLemma}, the first $k-1$ shrinking sequences of $w\cdot \hat{y}_\mu$ must all exclusively contain letters from $\hat{y}_\mu$ and match the shrinking sequences of $\hat{y}_\nu$. In particular, after removing those shrinking sequences from $\hat{y}_\nu$ and $\hat{y}_\mu$,  we are left with the words $\hat{y}_{\dot{\nu}}$ and $\hat{y}_{\dot{\mu}}$ for $\dot{\nu}=(\nu_k,\nu_{k+1}, \nu_{k+2},\ldots,\nu_{\ell(\nu)})$ and $\dot{\mu}=(\mu_k,\mu_{k+1}, \mu_{k+2},\ldots,\mu_{\ell(\mu)})$, respectively.
    
    Accordingly, as all the letters in $\hat{y}_{\dot{\mu}}$ are all smaller than $\nu_k$, the word $w$ is forced to contain an instance of $\nu_k$. The $k$th shrinking sequence of $w\cdot \hat{y}_\mu$ must then start with this $\nu_k$ in $w$ and decrease by $1$ until reaching the value $1$. Let $m$ be the smallest element of $w$ in this shrinking sequence. 
    Furthermore, $m-1=\max(\hat{y}_{\dot{\mu}})$ by Lemma~\ref{lem:Imalemma}. 

    Thus the elements of the first shrinking sequence of $w$ are the letters in $(m-1,\nu_k]=(\mu_k, \nu_k]$. Noting that the intervals $(\mu_j,\nu_j]$ are empty for $j<k$ and that the same argument applies for subsequent values of $k$, we obtain that the shrinking decomposition of $w$ is given by the sequences $(\mu_j,\nu_j]_{1 \leq j \leq \ell(\nu)}$. 

    Our next step is to show that the collection of shrinking sequences $(\mu_j,\nu_j]_{1 \leq j \leq \ell(\nu)}$ is in fact a collection of Serrano--Pieri strips making $T_\lambda$ constructible. 
    We will routinely use Proposition~\ref{prop:leftImpliesInterlacing} (that left factors of barely Yamanouchi words are interlacing) without mention. 

    For arbitrary $j$, the unprimed elements of the shrinking sequence $(\mu_j, \nu_j]$ form a vertical strip in $T_\lambda$ because of Lemma~\ref{lem:bigRowLemma}(1) and (3). (Note that part (1) applies by Lemma~\ref{lem:rowandColBoundedLemma}.) Similarly, for $j$ arbitrary the primed entries of $(\mu_j,\nu_j]$ form a horizontal strip by Lemma~\ref{lem:eastSameSeqLemma}. If we let the entries of the vertical strip correspond to the shape $\xi/\pi$ and those of the horizontal strip to the shape $\theta/\eta$, then we further obtain $\xi \subseteq \eta$ as all the primed entries of a given segment $(\mu_j,\nu_j]$ must be greater than all the unprimed entries by Lemma~\ref{lem:bigRowLemma}(3) and the definition of mixed insertion. Thus, $\xi /\pi \olessthan \theta/\eta$ is a Serrano--Pieri strip inside $T_\lambda.$

Let us now see that these Serrano--Pieri strips make $T_\lambda$ a constructed tableau. 

    Consider now two such strips $(\mu_j,\nu_j]$ and $(\mu_m,\nu_m]$ for $j<m$. For every row $r$, iterating Lemma~\ref{lem:bigRowLemma} shows that either $(i-1)_m$ is inserted before $i_j$ to $r$ , or $r=i$. In the first case, it suffices to recall that the unprimed letters in $(\mu_m,\nu_m]$ form a vertical strip to conclude that $i_j$ is east of the only unprimed entry of $(\mu_m,\nu_m]$ in its row. Otherwise, $r=i$ and the fact that $(i-1)_m$ is inserted to row $r-1$ before $i_j$, coupled with Lemma~\ref{lem:rowandColBoundedLemma}, allows us to infer that $i_m$ is inserted to row $r$ before $i_j$, and that it is the only unprimed entry of $(\mu_m,\nu_m]$ therein. Again, we conclude that $i_j$ is east of all unprimed entries in $(\mu_m,\nu_m]$. As the choice of $i$ was arbitrary throughout, the unprimed entries of $(\mu_j,\nu_j]$ occur after the unprimed entries of $(\mu_m,\nu_m]$ in all rows.  

     Similarly, either the primed entries of $(\mu_m,\nu_m]$ are all less than all the primed entries of $(\mu_j,\nu_j]$, and thus after the primed entries of $(\mu_j,\nu_j]$ in every column; or there is at least one primed letter in $(\mu_m,\nu_m]$ that has value greater than a primed entry of $(\mu_j,\nu_j]$. Let $(i+1)_j'$ be a primed entry such that there exists a primed entry greater than or equal to it in $(\mu_m,\nu_m]$. Lemma~\ref{lem:noTwoRepeatedCols} forbids $(i+1)_j'$ and $(i+1)_m'$ from being in the same column, so it is enough to restrict our attention to the case where there exists an entry $(i+k)'_m$ for $k>1$. Recall that $(i+1)_j'$ is weakly west of $i_m'$, and that $i_m'$ exists by to Lemma~\ref{lem:constructedAncillaLemma}(2). Hence, the fact that the primed entries of $(\mu_m,\nu_m]$ form a horizontal strip implies that $(i+k)_m'$ is not in the same column as $(i+1)_j'$. That is, the primed entries of $(\mu_j,\nu_j]$ happen after the primed entries coming from $(\mu_m,\nu_m]$ in every column where there is at least one of both. Thus, the Serrano--Pieri strips defined by the segments $(\mu_j,\nu_j]_{1 \leq j \leq \ell(\nu)}$ witness that $T_\lambda$ is constructed from $\mu < \nu$. 
\end{proof}

\begin{corollary}\label{cor:constructed=interlacing}
    A tableau $T$ of shape $\lambda$ is interlacing if and only if it is constructed from some pair of strict partitions $\mu < \nu$ with $|\nu| - |\mu| = |\lambda|$.
\end{corollary}
\begin{proof}
    This is immediate from the proof of Theorem~\ref{thm:main_fullversion}.
\end{proof}

\section{Application of Theorem~\ref{thm:main}}\label{sec:examples}

In this section, we explain a method for implementing Theorem~\ref{thm:main} in practice. We do so in a way that is convenient for manual calculation of shifted Littlewood--Richardson coefficients. (In Section~\ref{sec:complexity}, we introduce a slight variation that is more computer-friendly and analyze its running time, showing that is faster than Stembridge's rule \cite{Stembridge} for a family of structure constants.) A Python implementation of the rule is available at \cite{shifted-code}.

Let $\lambda,$ $\mu$, and $\nu$ be strict partitions, set $\ell = \ell(\nu)$, and suppose that we wish to find the structure coefficient $b_{\lambda,\mu}^\nu$. We will determine it by counting all possible fillings of $D_\lambda$ with the entries of $\cup_{i} (\mu_i, \nu_i]$, some possibly primed, satisfying certain conditions. Throughout the process, the addition of an entry must result in a partial shifted tableau contained in $D_\lambda$. 

\begin{algorithmic}
    \FOR{$i = 1$ to $\ell$}
        \STATE Choose $m_i \in \mathbb{N}$ such that $\mu_{\ell-i+1} \leq m_i \leq \nu_{\ell-i+1}$, ensuring:
        \IF{$\theta^{(i-1)}/\eta^{(i-1)}$ is non-empty}
            \STATE $m_i+1$ is strictly greater than the smallest primed entry in $\theta^{(i-1)}/\eta^{(i-1)}$.
        \ELSE
            \STATE $m_i+1$ is at least two greater than the largest entry in $\xi^{(i-1)}/\psi^{(i-1)}$.
        \ENDIF

        \STATE Successively place entries of $(\mu_{\ell-i+1}, m_i]$ as unprimed entries in $D_\lambda$ ensuring:
        \STATE \hspace{1em} $\bullet$ The placement forms a vertical strip $\xi^{(i)}/\pi^{(i)}$ where $\xi^{(e)} \subseteq \pi^{(i)}$ for all $e<i$.
        \STATE \hspace{1em} $\bullet$ If placing in a cell $\bbb$ with a primed entry, displace all entries weakly east of $\bbb$ in its row one column to the right (moving up a row if the box above is empty) or displace all entries weakly south of $\bbb$ in its column one row down (moving left if the boxes to the left are empty).
        \STATE \hspace{1em} $\bullet$ Each $j \in \xi^{(i)}/\pi^{(i)}$ is weakly north of $j-1$ from any $\gamma^{(k)}$ with $k < i$, unless $j-1$ is in row $j-1$, in which case $j$ may be in the next row.

        \STATE Successively place entries of $(m_i, \nu_i]$ as primed entries in $D_\lambda$ ensuring:
        \STATE \hspace{1em} $\bullet$ The placement forms a horizontal strip $\theta^{(i)}/\eta^{(i)}$ where $\xi^{(e)}, \theta^{(e)} \subseteq \eta^{(i)}$ for all $e<i$.
        \STATE \hspace{1em} $\bullet$ Each $j' \in \theta^{(i)}/\eta^{(i)}$ is weakly west of $(j-1)' \in \theta^{(i-1)}/\eta^{(i-1)}$ if it exists.
    \ENDFOR
\end{algorithmic}

We proceed to consider two examples, but let us note before that Lemma~\ref{lem:constructedSmallChar} guarantees that the output of this algorithm is equal to the number of tableaux constructed from $\mu<\nu.$ Observe that checking if any shrinking sequence inside the tableau extends a previous sequence was reduced to verifying that for any sequence, it does not extend the previous sequence. Furthermore, to decide if an element of a sequence is extending the previous sequence, it suffices to look at exactly one element from the previous sequence. 

\begin{example}\label{ex:Example}
    Let $\nu=(11,9,5)$, $\mu=(4,2)$, and $\lambda=(8,7,4)$. We need to place the elements of 
    \[\bigcup_{i}(\mu_i,\nu_i] = (0,5] \cup (2,9] \cup (4,11]\]
    according to our rule.

    We start by adding the letters of $(0,5]$. We must insert an initial interval of these letters to form a vertical strip, and then the remaining letters to form a horizontal strip. A choice of initial interval is equivalent to a choice of $m_1$ in the language of the algorithm. 
    
    In this case, we add first the letters in $(\mu_\ell, m_1]=(0,1]=\{1\}.$ There is only one way to insert them and have a valid shifted tableau. 

    Next, add the elements in $(m_1,\nu_\ell]$ so that the resulting shape is a valid shifted shape, and they form a horizontal strip inside it. Again, there is a unique way of doing so, illustrated in Figure~\ref{fig:ins1}.

        \begin{figure}[htbp]
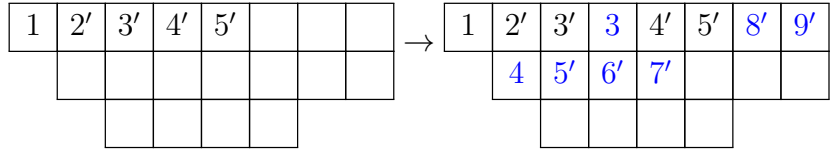
    
       \ytableausetup{boxsize =normal}
           \[ \begin{ytableau}
        1   &2' & 3' & 4'              & 5' & & &\\
        \none& & & & & &   &    \\
        \none&\none            &                 & &             &   
    \end{ytableau}  \rightarrow  \begin{ytableau}
        1   &2' & 3' &{\color{blue}{3}} & 4'              & 5' &{\color{blue}{8'}}&{\color{blue}{9'}}\\
        \none&{\color{blue}{4}}& {\color{blue}{5'}}& {\color{blue}{6'}}&{\color{blue}{7'}}&   &  &    \\
        \none&\none            &               & &               &   
    \end{ytableau}  \]
        \caption{Insertion of the first two shrinking sequences. }
        \label{fig:ins1}
    \end{figure}

    Now, we consider the letters of $(\mu_2,\nu_2]=(2,9].$ This time there are two ways to choose an initial interval and assign its letters to a vertical strip. We opt for the interval $(\mu_2,m_2] = (\mu_2, 4]=(2,4]$, and place $3$ in the first row, $4$ in the second. Since there were primed entries from the first sequence in that row before the placement of $3$, it displaces them eastwards. 

    The letter $4$ from the previous sequence is primed, so $3$ from the new sequence cannot extend it; as is $5$ from the previous sequence so that $4$ from the new one cannot either. 

    The addition of the letters in $(m_2,\nu_2]$ is performed so that they are primed and constitute a horizontal strip inside $D_\lambda.$ We must be careful so that each entry $i$ is inserted weakly west of $(i-1)'$ from the previous sequence. The new sequence is so inserted then, that it does not extend the first sequence.

 Finally, we turn our attention to the position of the letters in $(\mu_1,\nu_1]$. We decide to extract letters from $(\mu_1,m_3]=(\mu_1,6]$ for the unprimed vertical strip, and the remaining letters for the horizontal primed strip. 

 Note that the position of $5$ is weakly south of $4$ from the previous sequence. The letter $5$ from the previous sequence is primed, hence $6$ from the new sequence does not extend it. Again, when inserting the primed letters we just place $i'$ so that it is weakly west of $(i-1)'$ from the previous sequence. The outcome of the process is the constructed tableau:
       \ytableausetup{boxsize =normal}
           \[   \begin{ytableau}
        1   &2' & 3' &{\color{blue}{3}} & 4'              & 5' &{\color{blue}{8'}}&{\color{blue}{9'}}\\
        \none&{\color{blue}{4}}& {\color{blue}{5'}}&{\color{red}{5}}&{\color{blue}{6'}}&{\color{blue}{7'}}&   {\color{red}{10'}}  &{\color{red}{11'}}     \\
        \none&\none            & {\color{red}{6}}                &{\color{red}{7'}}&   {\color{red}{8'}}            &  {\color{red}{9'}}  
    \end{ytableau}. \]

    Using our rule, we also obtain the following tableaux constructed from $\mu<\nu$ and of shape $\lambda$:
    \[
            \begin{ytableau}
                1   &2' & 3' &{\color{blue}{3}} & 4'              & 5' &{\color{red}{5}}&{\color{blue}{9'}}\\
                \none&{\color{blue}{4}}& {\color{blue}{5'}}&{\color{blue}{6'}}&{\color{blue}{7'}}& {\color{blue}{8'}} &   {\color{red}{10'}}  &{\color{red}{11'}}     \\
                \none&\none            & {\color{red}{6}}                &{\color{red}{7'}}&   {\color{red}{8'}}            &  {\color{red}{9'}}  
            \end{ytableau} \quad \quad \quad \quad \quad 
            \begin{ytableau}
                1   &2' & 3' &{\color{blue}{3}} & 4'              & 5' &{\color{red}{5}}&{\color{blue}{9'}}\\
                \none&{\color{blue}{4}}& {\color{blue}{5'}}&{\color{blue}{6'}}& {\color{red}{6}} &{\color{blue}{7'}}& {\color{blue}{8'}} &   {\color{red}{11'}}    \\
                \none&\none            & {\color{red}{7}}                &{\color{red}{8'}}&   {\color{red}{9'}}            &  {\color{red}{10'}}  
            \end{ytableau}
    \]

   The tableau on the right is such that an entry from a previous segment, $8'$, occurs south of an entry from a later segment, $5$. This is in agreement with the procedure to gradually build a constructed tableau, as $8'$ is primed and $5$ is unprimed.

   If the situation were the opposite, that is, if a primed entry from a later sequence were north of an unprimed entry from a previous sequence; then the tableau would not be constructed. 

   As we were able to find $3$ tableaux of shape $\lambda$ constructed from $\mu<\nu$, we learn that $b_{\lambda,\mu}^\nu = 3.$
\end{example}

\begin{example}
    Let $\lambda$, $\mu$, and $\nu$ be as for Example~\ref{ex:Example}. One can validly produce the partial tableau
    \[\begin{ytableau}
        1   &2' & 3' & 4'              & 5' &{\color{blue}{7'}}& {\color{blue}{8'}}  & {\color{blue}{9'}} \\
        \none&{\color{blue}{3}}& {\color{blue}{4'}} & {\color{blue}{5'}}& {\color{blue}{6'}} & & &     \\
        \none&\none            &             & &              &  
    \end{ytableau}\]
    as part of the process to build a constructed tableau. However, regardless of how the entries in the last shrinking sequence are added to it, the resulting tableau will not be constructed. 
\end{example}

For another example of the application of Theorem~\ref{thm:main}, see \cite[\S 5]{EstupinanSalamanca.Pechenik:FPSAC}.

\section{A new proof of the shifted Pieri rule}\label{sec:Pieri}
The classical Pieri rule is the following. 

\begin{theorem}[Pieri rule]\label{thm:pieri}
    The Littlewood--Richardson coefficient $c_{(p),\mu}^\nu$ in
    \[
    s_{(p)} \cdot s_\mu = \sum_{\nu} c_{(p),\mu}^\nu s_\nu
    \]
    is 
    \[
    c_{(p),\mu}^\nu = \begin{cases}
        1, & \text{if } \nu / \mu \text{ is a horizontal strip of length } p; \\
        0, & \text{otherwise.}
    \end{cases}
    \]
\end{theorem}

Theorem~\ref{thm:pieri}
shows that $s_\nu$ appears in the expansion of $s_\lambda \cdot s_\mu$ if and only if $\nu/\mu$ is a horizontal strip. In the shifted setting, a similar description applies, with the notion of semi-horizontal strip playing a comparable---if not completely analogous---role. 

\begin{definition}
    A \emph{rim} is a skew shifted shape $\nu/ \mu$ for some strict partitions $\mu<\nu$ containing no $2 \times 2$ square of boxes.
\end{definition}

The shifted Pieri rule is originally due to H.~Hiller and B.~Boe \cite{Hiller.Boe}. We state an equivalent version in terms of the combinatorial framework of \cite{Stembridge}. For completeness, we include a sketch of how to extract this Pieri rule from Stembridge's general formula.

\begin{theorem}[Shifted Pieri rule \cite{Hiller.Boe}]\label{thm:shifted_Pieri}
    Consider strict partitions $\mu<\nu$. Further, suppose that $|\nu| = |\mu| + p$ and let $c$ denote the number of columns in $\nu/\mu$ with more than one cell. Then,
    \[   
b_{(p),\mu}^{\nu} = 
     \begin{cases}
       2^{\ell(\nu)-c-1}, &\quad\text{if $\nu/\mu$ is a rim of size $p$;}\\ 
       0, &\quad\text{otherwise.}
     \end{cases}
\]
\end{theorem}
\begin{proof}[Proof of Theorem~\ref{thm:shifted_Pieri} via \cite{Stembridge}]
    Suppose that $\nu/\mu$ is a rim. Then, Stembridge's rule implies that $b_{\lambda,\mu}^{\nu}$ is the number of skew tableaux $\nu/\mu$, with the first entry in reading order unprimed, and having entries $1$ or $1'$. Subject to that restriction on the content, no two adjacent rows contain a box of cells, i.e., $2$ adjacent cells immediately below $2$ other adjacent cells. Hence, only rims can be lattice tableaux.

    Given one such rim $\nu/\mu$, the requirement imposing that the starting entry of the southmost row be unprimed, determines the entries of the first row: all are forced to be $1.$ For the remaining rows, if the leftmost cell has a cell immediately below it, we must have a $1'$ in it; otherwise it is free to contain a $1$ or a $1'$. And again, after the first cell of each row has been determined, the remaining entries are easily seen to be instances of $1$ by order considerations.

    Note now that the number of cells without a cell below it, is equal to $\ell(\nu)-c-1$. As the number of lattice tableaux is a function of one of two choices made for each of those cells, we find that it is equal to $2^{\ell(\nu)-c-1}.$ 
\end{proof}

Using our shifted Littlewood--Richardson rule (Theorem~\ref{thm:main}), we can give a new proof. 

\begin{proof}[Proof of Theorem~\ref{thm:shifted_Pieri} via Theorem~\ref{thm:main}]
By Theorem~\ref{thm:main}, the Pieri coefficient $b_{(p),\mu}^{\nu}$ is the number of tableaux of shape $(p)$ constructed from $\mu < \nu$.

    First we see that a requirement for the multiplicity to be nonzero is for the shape $\nu/\mu$ to be a rim.
    
    Let $j$ be arbitrary with the property that $\nu/\mu$ contains a $2 \times 2$ square in rows $j$ and $j+1$. This means that $(\mu_j, \nu_j]\cap (\mu_{j+1}, \nu_{j+1}] \neq \emptyset$. Let $m =  \min (\mu_j, \nu_j]\cap (\mu_{j+1}, \nu_{j+1}]$ be the smallest integer in the intersection. 

    Consider now the placement of the entries of $(\mu_j,\nu_j]$ inside a constructed tableau of shape $\lambda=(p).$ By Definition~\ref{def:constructedDef}, those interval entries are only inserted after the letters of $(\mu_{j+1},\nu_{j+1}]$ have all been placed. Accordingly, $m_{j+1}$ is already in the partial tableau when they are being placed, and as $m_j$ must be inserted primed or unprimed in cells of $\lambda=(p)$, the next letter, $(m+1)_j$, must be primed when it is placed. There is indeed such a next letter because $\nu$ is a strict partition. But then $(m+1)'_j$ can be employed to extend $m_{j+1}.$ 

    Let us then assume that $\nu/\mu$ is a rim and show that the multiplicity matches $2^{\ell(\nu)-c-1}$. First, recall that this is the total number of cells without a cell below it. It therefore suffices to prove that there are two ways to place $(\mu_j,\nu_j]$ in a constructed tableaux where the previous shrinking sequences have already been inserted, if and only if the cells of the $j$-th row have no cells below them. 

    If they have no cells below them, then the least element of $(\mu_j,\nu_j]$ is strictly greater than the greatest value of $(\mu_{j+1},\nu_{j+1}]$. No matter the way in which the letters of the first interval are marked when inserted, they will never extend the second interval. That is, the resulting tableaux will be constructed in all circumstances. Since there are two ways to mark the smallest entry of $(\mu_j,\nu_j]$, and the rest of the letters must be marked when inserted, this results in two constructed tableaux.

    If the cells of the $j$-th row have cells below them, we know that exactly one such cell has a cell below it, for $\nu/\mu$ is assumed to be a rim. But then, the smallest entry of $(\mu_j,\nu_j]$ is exactly one greater than the greatest entry of $(\mu_{j+1},\nu_{j+1}]$. Marking the smallest entry would then result in a configuration where one of the sequences can be extended. Hence, the placement of the first entry of the $j$-th row is determined and as before, the remaining letters in the same row must be primed. Consequently, only one way of inserting the letters is possible for this case. 
\end{proof}

\section{Some remarks on computational complexity}\label{sec:complexity}
In this section, we discuss the computational complexity of the shifted Littlewood--Richardson rule of Theorem~\ref{thm:main} and compare it with that of Stembridge \cite{Stembridge}.

To carry out this analysis, we need an explicit algorithm for implementing the rule of Theorem~\ref{thm:main}. We now describe such an algorithm. This algorithm could be optimized further, but is sufficient for the current purposes. An unoptimized Python implementation is available at \cite{shifted-code}.

\subsection{An algorithm}
The inputs of the algorithm are the three strict partitions $\lambda, \mu$, and $\nu$. Let $D_\lambda$ denote the Young diagram of $\lambda$.

We start by constructing a $2$-dimensional array \textbf{\textit{tableau}} with an item $0$ for every coordinate $(i,j)$ inside $D_\lambda$. We will then place the entries of the strips $\cup_{i} (\mu_i, \nu_i]$ into an empty array, eventually producing an array $T_\lambda$, representing a constructed tableau of shape $\lambda$. We will do this in all possibly ways, thereby producing the set of all constructed tableaux contributing to the structure constant according to Theorem~\ref{thm:main}.

Recall that the number of shrinking sequences in a tableau constructed from $\mu<\nu$ is given by the number of parts of $\nu$. (The coordinates employed in the algorithm are in the ``English'' matrix notation convention, so that row $1$ is on the top.)

The array \textbf{\textit{pos}} is initialized and modified so that it first contains all free positions in the first row, then those of the second, and so on. 
It is assumed that when the method is first called, it is called with tableau as a $2$-dimensional array \textbf{\textit{tableau}} with an item $0$ for every coordinate $(i,j)$ inside the partition $\lambda$.

\begin{algorithm}[htbp]
\caption{\textsf{LR-RULE} (\textbf{Inputs:} Partitions $\lambda, \mu, \nu$)}\label{alg:lr-rule} 
\begin{algorithmic}
    \STATE {\bf initialize} \textbf{\textit{count}} $= 0$ and an empty array $T_\lambda$.
    \STATE {\bf initialize} a 2-dimensional array \textbf{\textit{tableau}} with an item $0$ for every coordinate $(i,j)$ inside $D_\lambda$.
    \STATE {\bf initialize} an array \textbf{\textit{pos}} with all the coordinates of $0$-cells of \textbf{\textit{tableau}}, listed inside element arrays corresponding to their rows:
    \STATE \hspace{1em} $\bullet$ The coordinates of $0$-cells in the first row form a single list inside \textbf{\textit{pos}}.
    \STATE \hspace{1em} $\bullet$ The coordinates in each row are listed from left to right.
    \STATE {\bf set} $L \coloneqq \ell(\nu)$ and $m \coloneqq \mu_L +1$. 
    \STATE {\bf initialize} a boolean array \textbf{\textit{primed}} with $L$ elements of value \FALSE.
    \STATE {\bf call} \textsf{CONSTRUCTOR} with inputs $\lambda, \mu, \nu$, $T_\lambda$, \textbf{\textit{pos}}, \textbf{\textit{primed}}, \textbf{\textit{count}}, $L$, and $m$.
    \RETURN \textbf{\textit{count}}
\end{algorithmic}
\end{algorithm}
\medskip

To run Algorithm~\ref{alg:lr-rule}, we will need the internal method \textsf{CONSTRUCTOR}, described in Algorithm~\ref{alg:constructor}. 
 
\begin{algorithm}[htbp]
\caption{\textsf{CONSTRUCTOR} (\textbf{Inputs:} Partitions $\lambda, \mu, \nu$, arrays $T_\lambda$, \textbf{\textit{pos}},  \textbf{\textit{primed}}, and integers \textbf{\textit{count}}, $L$, $m$)}\label{alg:constructor} 
{\scriptsize
\begin{algorithmic}
    \FOR{$i = m$ to $\nu_1$}
        \FOR{$j = L$ down to $1$}
            \IF{$i \in (\mu_j, \nu_j]$}
                \FOR{$r = 0$ to $\text{len}(\textbf{\textit{pos}}) - 1$}
                    \STATE {\bf delete} \textbf{\textit{pos}}$[r][0] = (i,j)$ from \textbf{\textit{pos}}$[r]$
                    \STATE Verify that $j \geq i+1$ and $j \leq \lambda[r+1]$. \algcomment{This guarantees that the addition of a new position results in a shape that is properly contained in $\lambda.$}
                    \STATE Verify that inserting $i'_u$ in $(i,j)$ extends the horizontal strip formed by the previous entries of the $u$th shrinking sequence.
                    \STATE Verify that inserting $i'_u$ does not extend prior shrinking sequences. \algcomment{This involves four direct checks as per Lemma~\ref{lem:constructedSmallChar}.}
                    \IF{all verifications succeed}
                        \IF{\textbf{\textit{pos}}$[r+1]$ exists}
                            \STATE {\bf append} $(i+1,j)$ to \textbf{\textit{pos}}$[r+1]$
                        \ELSE
                            \STATE {\bf append} $[(i+1,j)]$ to \textbf{\textit{pos}}
                        \ENDIF
                        \IF{$(i,j) = (1,j)$ and $(i,j)$ was the last entry of \textbf{\textit{pos}}$[r]$}
                            \STATE {\bf append} $(i,j+1)$ to \textbf{\textit{pos}}$[r]$
                        \ENDIF
                        \STATE {\bf set} $T_\lambda[\textbf{\textit{pos}}[r][0]] = i'_u$
                        \STATE {\bf initialize} \textbf{\textit{primed}}$'$ as a copy of \textbf{\textit{primed}} with $\textbf{\textit{primed}}'[u] = \text{True}$.
                        \IF{$L= 1$}
                            \STATE {\bf set} $L = \ell(\nu)$ and $m = i+1$
                        \ELSE 
                            \STATE {\bf set} $L = u-1$ and $m = i$
                        \ENDIF
                        \STATE {\bf call} \textsf{CONSTRUCTOR} with inputs $\lambda, \mu, \nu$, the modified $T_\lambda$, \textbf{\textit{pos}}, \textbf{\textit{primed}}, \textbf{\textit{count}}, $L$, and $m$
                    \ENDIF
                    \IF{\textbf{\textit{primed}}$[j]$ is \FALSE}
                        \STATE {\bf delete} \textbf{\textit{pos}}$[r][0] = (i,j)$ from \textbf{\textit{pos}}$[r]$
                        \STATE Verify that $j \geq i+1$ and $j \leq \lambda[r+1]$. \algcomment{This guarantees that the addition of a new position does not result in a shape that is not properly contained in $\lambda.$}
                        \STATE Verify that inserting $i_u$ in $(i,j)$ extends the vertical strip formed by the previous entries of the $u$th shrinking sequence.
                        \STATE Verify that inserting $i_u$ does not extend prior shrinking sequences. \algcomment{This involves four direct checks as per Lemma~\ref{lem:constructedSmallChar}.}
                        \IF{all verifications succeed}
                            \IF{\textbf{\textit{pos}}$[r+1]$ exists and $(i+1,j)$ is not its final entry}
                                \STATE {\bf append} $(i+1,j)$ to \textbf{\textit{pos}}$[r+1]$
                            \ELSE
                                \STATE {\bf append} $[(i+1,j)]$ to \textbf{\textit{pos}}
                            \ENDIF
                            \IF{$(i,j) = (1,j)$ and $(i,j)$ was last in \textbf{\textit{pos}}$[r]$}
                                \STATE {\bf append} $(i,j+1)$ to \textbf{\textit{pos}}$[r]$
                            \ENDIF
                            \STATE {\bf set} $T_\lambda[\textbf{\textit{pos}}[r][0]] = i_u$
                            \IF{$L= 1$}
                                \STATE {\bf set} $L = \ell(\nu)$ and $m = i+1$
                            \ELSE 
                                \STATE {\bf set} $L = u-1$ and $m = i$
                            \ENDIF
                            \STATE {\bf call} \textsf{CONSTRUCTOR} with inputs $\lambda, \mu, \nu$, the modified $T_\lambda$, \textbf{\textit{pos}}, \textbf{\textit{primed}}, \textbf{\textit{count}}, $L$, and $m$.
                        \ENDIF
                    \ENDIF
                \ENDFOR
            \ENDIF
        \ENDFOR
    \ENDFOR
    \IF{\textbf{\textit{pos}} is empty}
        \STATE {\bf set} \textbf{\textit{count}} $= \textbf{\textit{count}} +1$
    \ENDIF
\end{algorithmic}
}
\end{algorithm}

\subsection{Complexity analysis}

We claim that all the steps of Algorithm~\ref{alg:lr-rule} (\textsf{LR-RULE}) are $\mathcal{O}(|\lambda|^2)$, except for the number of times that Algorithm~\ref{alg:constructor} (\textsf{CONSTRUCTOR}) is called, which may be exponentially-many times. This latter number may be obtained as a function of $|\SYT(\lambda)|$, the number of standard Young tableaux of shifted shape $\lambda$. 

We will argue that there is an injection from the set of partial tableaux $T_\lambda$ produced by \textsf{LR-RULE} that are not further enlarged (either because they have shape $\lambda$ or because no consistent extension is possible), into $\SYT(\lambda)$. Hence, the number of calls of \textsf{CONSTRUCTOR} is polynomial in the number of such tableaux.

To be precise, \textsf{CONSTRUCTOR} will be called every time an entry is added to a tableau whose process is ultimately successful; but also throughout the construction of a tableau whose last insertion does not extend the corresponding vertical or horizontal strip, or extending one nonetheless fails the test set by Lemma~\ref{lem:constructedSmallChar}. Note that when a certain entry cannot be inserted, the construction of its tableau stops {\it ipso facto}. This means that the final tableau would not have shape $\lambda$, making it difficult to represent it as a member of $\SYT(\lambda)$. For the sake of the injection, we therefore find convenient to allow \textsf{CONSTRUCTOR} to proceed as though the insertion of the final entry was valid. In like manner, supposing it were to fail at a later point due to a different entry, we allow it to insert the entry and stop only when the tableau is complete.

The family of tableaux coming from a successful application of the algorithm or from the extension just described, can then be standardized so that entries indexed by greater shrinking sequences are smaller than those indexed by smaller shrinking sequences; e.g., the value of $i_j$ after standardization is smaller than that of $i_{j-1}$. The resulting tableau respects the column and row order because in \textsf{CONSTRUCTOR}, entries with greater subindices are inserted first to every row and every column. Note that the resulting standard tableaux may contain primed entries off the diagonal. 

The time complexity of the algorithm is bounded above by the time complexity of its modified version, permitted to neglect one or more of the verifications to insert an entry. Indeed, every part of the process can now be understood as either a step to initialize \textbf{\textit{pos}}, \textbf{\textit{primed}}, and the remaining global variables; or a step in the construction of one and only one of the resulting tableaux (not all of them constructed). Hence, it suffices to consider the time complexity of constructing one such configuration and multiplying it by $|\SYT(\lambda)|.$ 

First, the  initializing \textbf{\textit{pos}}, an array of size $\lambda_1$, requires at most $\mathcal{O}(\sqrt{|\lambda|})$ steps. The number $L$ of shrinking sequences is bounded by $|\lambda|$ so that initializing \textbf{\textit{primed}} and \textbf{\textit{pos}} requires at most $\mathcal{O}(|\lambda|)$ steps. 

Now, for every element in $\cup_i (\mu_i,\nu_i]$, whose size is exactly $|\lambda|$, we can calculate the time complexity corresponding to its placement inside $T_\lambda$. Placing the elements of value $m$ and those of generic value is carried out separately throughout the procedure. However, as the generic case involves more (and more costly) steps than the case of $m$, it is sufficient to calculate the complexity associated to adding an entry in the generic case, and multiplying it by $|\lambda|.$

To add a generic entry, we must first delete an initial element of \textbf{\textit{pos}}; this step is naturally linear in time. Verifying that $j\geq i+1$ and $j\leq \lambda[r+1]$ is likewise linear. Whether the position $(i,j)$ would extend the vertical strip formed by the entries of the $u$th shrinking sequence can be ascertained from the position of the last entry belonging to that shrinking sequence. Finding that entry requires at most $|\lambda|$ steps, and deciding if $(i,j)$ is southwest of it is linear. Hence this step requires at most $\mathcal{O}(|\lambda|)$ steps. An analogous analysis for the insertion of a primed entry, and determining if it extends the horizontal strip formed by previous entries, shows that this process also requires $\mathcal{O}(|\lambda|)$ steps.  

Checking whether a copy of $i_j$ in the cell $(i,j)$ would be extending a shrinking sequence of greater index is equivalent to the corroboration of conditions (1)--(4) of Lemma~\ref{lem:constructedSmallChar}. They all require the position of $(i-1)_{j+1}$, which can be found after at most $|\lambda|$ steps, and the remaining verifications are linear. Thus, seeing if the insertion of $i_j$ would result in a tableau that is not constructed is $\mathcal{O}(|\lambda|)$. 

The corresponding modifications to \textbf{\textit{pos}}, $T_\lambda$, and \textbf{\textit{primed}} are linear in time. Creating a copy of primed, the array \textbf{\textit{primed}}$'$, is at most $\mathcal{O}(|\lambda|).$ 

Consequently, the cost of inserting one entry is $\mathcal{O}(|\lambda|)$ and for every such configuration there are $|\lambda|$ entries. We obtain the following result.
\begin{theorem}
    \pushQED{\qed}
    The time complexity of Algorithm~\ref{alg:lr-rule} (\emph{\textsf{LR-RULE}}) is bounded above by \begin{equation*} |\SYT(\lambda)|\cdot \mathcal{O}( |\lambda|^2).  \qedhere \end{equation*} \popQED
\end{theorem}

\subsection{Comparison with Stembridge's rule in a particular case}

Now consider the case where $\mu$ and $\nu$ are such that $\nu \coloneqq \mu+\vec{1}=(\mu_1+1,\mu_2+1,\ldots,\mu_{\ell(\mu)}+1)$ and $\lambda$ has staircase shape. We analyze the rule from \cite{Stembridge} without recalling the details, for which see \cite{Stembridge}.

The time complexity of Stembridge's rule is bounded below by $\mathcal{O}(|\SYT(\nu/\mu)|)\cdot |\lambda|$, since at least it is necessary to produce all skew tableaux of shape $\nu/\mu$ before checking whether they satisfy the lattice condition, and constructing one such tableau requires at least linear time in $|\nu / \mu| = |\lambda|$. By the choice of $\mu$ and $\nu$, the cardinality of $\SYT(\nu/\mu)$ is found to be
\[ \binom{|\lambda|}{\lambda_1, \lambda_2, \ldots, \lambda_{\ell(\lambda)}} = \frac{|\lambda|!}{\lambda_1!\cdot \dots \cdot \lambda_{\ell(\lambda)}!}\] 
modulo priming of the entries. (This follows for example from the general formula of \cite{Naruse.Okada}, but in this case can be computed easily from first principles.) On the other hand, the complexity of our rule from Theorem~\ref{thm:main} is a function of $|\SYT(\lambda)|$, which by the shifted hook formula \cite{Thrall} (see also, \cite{Sagan:random}) is
\[\frac{|\lambda|!}{\prod_{i\in D_\lambda}h(i)}\]
again, modulo priming of the entries; where $h_i$ is the classical hook of a cell inside the symmetric counterpart of the shifted diagram $D_\lambda$. Note now that for positions $i$ strictly above the diagonal of $D_\lambda$,
\[h(i) \geq 2\cdot a(i),     \]
where $a(i)$ denotes the number of cells weakly east of, i.e.\ including, $i$. Similarly, for cells on the diagonal one has the lower bound:
\[h(i)\geq 2a(i)-1.\]
In fact, the stronger statement of equality is true in this case. Accordingly, we have
\begin{align*}
    \prod_{i\in D_\lambda}h(i)  &=  \left( \prod_{(j,j) \: : \: j\in [1,\ell(\lambda)]}h(j,j) \right)\left( \prod_{(c,d)\in D_\lambda \: : \: c\neq d } h(c,d) \right) \\
                                                                &\geq   \left( \prod_{(j,j) \: : \: j\in [1,\ell(\lambda)]}a(j,j) \right)\left( \prod_{(c,d)\in D_\lambda \: : \: c\neq d } 2a(c,d) \right) \\
                                                                &= 2^{|\lambda|-\ell(\lambda)} \prod_{i\in D_\lambda}a(i),
\end{align*}
as $|\lambda|-\ell(\lambda)$ corresponds to the number of cells in $D_\lambda$ strictly above the diagonal. Observing that the values of $a(i)$ over every row correspond to the factorial of the elements in the row,
\[\frac{|\lambda|!}{\prod_{i\in D_\lambda}h(i)} \leq \frac{|\lambda|!}{2^{|\lambda|-\ell(\lambda)}  \lambda_1!\lambda_2!\ldots \lambda_{\ell(\lambda)}!} \leq  \frac{1}{2^{|\lambda|-\ell(\lambda)}} \binom{|\lambda|}{\lambda_1, \lambda_2, \ldots, \lambda_{\ell(\lambda)}}.  \]
Accordingly, $|\SYT(\lambda)| \leq \frac{1}{2^{|\lambda|-\ell(\lambda)}} |\SYT(\nu/\mu)| $, using that $\lambda$ has staircase shape.

Taking now into account the primed entries, a tableau $T\in \SYT(\lambda)$ may contain such entries everywhere apart from the diagonal. A tableau $S\in \SYT(\nu/\mu)$ on the other hand, is free to have any entries primed except possibly its southernmost entry. We obtain the sharper estimate
\[|\SYT(\lambda)| \leq \frac{1}{2^{|\lambda|-\ell(\lambda)}} |\SYT(\nu/\mu)|\cdot \frac{1}{2^{\ell(\lambda)-1}} = \frac{1}{2^{|\lambda|-1}} |\SYT(\nu/\mu)|. \]

Hence, the time complexity of our algorithm is bounded above by
\[|\SYT(\lambda)|\cdot \mathcal{O}(|\lambda|^2) \leq \frac{1}{2^{|\lambda|-1}} |\SYT(\nu/\mu)| \cdot \mathcal{O}(|\lambda|^2),\]
so that for $|\lambda| \gg 0$, the time complexity of our rule, Theorem~\ref{thm:main}, is less than than of Stembridge in these cases. 
    
\section*{Acknowledgements}
The authors are grateful for useful conversations with Olya Mandelshtam,  Alejandro Morales, Colleen Robichaux, Bruce Sagan, Kartik Singh, Jer\'onimo Valencia-Porras, and Dave Wagner. We also thank Andrew Naguib for reading and commenting on a draft of Section~\ref{sec:complexity}.

Both authors were partially supported by a Discovery Grant (RGPIN-2021-02391) and Launch Supplement (DGECR-2021-00010) from
the Natural Sciences and Engineering Research Council of Canada. SES also acknowledges partial support from a William Tutte Postgraduate Scholarship from the University of Waterloo.

\bibliographystyle{amsalpha} 
\bibliography{shifted.bib}

\newcommand{\etalchar}[1]{$^{#1}$}
\providecommand{\bysame}{\leavevmode\hbox to3em{\hrulefill}\thinspace}
\providecommand{\MR}{\relax\ifhmode\unskip\space\fi MR }
\providecommand{\MRhref}[2]{%
  \href{http://www.ams.org/mathscinet-getitem?mr=#1}{#2}
}
\providecommand{\href}[2]{#2}
\begin{thebibliography}{GJK{\etalchar{+}}14}

\bibitem[Ass18]{Assaf}
Sami Assaf, \emph{Shifted dual equivalence and {S}chur {$P$}-positivity}, J.
  Comb. \textbf{9} (2018), no.~2, 279--308.

\bibitem[BR12]{Buch.Ravikumar}
Anders~Skovsted Buch and Vijay Ravikumar, \emph{Pieri rules for the
  {$K$}-theory of cominuscule {G}rassmannians}, J. Reine Angew. Math.
  \textbf{668} (2012), 109--132. \MR{2948873}

\bibitem[BS16]{Buch.Samuel}
Anders~Skovsted Buch and Matthew~J. Samuel, \emph{{$K$}-theory of minuscule
  varieties}, J. Reine Angew. Math. \textbf{719} (2016), 133--171.

\bibitem[Cho13]{Cho}
Soojin Cho, \emph{A new {L}ittlewood-{R}ichardson rule for {S}chur
  {$P$}-functions}, Trans. Amer. Math. Soc. \textbf{365} (2013), no.~2,
  939--972.

\bibitem[CK18]{Choi.Kwon}
Seung-Il Choi and Jae-Hoon Kwon, \emph{Crystals and {S}chur {$P$}-positive
  expansions}, Electron. J. Combin. \textbf{25} (2018), no.~3, Paper No. 3.7,
  27 pages.

\bibitem[CTY14]{Clifford.Thomas.Yong}
Edward Clifford, Hugh Thomas, and Alexander Yong, \emph{{$K$}-theoretic
  {S}chubert calculus for {${\rm OG}(n,2n+1)$} and jeu de taquin for shifted
  increasing tableaux}, J. Reine Angew. Math. \textbf{690} (2014), 51--63.
  \MR{3200334}

\bibitem[EP25a]{EstupinanSalamanca.Pechenik:FPSAC}
Santiago Estupi{\~{n}}\'{a}n{-}Salamanca and Oliver Pechenik, \emph{A new
  shifted {L}ittlewood--{R}ichardson rule and related developments}, Sém.
  Lothar. Combin. \textbf{89B} (2025), to appear, 12 pages.

\bibitem[EP25b]{EstupinanSalamanca.Pechenik}
Santiago Estupi{\~n}\'an{-}Salamanca and Oliver Pechenik, \emph{A universal
  characterization of the shifted plactic monoid}, Semigroup Forum \textbf{110}
  (2025), no.~3, 566--587. \MR{4920045}

\bibitem[Est25]{shifted-code}
Santiago Estupi{\~{n}}\'{a}n{-}Salamanca, \emph{Shifted littlewood--richardson
  rule code}, 2025, GitHub,
  \url{https://github.com/SantiagoSES/Shifted_LR_Rule/blob/main/Shifted_LR_Rule.ipynb}.

\bibitem[Fom95]{Fomin}
Sergey Fomin, \emph{Schur operators and {K}nuth correspondences}, J. Combin.
  Theory Ser. A \textbf{72} (1995), no.~2, 277--292. \MR{1357774}

\bibitem[Ful97]{Fulton:YT}
William Fulton, \emph{Young tableaux. {W}ith applications to representation
  theory and geometry}, London Mathematical Society Student Texts, vol.~35,
  Cambridge University Press, Cambridge, 1997.

\bibitem[GJK{\etalchar{+}}14]{Grancharov.Jung.Kang.Kashiwara.Kim}
Dimitar Grantcharov, Ji~Hye Jung, Seok-Jin Kang, Masaki Kashiwara, and Myungho
  Kim, \emph{Crystal bases for the quantum queer superalgebra and semistandard
  decomposition tableaux}, Trans. Amer. Math. Soc. \textbf{366} (2014), no.~1,
  457--489.

\bibitem[GLP20]{Gillespie.Levinson.Purbhoo}
Maria Gillespie, Jake Levinson, and Kevin Purbhoo, \emph{A crystal-like
  structure on shifted tableaux}, Algebr. Comb. \textbf{3} (2020), no.~3,
  693--725.

\bibitem[Hai89]{Haiman}
Mark~D. Haiman, \emph{On mixed insertion, symmetry, and shifted {Y}oung
  tableaux}, J. Combin. Theory Ser. A \textbf{50} (1989), no.~2, 196--225.

\bibitem[HB86]{Hiller.Boe}
Howard Hiller and Brian Boe, \emph{Pieri formula for {${\rm SO}_{2n+1}/{\rm
  U}_n$} and {${\rm Sp}_n/{\rm U}_n$}}, Adv. in Math. \textbf{62} (1986),
  no.~1, 49--67. \MR{859253}

\bibitem[HCL23]{Huang.Chu.Li}
Fang Huang, Yanjun Chu, and Chuanzhong Li, \emph{Littlewood-{R}ichardson rule
  for generalized {S}chur {Q}-functions}, Algebr. Represent. Theory \textbf{26}
  (2023), no.~6, 3143--3165. \MR{4681346}

\bibitem[Knu70]{Knuth}
Donald~E. Knuth, \emph{Permutations, matrices, and generalized {Y}oung
  tableaux}, Pacific J. Math. \textbf{34} (1970), 709--727. \MR{272654}

\bibitem[Lot02]{Lothaire}
M.~Lothaire, \emph{Algebraic combinatorics on words}, Encyclopedia of
  Mathematics and its Applications, vol.~90, Cambridge University Press,
  Cambridge, 2002, A collective work by Jean Berstel, Dominique Perrin, Patrice
  Seebold, Julien Cassaigne, Aldo De Luca, Steffano Varricchio, Alain Lascoux,
  Bernard Leclerc, Jean-Yves Thibon, Veronique Bruyere, Christiane Frougny,
  Filippo Mignosi, Antonio Restivo, Christophe Reutenauer, Dominique Foata,
  Guo-Niu Han, Jacques Desarmenien, Volker Diekert, Tero Harju, Juhani
  Karhumaki and Wojciech Plandowski, With a preface by Berstel and Perrin.

\bibitem[LS81]{Lascoux.Schutzenberger:plaxique}
Alain Lascoux and Marcel-Paul Sch\"{u}tzenberger, \emph{Le mono\"{\i}de
  plaxique}, Noncommutative structures in algebra and geometric combinatorics
  ({N}aples, 1978), Quad. ``Ricerca Sci.'', vol. 109, CNR, Rome, 1981,
  pp.~129--156.

\bibitem[Ngu22]{Nguyen}
Duc~Khanh Nguyen, \emph{On the shifted {L}ittlewood-{R}ichardson coefficients
  and the {L}ittlewood-{R}ichardson coefficients}, Ann. Comb. \textbf{26}
  (2022), no.~1, 221--260. \MR{4407061}

\bibitem[NO19]{Naruse.Okada}
Hiroshi Naruse and Soichi Okada, \emph{Skew hook formula for {$d$}-complete
  posets via equivariant {$K$}-theory}, Algebr. Comb. \textbf{2} (2019), no.~4,
  541--571. \MR{3997510}

\bibitem[Pra91]{Pragacz}
Piotr Pragacz, \emph{Algebro-geometric applications of {S}chur {$S$}- and
  {$Q$}-polynomials}, Topics in invariant theory ({P}aris, 1989/1990), Lecture
  Notes in Math., vol. 1478, Springer, Berlin, 1991, pp.~130--191.

\bibitem[PY17]{Pechenik.Yong:genomic}
Oliver Pechenik and Alexander Yong, \emph{Genomic tableaux}, J. Algebraic
  Combin. \textbf{45} (2017), no.~3, 649--685. \MR{3627499}

\bibitem[Rob38]{Robinson}
G.~de~B. Robinson, \emph{On the {R}epresentations of the {S}ymmetric {G}roup},
  Amer. J. Math. \textbf{60} (1938), no.~3, 745--760. \MR{1507943}

\bibitem[RYY22]{Robichaux.Yadav.Yong}
Colleen Robichaux, Harshit Yadav, and Alexander Yong, \emph{Equivariant
  cohomology, {S}chubert calculus, and edge labeled tableaux}, Facets of
  algebraic geometry. {V}ol. {II}, London Math. Soc. Lecture Note Ser., vol.
  473, Cambridge Univ. Press, Cambridge, 2022, pp.~284--335. \MR{4381918}

\bibitem[Sag80]{Sagan:random}
Bruce Sagan, \emph{On selecting a random shifted {Y}oung tableau}, J.
  Algorithms \textbf{1} (1980), no.~3, 213--234. \MR{604864}

\bibitem[Sag87]{Sagan}
Bruce~E. Sagan, \emph{Shifted tableaux, {S}chur {$Q$}-functions, and a
  conjecture of {R}. {S}tanley}, J. Combin. Theory Ser. A \textbf{45} (1987),
  no.~1, 62--103. \MR{883894}

\bibitem[Sch61]{Schensted}
C.~Schensted, \emph{Longest increasing and decreasing subsequences}, Canadian
  J. Math. \textbf{13} (1961), 179--191. \MR{121305}

\bibitem[Ser10a]{Serrano}
Luis Serrano, \emph{The shifted plactic monoid}, Math. Z. \textbf{266} (2010),
  no.~2, 363--392.

\bibitem[Ser10b]{SerranoThesis}
Luis~Guillermo Serrano{~}Herrera, \emph{Noncommutative {S}chur {P}-functions
  and the shifted plactic monoid}, ProQuest LLC, Ann Arbor, MI, 2010, Thesis
  (Ph.D.)--University of Michigan.

\bibitem[Shi99]{Shimozono}
Mark Shimozono, \emph{Multiplying {S}chur {$Q$}-functions}, J. Combin. Theory
  Ser. A \textbf{87} (1999), no.~1, 198--232. \MR{1698253}

\bibitem[Shi17]{Shigechi}
Keiichi Shigechi, \emph{Shifted tableaux and products of {S}chur's symmetric
  functions}, preprint (2017), 57 pages, \arxiv{1705.06437}.

\bibitem[Ste89]{Stembridge}
John~R. Stembridge, \emph{Shifted tableaux and the projective representations
  of symmetric groups}, Adv. Math. \textbf{74} (1989), no.~1, 87--134.

\bibitem[Thr52]{Thrall}
R.~M. Thrall, \emph{A combinatorial problem}, Michigan Math. J. \textbf{1}
  (1952), 81--88. \MR{49844}

\bibitem[TY18]{Thomas.Yong:HT}
Hugh Thomas and Alexander Yong, \emph{Equivariant {S}chubert calculus and jeu
  de taquin}, Ann. Inst. Fourier (Grenoble) \textbf{68} (2018), no.~1,
  275--318. \MR{3795480}

\bibitem[Wor84]{Worley}
Dale~Raymond Worley, \emph{A theory of shifted {Y}oung tableaux}, ProQuest LLC,
  Ann Arbor, MI, 1984, Thesis (Ph.D.)--Massachusetts Institute of Technology.

\end{thebibliography}
\end{document}